\documentclass[12pt]{amsart}

\usepackage{amssymb,amsthm,amsmath,url,relsize,etoolbox,cite,nameref,enumerate,fancyhdr}
\usepackage{mathtools,xcolor,enumitem,bm,bbm}
\usepackage[utf8]{inputenc}
\usepackage[hidelinks]{hyperref}
\usepackage{mathrsfs}
\usepackage{mnsymbol}
\usepackage{nicematrix}

\theoremstyle{plain}
\newtheorem*{theorem*}{Theorem}
\newtheorem{theorem}[subsubsection]{Theorem}

\newtheorem{lemma}[subsubsection]{Lemma}

\newtheorem{corollary}[subsubsection]{Corollary}

\newtheorem{definition}[subsubsection]{Definition}

\newtheorem{remark}[subsubsection]{Remark}

\DeclareMathOperator{\degree}{deg}

\DeclareMathOperator{\rank}{rank}
\DeclareMathOperator{\measure}{meas}
\DeclareMathOperator{\kernel}{ker}

\DeclareMathOperator{\HankelChar}{char}

\newcommand*\mathbfe[1]{\boldsymbol{#1}}

\makeatletter
\def\l@subsection{\@tocline{2}{0pt}{2.5pc}{5pc}{}}
\makeatother

\makeatletter
\renewcommand*\env@matrix[1][*\c@MaxMatrixCols c]{%
  \hskip -\arraycolsep
  \let\@ifnextchar\new@ifnextchar
  \array{#1}}
\makeatother

\begin{document}

\title[Variance and Correlations of the Divisor Function, and Hankel Matrices]{The Variance and Correlations of the Divisor Function in $\mathbb{F}_q [T]$, and Hankel Matrices}
\author{Michael Yiasemides}
\date{\today}
\address{Department of Mathematics, University of Exeter, Exeter, EX4 4QF, UK}
\email{my298@exeter.ac.uk}

\maketitle

\begin{abstract}
We prove an exact formula for the variance of the divisor function over short intervals in $\mathcal{A} := \mathbb{F}_q [T]$, where $q$ is a prime power. A slight adaption of the proof allows us to obtain an exact formula for correlations of the form $d(A) d(A+B)$, where we average both $A$ and $B$ over certain intervals in $\mathcal{A}$. We also consider correlations of the form $d(KQ+N) d (N)$, where $Q$ is prime and $K$ and $N$ are averaged over certain intervals. If $\degree K < \degree Q -1$, then these correlations appear in the off-diagonal terms for the fourth moment of Dirichlet $L$-functions. We consider the case $\degree K \geq \degree Q -1$ and obtain an exact formula for the correlations. Further, we demonstrate that $d(KQ+N)$ and $d (N)$ are uncorrelated for the given ranges of $K$ and $N$. Our approach to these problems is to use the orthogonality relations of additive characters on $\mathbb{F}_q$ to translate the problems to ones involving the ranks of Hankel matrices over $\mathbb{F}_q$. Most of the paper is dedicated to proving several results regarding the rank and kernel structure of these matrices, and thus demonstrating their number-theoretic properties. We briefly discuss extending our method to moments higher than the second (the variance) over intervals; to the $k$-th divisor function; and to correlations of the divisor function with applications to moments of Dirichlet $L$-functions in function fields.
\end{abstract}


\allowdisplaybreaks


\section{Introduction and Results}

\subsection{Background}

Classically, for $k \geq 2$, the $k$-th divisor function is defined, for $n \in \mathbb{N}$, by
\begin{align*}
d_k (n)
:= \lvert \{ (a_1 , \ldots , a_k ) \in \mathbb{N}^k : a_1 \ldots a_k = n \} \rvert ,
\end{align*}
where $\mathbb{N}$ is the set of positive integers; and when $k=2$ we will often write $d$ instead of $d_2$. \\

It was shown by Dirichlet that
\begin{align} \label{statement, background, Dirichlet's average for d(n) for n leq x}
\sum_{n \leq x} d(n)
= x \log x + (2 \gamma -1) x + \Delta (x) ,
\end{align}
where the remainder satisfies $\Delta (x) = O (x^{\frac{1}{2}})$; while, for $k \geq 2$, it can be shown that 
\begin{align*}
\sum_{n \leq x} d_k (n)
= x P_k (\log x ) + \Delta_k (x) ,
\end{align*}
where $P_k$ is a polynomial of degree $k-1$ and $\Delta_k (x)$ is a lower order term. It is of particular interest to understand the behaviour of the remainder $\Delta_k (x)$, and we do so by studying its moments. Various results on this and the above are given in Chapter 12 of \cite{TheoRiemannZetaFunc_Titchmarsh_Heathbrown1987}. \\

We mention here that Cram\'{e}r \cite{Cramer1922_UberZweiSatzeHardy} proved that
\begin{align*}
X^{-\frac{3}{2}} \int_{x=0}^{X} \Delta (x)^2 \mathrm{d} x
\sim \frac{1}{6 \pi^2} \sum_{n=1}^{\infty} d(n)^2 n^{-\frac{3}{2}} ,
\end{align*}
while Tong \cite{Tong1956_DivProbIII} proved that
\begin{align*}
X^{-\frac{5}{3}} \int_{x=0}^{X} \Delta_3 (x)^2 \mathrm{d} x
\sim \frac{1}{10 \pi^2} \sum_{n=1}^{\infty} d_3 (n)^2 n^{-\frac{4}{3}} .
\end{align*}

We can also consider higher moments of $\Delta$: $\int_{x=0}^{X} \Delta (x)^k \mathrm{d} x$. For this, we refer the reader to the works of Heath-Brown \cite{Heath-Brown1992_DistrMomErrTermDirichletDivProb} and Tsang \cite{Tsang1992_HighPowMomDeltaEP}. We also mention that Heath-Brown proved that $x^{-\frac{1}{4}} \Delta (x)$ has distribution function. That is, there is a function $f$ such that
\begin{align*}
X^{-1} \measure \{ x \in [1,X]: x^{-\frac{1}{4}} \Delta (x) \in I \}
\longrightarrow \int_{t \in I} f(t) \mathrm{d} t 
\end{align*}
as $X \longrightarrow \infty$. The function $f$ extends to an entire function on $\mathbb{C}$ and satisfies certain bounds on its derivatives. \\

A related topic is that of the divisor function over intervals. That is, we are interested in
\begin{align*}
\sum_{ x < n \leq x+H} d(n)
= \sum_{n \leq x+H} d(n) - \sum_{n \leq x} d(n) .
\end{align*}
Applying (\ref{statement, background, Dirichlet's average for d(n) for n leq x}), we obtain
\begin{align*}
\sum_{ x < n \leq x+H} d(n)
= &x \log \Big( 1 + \frac{H}{x} \Big) + H \log (x+H) + (2 \gamma -1) H \\
	&+ \Delta (x+H) - \Delta (x) .
\end{align*}
Given that $\Delta (x) = O (x^{\frac{1}{2}})$, it is clear that, for $H \leq x$, the error term 
\begin{align*}
\Delta (x;H) := \Delta (x+H) - \Delta (x)
\end{align*}
is of lower order. Nonetheless, it is not fully understood. \\

For the $k$-th divisor problem, the analogous object to study is
\begin{align} \label{def statement, Delta_k (x;H) NF def}
\Delta_k (x;H) := \Delta_k (x+H) - \Delta_k (x) .
\end{align}
It is the short intervals that have $H \leq x^{1-\frac{1}{k}}$ that are of particular interest. We highlight some of the main results in this area. Let $\epsilon > 0$, and consider the range $X^{\epsilon} < H < X^{\frac{1}{2} - \epsilon}$. Ivi\'{c} \cite{Ivic2009_DivFuncRZFShortInterv} (see also \cite{Jutila1984_DivFuncShortInterv} and \cite{CoppolaSalerno2004_SymmDivFuncAlmostShortInt}) proved the asymptotic formula 
\begin{align*}
\frac{1}{X} \int_{x = X}^{2X} \Delta (x;H)^2 \mathrm{d} x
= H \sum_{j=0}^{3} c_j \log^j \Big( \frac{X^{\frac{1}{2}}}{H} \Big) 
	\;\; + O_{\epsilon} \big( X^{- \frac{1}{2} + \epsilon} H^2 \big) 
	+ O_{\epsilon} \big( X^{\epsilon} H^{\frac{1}{2}} \big) ,
\end{align*}
where $c_0, c_1 , c_2$ are constants and $c_3 = \frac{8}{\pi^2}$. Assuming the Riemann hypothesis, for $k \geq 3$ and the range $X^{\epsilon} < H < X^{1-\epsilon}$, Milinovich and Turnage-Butterbaugh \cite{MilinovichTurnage-Butterbaugh2014_MomProdAutomorphLFunc} obtained the upper bound
\begin{align*}
\frac{1}{X} \int_{x = X}^{2X} \Delta_k (x;H)^2 \mathrm{d} x
\ll H (\log X)^{k^2 + o (1)}.
\end{align*}
Asymptotic formulas can be obtained given certain restrictions on $H$. For $k \geq 3$ (and assuming the Lindel\"{o}f Hypothesis for $k >3$) and $2 \leq L \leq X^{\frac{1}{k(k-1)} - \epsilon}$, Lester \cite{Lester2016_VarSumDivFuncShortInt} proved
\begin{align*}
\frac{1}{X} \int_{x = X}^{2X} \Delta_k \Big( x; \frac{x^{1-\frac{1}{k}}}{L} \Big)^2 \mathrm{d} x
=C_k \frac{X^{1-\frac{1}{k}}}{L} (\log L)^{k^2 -1}
	+ O \Big( \frac{X^{1-\frac{1}{k}}}{L} (\log L)^{k^2 -2} \Big) .
\end{align*}
Finally, as $L,X \longrightarrow \infty$ with $\log L = o (\log T)$, and $\alpha < \beta$, Lester and Yesha \cite{LesterYesha2016_DistrDivFuncHeckeEigenV} prove that
\begin{align*}
\frac{1}{X} \measure \bigg\{ x \in [X , 2X] : \alpha \leq \frac{ \Delta \Big( x ; \frac{x^{\frac{1}{2}}}{L} \Big)}{x^{\frac{1}{4}} \sqrt{\frac{8}{\pi^2} \frac{\log^3 L}{L} }} \leq \beta \bigg\}
\sim \frac{1}{\sqrt{2 \pi}} \int_{t=\alpha}^{\beta} e^{-\frac{t^2}{2}} \mathrm{d} t .
\end{align*}
That is, we have a Gaussian distribution function. \\

Let us now consider the divisor function over short intervals in the polynomial ring $\mathcal{A} := \mathbb{F}_q [T]$, where $q$ is a prime power. Before proceeding, we define $\mathcal{M}$ to be the set of monic polynomials in $\mathcal{A}$; and for $\mathcal{B} = \mathcal{A} , \mathcal{M}$ we define $\mathcal{B}_n$ and $\mathcal{B}_{\leq n}$ to be the set of polynomials in $\mathcal{B}$ with degree equal to $n$ and degree $\leq n$, respectively. It should be noted that, as $\mathcal{A}$ is a Euclidean domain, primality and irreducibility are equivalent. For non-zero $A \in \mathcal{A}$ we define $\lvert A \rvert := q^{\degree A}$, and we define $\lvert 0 \rvert := 0$. It is convenient to take $\degree 0 = - \infty$, and so (unless otherwise indicated) the range $\degree A \leq n$ should be taken to include the zero polynomial. The $k$-th divisor function is defined for $N \in \mathcal{M}$ by 
\begin{align*}
d_k (N)
:= \lvert \{ (A_1 , \ldots , A_k ) \in \mathcal{M}^k : A_1 \ldots A_k = N \} \rvert .
\end{align*}
For $A \in \mathcal{M}_n$ and $0 \leq h \leq n$, we define the interval 
\begin{align} \label{def state, I(A;h) def}
I (A;h)
:= \{ B \in \mathcal{M} : \degree (B-A) < h \} 
\end{align}
and define 
\begin{align} \label{def state, FF def of N_(d_k)}
\mathcal{N}_{d_k} (A;h)
:= \sum_{B \in I(A;h)} d_k (B) .
\end{align}
This notation is in keeping with \cite{KeatingRodgersRudnik2018_SumDivFuncFqtMatrInt, Gorodetsky2021_PhDThesis}. Although, in \cite{KeatingRodgersRudnik2018_SumDivFuncFqtMatrInt, Gorodetsky2021_PhDThesis}, they take $\leq h$ in (\ref{def state, I(A;h) def}), instead of $<h$; and it should be noted that when we reference their results below it will be in terms of our notation and so it will appear slightly different. We feel that our definition is more natural. For example, it gives $\lvert I (A;h) \rvert = q^h$ as opposed to $\lvert I (A;h) \rvert = q^{h+1}$; and, for $A \in \mathcal{M}_n$, it gives
\begin{align*}
\lvert \{ A' \in \mathcal{M}_n : I (A';h) = I (A;h) \} \rvert = q^{n-h}
\end{align*}
as opposed to
\begin{align*}
\lvert \{ A' \in \mathcal{M}_n : I (A';h) = I (A;h) \} \rvert = q^{n-h-1}.
\end{align*}
Continuing, it is not difficult to obtain an exact expression for the mean value of $\mathcal{N}_{d_k} (A;h)$ (see \cite{AndradeBary-SorokerRudnik2015_ShiftConvolTitchmarshDivProbFqT} for a proof):
\begin{align*}
\frac{1}{q^n} \sum_{A \in \mathcal{M}_n} \mathcal{N}_{d_k} (A;h)
= q^{h} \binom{n+k-1}{k-1} .
\end{align*}
We can now define
\begin{align} \label{def state, FF def of Delta_k (A;h)}
\Delta_k (A;h)
:= \mathcal{N}_{d_k} (A;h) - q^{h} \binom{n+k-1}{k-1} .
\end{align}
It was shown by Keating \textit{et al.} \cite{KeatingRodgersRudnik2018_SumDivFuncFqtMatrInt} that, as $q \longrightarrow \infty$,
\begin{align*}
&\frac{1}{q^n} \sum_{A \in \mathcal{M}_n} \lvert \Delta_k (A;h) \rvert^2 \\
= &\begin{cases}
0 &\text{ for $\Big\lfloor \Big( 1 - \frac{1}{k} \Big) n \Big\rfloor +1 \leq h \leq n+1$,} \\
O \Big( \frac{q^h}{\sqrt{q}} \Big) &\text{ for $h = \Big\lfloor \Big( 1 - \frac{1}{k} \Big) n \Big\rfloor$,} \\
q^h I_k (n ; n-h-3) + O \Big( \frac{H}{\sqrt{q}} \Big) &\text{ for $1 \leq h \leq \min \Big\{ n-4 , \Big\lfloor\Big( 1 - \frac{1}{k} \Big) n \Big\rfloor -1 \Big\}$;} 
\end{cases}
\end{align*}
where $I_k (n ; n-h-3)$ is an integral over a group of unitary matrices, defined by (1.27) in \cite{KeatingRodgersRudnik2018_SumDivFuncFqtMatrInt}. In particular, when $k=2$, and $n \geq 5$ and $h \leq \frac{n}{2} - 1$, we have
\begin{align} \label{statement, intro section, background subsection, keating et al result on variance, q prime power}
\frac{1}{q^n} \sum_{A \in \mathcal{M}_n} \lvert \Delta_2 (A;h) \rvert^2
\sim q^h \frac{(n-2h-1)(n-2h)(n-2h+1)}{6} 
\end{align}
as $q \longrightarrow \infty$. (This result follows from equation (1.34) in \cite{KeatingRodgersRudnik2018_SumDivFuncFqtMatrInt} with k=2. It is also given explicitly in (1.33), although there is a slight error in the evaluation of the binomial there). Recently, in his thesis \cite[Subsection 3.2.1]{Gorodetsky2021_PhDThesis}, Gorodetsky obtained an exact formula for the case $k=2$:
\begin{align} 
\begin{split} \label{statement, reference of Gorodetsky result on div var}
\frac{1}{q^n} \sum_{A \in \mathcal{M}_n} \lvert \Delta (A;h) \rvert^2
= \begin{cases}
(q-1) q^{h-1} \frac{(n-2h-1)(n-2h)(n-2h+1)}{6} &\text{ for $h \leq \lfloor \frac{n}{2} \rfloor -1$,} \\
0 &\text{ for $h \geq \lfloor \frac{n}{2} \rfloor$.}
\end{cases} \\
\end{split}
\end{align}

Let us now turn our attention to divisor correlations, and consider the classical case first. The most common example is
\begin{align*}
\sum_{n \leq x} d(n) d(n+h) ,
\end{align*}
where $h$ is a fixed positive integer, and we are interested in the limit as $x \longrightarrow \infty$. Ingham \cite{Ingham1927_AsymptFormuTheoNumb} showed that
\begin{align*}
\sum_{n \leq x} d(n) d(n+h)
\sim \frac{1}{\zeta (2) } \sigma_{-1} (h) x (\log x)^2 ,
\end{align*}
where $\sigma_{t} (h) := \sum_{a \mid k} a^t$. Estermann \cite{Estermann1931_UberDarstellungenZahlDiffZweiProd} later proved that there exist constants $a_{1,h}$ and $a_{0,h}$ (dependent on $h$) such that for all $\epsilon > 0$ we have
\begin{align*}
\sum_{n \leq x} d(n) d(n+h)
= \frac{1}{\zeta (2) } \sigma_{-1} (h) x (\log x)^2 
	+ a_{1,h} x \log x 
	+ a_{0,h} x
	+ O_{\epsilon} \big( x^{\frac{11}{12}} (\log x)^{\frac{17}{6} + \epsilon} \big) .
\end{align*}
Heath-Brown \cite{FourthPowerMomRZF_HeathBrown1979} subsequently showed that, given $h \leq x^{\frac{5}{6}}$ (and uniformly over this range), we can improve the error term above to $O_{\epsilon} (x^{\frac{5}{6} + \epsilon})$. The importance of these results lies in their application to the fourth moment of the Riemann zeta-function on the critical line \cite{FourthPowerMomRZF_HeathBrown1979}. \\

The analogous problem for higher divisor functions, namely
\begin{align*}
\sum_{n \leq x} d_k (n) d_k (n+h)
\end{align*}
for $k \geq 3$, is also of great importance, specifically in the application to higher moments of the Riemann zeta-function (see the work of Ivi\'{c} \cite{Ivic1997_TernaryAdditDivProbSixMomZetaFunc}, Conrey and Gonek \cite{ConreyGonek1999_HighMomRZF}, and the five papers by Conrey and Keating \cite{ConreyKeating2015_MomZetaCorrelDivSumI, ConreyKeating2015_MomZetaCorrelDivSumII, ConreyKeating2015_MomZetaCorrelDivSumIII, ConreyKeating2015_MomZetaCorrelDivSumIV, ConreyKeating2015_MomZetaCorrelDivSumV}). It is conjectured (see equation (1.6) of \cite{Ivic1997_TernaryAdditDivProbSixMomZetaFunc}) that
\begin{align*}
\sum_{n \leq x} d_k (n) d_k (n+h)
= x P_{2k-2} (\log x ;h) + \Delta (x;h) ,
\end{align*}
where $P_{2k-2} (\log x ;h)$ is a polynomial in $\log x$ of degree $2k-2$ with coefficients dependent on $h$, and we expect the error term to satisfy $\Delta (x;h) = o (x)$ as $x \longrightarrow \infty$. \\

Another example of divisor correlations is
\begin{align} \label{statement, NF Dir L-func correlations}
\sum_l \sum_n d(lk+n) d(n) ,
\end{align}
where $k$ is fixed, $l$ ranges over a certain interval, and $n$ ranges over integers of a certain interval that also satisfy $(n,k)=1$. This appears in the off-diagonal terms for the fourth moment of Dirichlet $L$-functions as can be seen in \cite{HeathBrown1981_FourthPowerMeanDLF, Sound2007_FourthMomDLF} (with the function field analogue appearing in \cite{AndradeYiasemides2021_4thPowMeanDLFuncField_Final}). \\

In the function field setting, Andrade, Bary-Soroker, and Rudnick \cite{AndradeBary-SorokerRudnik2015_ShiftConvolTitchmarshDivProbFqT} proved
\begin{align*}
\frac{1}{q^n} \sum_{A \in \mathcal{M}_n} d_k (A) d_k (A+B)
= \binom{n+k-1}{k-1}^2 + O \big( q^{-\frac{1}{2}} \big) ,
\end{align*}
uniformly over all $B \in \mathcal{A} \backslash \{ 0 \}$ with $\degree B \leq n$, and as $q \longrightarrow \infty$ (recall $q$ is the order of the finite field $\mathbb{F}_q$ and $\mathcal{A} := \mathbb{F}_q [T]$). Gorodetsky \cite[Lemma 3.3]{Gorodetsky2021_PhDThesis} improves on the case $k=2$ by showing that
\begin{align} \label{statement, reference of Gorodetsky result on corr, fixed B}
\frac{1}{q^n} \sum_{A \in \mathcal{M}_n} d (A) d (A+B)
= (n+1)^2 + \sum_{i=1}^{\lfloor \frac{n}{2} \rfloor} \frac{(n-2i+1)^2}{q^i} \Big( d (B;i) - d (B;i-1) \Big) ,
\end{align}
 where $d(B;i)$ is the number of monic divisors of $B$ of degree $i$. \\

\subsection{Statement of Results}

We obtain an exact formula for the variance of the divisor function over intervals in $\mathbb{F}_q [T]$, where $q$ is a prime power.

\begin{theorem} \label{main theorem, variance divisor function intervals}
For $n \geq 4$ we have
\begin{align}
\frac{1}{q^n} \sum_{A \in \mathcal{M}_n} \lvert \Delta (A;h) \rvert^2
= \begin{cases}
(q-1) q^{h-1} \frac{(n-2h-1)(n-2h)(n-2h+1)}{6} &\text{ for $h \leq \lfloor \frac{n}{2} \rfloor -1$,} \\
0 &\text{ for $h \geq \lfloor \frac{n}{2} \rfloor$.}
\end{cases}
\end{align}
\end{theorem}
To prove Theorem \ref{main theorem, variance divisor function intervals}, we use the orthogonality relation of a non-trivial additive character on $\mathbb{F}_q$ to express the problem in terms of Hankel matrices over $\mathbb{F}_q$. Most of this article comprises of results on Hankel Matrices over finite fields, which can be found in Section \ref{Section, Hankel matrices}. Once these results are established, the proof of Theorem \ref{main theorem, variance divisor function intervals}, in Section \ref{section, variance of divisor function}, is relatively short. \\

In Section \ref{section, correlations}, a very slight adaptation of our proof of Theorem \ref{main theorem, variance divisor function intervals} allows us to prove the following result on divisor correlations:

\begin{theorem} \label{theorem, correlations d(A) d(A+B)}
For $n \geq 4$ and $h \leq n$ we have
\begin{align*}
&\frac{1}{q^{h+n}} \sum_{A \in \mathcal{M}_n } \sum_{B \in \mathcal{A}_{<h} } d (A) d (A+B) \\
= &(n+1)^2 + \frac{1}{q^{2h+n}} \sum_{A \in \mathcal{M}_n} \lvert \Delta (A;h) \rvert^2 \\
= &\begin{cases}
(n+1)^2 + (1 - q^{-1}) (q^{-h}) \frac{(n-2h-1) (n-2h) (n-2h+1)}{6} &\text{ for $h \leq \lfloor \frac{n}{2} \rfloor - 1$} , \\
(n+1)^2 &\text{ for $h \geq \lfloor \frac{n}{2} \rfloor $} .
\end{cases}
\end{align*}
\end{theorem}
This result, and its proof, allow us to clearly see the relationship between divisor variance and correlations. \\

Note that Theorem \ref{main theorem, variance divisor function intervals} is identical to Gorodetsky's result (\ref{statement, reference of Gorodetsky result on div var}), and Theorem \ref{theorem, correlations d(A) d(A+B)} can be deduce from his other result (\ref{statement, reference of Gorodetsky result on corr, fixed B}). However, we employ a completely different approach, which we also adapt to prove the following result in Section \ref{section, correlations}: 

\begin{theorem} \label{theorem, divisor correlations d(KQ+N)d(N)}
Let $Q \in \mathcal{M}$ be prime, and let $n,k$ be such that $0 \leq n \leq \degree Q - 1 \leq k$. Then,
\begin{align*}
&\frac{1}{q^{k+n}} \sum_{K \in \mathcal{M}_k } \sum_{N \in \mathcal{M}_n } d (KQ+N) d (N) \\
= &\bigg( \frac{1}{q^n } \sum_{N \in \mathcal{M}_n} d(N) \bigg)
	\bigg( \frac{1}{q^{k + n}} \sum_{K \in \mathcal{M}_k } \sum_{N \in \mathcal{M}_n } d (KQ+N) \bigg) \\
= &(\degree Q + k +1) (n+1) - q^{-\degree Q} (k - \degree Q -1) (n+1) .
\end{align*}
\end{theorem}
This is the function field analogue of (\ref{statement, NF Dir L-func correlations}), although we are considering the special case where $Q$ is prime. However, if we wish to apply this to the fourth moment of Dirichlet $L$-functions in function fields, then we would require the restriction $k < \degree Q -1$ instead of $k \geq \degree Q -1$, which is more difficult. We discuss this further in Subsection \ref{subsection, intro, motivation} and Remark \ref{remark, extending FF Dir L func Corr to k < deg Q -1}. Nonetheless, Theorem \ref{theorem, divisor correlations d(KQ+N)d(N)} is an interesting result, not only because it is exact, but also because it shows that $d (KQ+N)$ and $d (N)$ are uncorrelated, assuming $K$ and $N$ are chosen randomly in the ranges given in the theorem. \\

Now that we have given our number-theoretic results, let us discuss our results on Hankel matrices and our use of additive characters. \\

An additive character $\psi$ on $\mathbb{F}_q$ is a function from $\mathbb{F}_q$ to $\mathbb{C}^*$ satisfying $\psi (a+b) = \psi (a) \psi (b)$ (note this implies $\psi (0) =1$ and $\psi (-a) = \psi (a)^{-1}$ for all $a \in \mathbb{F}_q$). We say $\psi$ is non-trivial if there exists $a \in \mathbb{F}_q^*$ such that $\psi (a) \neq 1$, and in this case we have the orthogonality relation
\begin{align} \label{statement, intro, exp identity sum}
\frac{1}{q} \sum_{\alpha \in \mathbb{F}_q } \psi (\alpha b)
= \begin{cases}
1 &\text{if $b=0$,} \\
0 &\text{if $b \in \mathbb{F}_q^*$.}
\end{cases}
\end{align}
The first case follows from the fact that if $b=0$, then $\psi (\alpha b) = 1$ for all $\alpha \in \mathbb{F}_q$. The second case follows from the fact that if $b \in \mathbb{F}_q^*$, then $\alpha b$ and $\alpha b + a$ both vary over $\mathbb{F}_q$ as $\alpha$ varies over $\mathbb{F}_q$, and so
\begin{align*}
\sum_{\alpha \in \mathbb{F}_q } \psi (\alpha b)
= \sum_{\alpha \in \mathbb{F}_q } \psi (\alpha b + a)
= \psi (a) \sum_{\alpha \in \mathbb{F}_q } \psi (\alpha b).
\end{align*}
Since $\psi (a) \neq 1$, we deduce that $\sum_{\alpha \in \mathbb{F}_q } \psi (\alpha b) =0$. In the remainder of this article, $\psi$ is a non-trivial character on $\mathbb{F}_q$, and we will make significant use of (\ref{statement, intro, exp identity sum}). \\

An $l \times m$ Hankel matrix over $\mathbb{F}_q$ is a matrix of the form
\begin{align*}
(\alpha_{i+j-2})_{\substack{1 \leq i \leq l \\ 1 \leq j \leq m}}
= \begin{pNiceMatrix}
\alpha_0 & \alpha_1 & \alpha_2 & \Cdots & \Cdots & \Cdots & \alpha_{m-1} \\
\alpha_1 & \alpha_2 &  &  &  &  & \Vdots \\
\alpha_2 &  &  &  &  &  & \Vdots \\
\Vdots &  &  &  &  &  & \Vdots \\
\Vdots &  &  &  &  &  & \alpha_{l+m-4} \\
\Vdots &  &  &  &  & \alpha_{l+m-4} & \alpha_{l+m-3} \\
\alpha_{l-1}  & \Cdots & \Cdots & \Cdots & \alpha_{l+m-4} & \alpha_{l+m-3} & \alpha_{l+m-2}
\end{pNiceMatrix} ,
\end{align*}
where $\alpha_0 , \ldots , \alpha_{l+m-2} \in \mathbb{F}_q$. It is natural to consider the sequence $\mathbfe{\alpha} = (\alpha_0 , \alpha_1 , \ldots , \alpha_{l+m-2}) \in \mathbb{F}_q^{l+m-1}$ that is associated to the matrix above, and it will be convenient to denote the matrix by $H_{l,m} (\mathbfe{\alpha})$. \\

Theorem \ref{theorem, matrices of given rank and size} gives the number of Hankel matrices of a given size and rank. Theorems \ref{theorem, all Hankel matrices have characteristic polynomial kernel} and \ref{theorem, all corpime characteristic polynomial have a Hankel matrix} demonstrate the kernel structure of Hankel matrices, and we see how function field arithmetic is incorporated in these matrices. To see this, we must view the coefficients of a polynomial as the entries in a vector, and vice versa. For example, the polynomial $a_0 + a_1 T + \ldots + a_n T^n$ should be associated to the vector $(a_0 , a_1 , \ldots , a_n )^T$. We prove that, generally, for a Hankel matrix $H$, there are polynomials $A_1 , A_2$ such that the kernel of $H$ consists exactly of the polynomials
\begin{align*}
B_1 A_1 + B_2 A_2
\end{align*}
where $B_1 , B_2 \in \mathcal{A}$ are any polynomials satisfying a certain bound on their degrees. The polynomials $A_1 , A_2$ are called the characteristic polynomials of $H$, not to be confused with the characteristic polynomial of a square matrix. In Theorem \ref{Theorem, Hankel matrices incorporate Euclidean algorithm} and Corollary \ref{corollary, Hankel matrices incorporate Euclidean algorithm theorem, full algorithm presented}, we show that if $H'$ is a top-left submatrix \footnote{A top-left submatrix of an $l \times m$ matrix $M$ is a submatrix consisting of the first $l_1$ rows and first $m_1$ columns of $M$, for some $l_1 \leq l$ and $m_1 \leq m$.} of $H$, then the characteristic polynomials associated to $H'$ are the same as the polynomials that we obtain after applying a certain number of steps of the Euclidean algorithm to $A_1 , A_2$; the exact number of steps is related to the size of the submatrix $H'$. Theorem \ref{theorem, kernel structure subsection, char polys of extended sequences} demonstrates how the kernel structure of a Hankel matrix changes if we extend it by a single row or column. We prove various other results as well. \\

\subsection{Motivation} \label{subsection, intro, motivation}

We will briefly give an explanation of how additive characters can allow us to express sums of the divisor function in $\mathbb{F}_q [T]$ in terms of Hankel matrices. This is best achieved by considering the sum 
\begin{align*}
\sum_{A \in \mathcal{M}_n} d(A)^2 .
\end{align*}
We have that
\begin{align*}
\sum_{A \in \mathcal{M}_n} d(A)^2 
= \sum_{A \in \mathcal{M}_n} 
	\bigg( \sum_{\substack{l,m \geq 0 \\ l+m=n }} 
	\sum_{\substack{ E \in \mathcal{M}_l \\ F \in \mathcal{M}_m \\ EF = A}} 1 \bigg)^2
= \sum_{A \in \mathcal{A}_{\leq n} } 
	\bigg( \sum_{\substack{l,m \geq 0 \\ l+m=n }} 
	\sum_{\substack{ E \in \mathcal{M}_l \\ F \in \mathcal{M}_m \\ EF = A}} 1 \bigg)^2 .
\end{align*}
We note that the conditions on $E$ and $F$ force $A$ to be monic and of degree $n$, and that is why for the last equality we were able to replace the condition $A \in \mathcal{M}_n$ with $A \in \mathcal{A}_{\leq n}$. Now, let us write $a_i , e_i , f_i$ for the $i$-th coefficient of $A,E,F$, respectively. We also write $\{ EF \}_i$ for the $i$-th coefficient of $EF$ when we do not yet wish to express the coefficients of $EF$ in terms of the $e_i$ and $f_i$. We have
\begin{align*}
\sum_{A \in \mathcal{M}_n} d(A)^2 
= &\sum_{A \in \mathcal{A}_{\leq n} } 
	\bigg( \sum_{\substack{l,m \geq 0 \\ l+m=n }} 
	\sum_{\substack{ E \in \mathcal{M}_l \\ F \in \mathcal{M}_m }}
	\prod_{k=0}^{n} \mathbbm{1}_{\{ EF \}_k = a_k } \bigg)^2 \\
= &\frac{1}{q^{2n+2}} \sum_{A \in \mathcal{A}_{\leq n} } 
	\bigg( \sum_{\substack{l,m \geq 0 \\ l+m=n }} 
	\sum_{\substack{ E \in \mathcal{M}_l \\ F \in \mathcal{M}_m }}
	\prod_{k=0}^{n} \sum_{\alpha_k \in \mathbb{F}_q} 
		\psi \Big( \alpha_k 
		\big( \sum_{\substack{i,j \geq 0 \\ i+j=k }} e_i f_j \hspace{1em} - a_k \big) \Big) \bigg)^2 .
\end{align*}
Here, for a proposition $\mathbf{P}$, we define $\mathbbm{1}_{\mathbf{P}}$ to be $1$ if the proposition is true, and $0$ if false. For the last equality, we used (\ref{statement, intro, exp identity sum}) with $b = \{ EF \}_k - a_k$. We will now collect all of the terms involving $\psi$. To do this, we note that
\begin{align*}
\sum_{k=0}^{n} \alpha_k \sum_{\substack{i,j \geq 0 \\ i+j=k }} e_i f_j 
= &(e_0 , e_1 , \ldots , e_l )
	\begin{pNiceMatrix}
	\alpha_0 & \alpha_1 & \alpha_2 & \Cdots & \Cdots & \Cdots & \alpha_{m} \\
	\alpha_1 & \alpha_2 &  &  &  &  & \Vdots \\
	\alpha_2 &  &  &  &  &  & \Vdots \\
	\Vdots &  &  &  &  &  & \Vdots \\
	\Vdots &  &  &  & &  & \alpha_{n-2} \\
	\Vdots &  &  &  &  & \alpha_{n-2} & \alpha_{n-1} \\
	\alpha_{l} & \Cdots & \Cdots & \Cdots & \alpha_{n-2} & \alpha_{n-1} & \alpha_{n}
	\end{pNiceMatrix} 
	\begin{pNiceMatrix} f_0 \\ f_1 \\ \Vdots \\ \\ \\ \\ f_m \end{pNiceMatrix} \\
= &\mathbf{e}^T H_{l+1 , m+1} (\mathbfe{\alpha}) \mathbf{f} ,
\end{align*}
where
\begin{align*}
\mathbf{e} := \begin{pNiceMatrix} e_0 \\ e_1 \\ \Vdots \\ e_l \end{pNiceMatrix} , \hspace{2em}
\mathbf{f} := \begin{pNiceMatrix} f_0 \\ f_1 \\ \Vdots \\ f_m \end{pNiceMatrix} , \hspace{2em}
\mathbfe{\alpha} := ( \alpha_0 , \alpha_1 , \ldots , \alpha_n ) .
\end{align*}
Thus, we have
\begin{align*}
\sum_{A \in \mathcal{M}_n} d(A)^2 
= &\frac{1}{q^{2n+2}} \sum_{A \in \mathcal{A}_{\leq n} } 
	\bigg( \sum_{\substack{l,m \geq 0 \\ l+m=n }} 
		\sum_{\substack{ \mathbf{e} \in \mathbb{F}_q^l \times \{ 1 \} \\ \mathbf{f} \in \mathbb{F}_q^m \times \{ 1 \} }}
		\sum_{\mathbfe{\alpha} \in \mathbb{F}_q^{n+1} } 
		\psi \Big( \mathbf{e}^T H_{l+1 , m+1} (\mathbfe{\alpha}) \mathbf{f} -  \sum_{k=0}^{n} \alpha_k a_k \Big) \bigg)^2 \\
= &\frac{1}{q^{2n+2}} \sum_{A \in \mathcal{A}_{\leq n} } 
	\bigg( \sum_{\substack{l,m \geq 0 \\ l+m=n }} 
		\sum_{\substack{ \mathbf{e} \in \mathbb{F}_q^l \times \{ 1 \} \\ \mathbf{f} \in \mathbb{F}_q^m \times \{ 1 \} }}
		\sum_{\mathbfe{\alpha} \in \mathbb{F}_q^{n+1} } 
		\psi \Big( \mathbf{e}^T H_{l+1 , m+1} (\mathbfe{\alpha}) \mathbf{f} -  \sum_{k=0}^{n} \alpha_k a_k \Big) \bigg) \\
	&\hspace{5em} \times \bigg( \sum_{\substack{l' , m' \geq 0 \\ l' + m' =n }} 
		\sum_{\substack{ \mathbf{g} \in \mathbb{F}_q^{l'} \times \{ 1 \} \\ \mathbf{h} \in \mathbb{F}_q^{m'} \times \{ 1 \} }}
		\sum_{\mathbfe{\beta} \in \mathbb{F}_q^{n+1} } 
		\psi \Big( \mathbf{g}^T H_{l' +1 , m' +1} (\mathbfe{\beta}) \mathbf{h} -  \sum_{k=0}^{n} \beta_k a_k \Big) \bigg) .
\end{align*}
Let us now consider only the terms involving a given $a_k$ in the sum above. Since $A \in \mathcal{A}_{\leq n}$, we can see that $a_k$ ranges over $\mathbb{F}_q$. We have
\begin{align*}
\frac{1}{q} \sum_{a_k \in \mathbb{F}_q } \psi \Big( - (\alpha_k + \beta_k) a_k \Big)
= \begin{cases}
1 &\text{ if $\alpha_k = -\beta_k$,} \\
0 &\text{ if $\alpha_k \neq - \beta_k$,}
\end{cases}
\end{align*}
where we have used (\ref{statement, intro, exp identity sum}) with $b = \alpha_k + \beta_k$. Essentially, by considering all $k= 0 , 1 , \ldots , n$ this means $\mathbfe{\alpha} = - \mathbfe{\beta}$, and we have effectively removed the sum over $A$. For simplicity, we will take $\mathbfe{\alpha} = \mathbfe{\beta}$. So, we have
\begin{align*}
\sum_{A \in \mathcal{M}_n} d(A)^2 
= &\frac{1}{q^{n+1}} \sum_{\mathbfe{\alpha} \in \mathbb{F}_q^{n+1} } 
	\bigg( \sum_{\substack{l,m \geq 0 \\ l+m=n }} 
		\sum_{\substack{ \mathbf{e} \in \mathbb{F}_q^l \times \{ 1 \} \\ \mathbf{f} \in \mathbb{F}_q^m \times \{ 1 \} }}
		\psi \Big( \mathbf{e}^T H_{l+1 , m+1} (\mathbfe{\alpha}) \mathbf{f} \Big) \bigg)^2 .
\end{align*}
Now, let us consider the sum over a given $e_i$:
\begin{align*}
\frac{1}{q} \sum_{e_i \in \mathbb{F}_q } \psi \big( e_i R_{i+1} \mathbf{f} \big) 
\end{align*}
where $R_{i+1}$ is the $(i+1)$-th row of $H_{l+1 , m+1} (\mathbfe{\alpha})$ (the row indexing begins at $1$, not $0$, and that is why we have $R_{i+1}$ and not $R_i$). By using (\ref{statement, intro, exp identity sum}) again, with $b = R_{i+1} \mathbf{f}$, we can see that the sum over $e_i$ will give a non-zero contribution only when $R_{i+1} \mathbf{f} = 0$, and this non-zero contribution will be $1$. Applying this to all $i = 0 , 1 , \ldots , l$, we see that a non-zero contribution occurs only when
\begin{align} \label{statement, intro, motivation, H_(l+1,m+1) (alpha) f = 0 requirement}
H_{l+1 , m+1} (\mathbfe{\alpha}) \mathbf{f} = \mathbf{0}.
\end{align}
That is, when $\mathbf{f}$ is in the kernel of $H_{l+1 , m+1} (\mathbfe{\alpha})$. \\

Technically, this is not quite true as the last entry of $\mathbf{e}$ is $1$, and so it cannot take any value in $\mathbb{F}_q$. Ultimately, this actually limits the $\mathbfe{\alpha}$ that we can take, and so simplifies our final calculations. However, for now, let us assume that the last entry of $\mathbf{e}$ can take any value in $\mathbb{F}_q$. \\

Continuing, noting that the number of $\mathbf{f}$ in the kernel of $H_{l+1 , m+1} (\mathbfe{\alpha})$ is $q^{m+1 - \rank H_{l+1 , m+1} (\mathbfe{\alpha})}$, we have
\begin{align*}
\sum_{A \in \mathcal{M}_n} d(A)^2 
\approx &q \sum_{\mathbfe{\alpha} \in \mathbb{F}_q^{n+1} } 
	\bigg( \sum_{\substack{l,m \geq 0 \\ l+m=n }} 
		q^{- \rank H_{l+1 , m+1} (\mathbfe{\alpha})} \bigg)^2 .
\end{align*}

So, we can now see how using additive characters allows us to express the sum $\sum_{A \in \mathcal{M}_n} d(A)^2 $ in terms of Hankel matrices in a concise manner, and how knowing the exact number of Hankel matrices of a given rank and size will allow us to obtain an exact evaluation of the original divisor sum. \\

Now, Theorem \ref{main theorem, variance divisor function intervals} is concerned with the variance of the divisor function. That is, it is concerned with the sum
\begin{align*}
\frac{1}{q^n} \sum_{A \in \mathcal{M}_n} \lvert \Delta_2 (A;h) \rvert^2
= &\frac{1}{q^n} \sum_{A \in \mathcal{M}_n} \Big\lvert \mathcal{N}_{d_2} (A;h) - q^{h} (n+1) \Big\rvert^2 \\
= &\frac{1}{q^n} \sum_{A \in \mathcal{M}_n} \bigg( \sum_{\substack{B \in \mathcal{M}_n \\ \degree (B-A) < h }} d(B) \bigg)^2
	\hspace{1em} - q^{2h} (n+1)^2 ,
\end{align*}
and it suffices to consider
\begin{align*}
\sum_{A \in \mathcal{M}_n} \bigg( \sum_{\substack{B \in \mathcal{M}_n \\ \degree (B-A) < h }} d(B) \bigg)^2
\end{align*}
which is similar to the sum $\sum_{A \in \mathcal{M}_n} d(A)^2 $ that we worked with earlier. Here, we have the sum over $A$ which appears outside the squared parentheses, while the sum over $B$ appears within. We can proceed similar to previously, and the sum over $A$ will force $\alpha_k + \beta_k = 0$ for $k = h , h+1 , \ldots , n$. Whereas, the sum over $B$, will force $\alpha_k = 0$ and $\beta_k = 0$ for $k=0 , 1 , \dots , h-1$. Ultimately, this means we will have to understand how many Hankel matrices there are of a given size and rank with the first $h$ skew-diagonals being $0$. \\

Note that if the first $h$ entries of $\mathbfe{\alpha}$ are zero, and $h \geq \frac{n}{2}$, then the matrix $H_{\frac{n}{2} +1 , \frac{n}{2} +1} (\mathbfe{\alpha})$ is lower skew-triangular (for simplicity we are assuming $n$ is even and so $\frac{n}{2}$ is an integer). We can easily determine the rank of such a matrix by determining the first non-zero skew diagonal; and, as we will see later, we can use that to easily determine the rank of all $H_{l , m} (\mathbfe{\alpha})$ for $l+m-2 = n$. On the other hand, if $h < \frac{n}{2}$, then it is more difficult to determine the rank of the matrix $H_{\frac{n}{2} +1 , \frac{n}{2} +1} (\mathbfe{\alpha})$, which is no longer necessarily lower skew-triangular. This demonstrates why it is easier to work with large intervals ($h \geq \frac{n}{2}$) than short intervals ($h < \frac{n}{2}$). Note that the condition $h < \frac{n}{2}$ is equivalent to $q^h = (q^n)^{\frac{1}{2}}$ which is analogous to the classical $H < x^{\frac{1}{2}}$ (see (\ref{def statement, Delta_k (x;H) NF def}) and the paragraph below that).

\subsection{Extensions} \label{subsection, intro, extensions}

The first extension that we consider is the analogue of Theorem \ref{main theorem, variance divisor function intervals} for the $k$-th divisor function, $d_k$. The approach is similar but we will ultimately be working with Hankel tensors instead of Hankel matrices. A matrix is a two dimensional array, which appeared because we were working with the standard divisor function, $d=d_2$. When working with higher divisor functions, we will work with higher dimensional arrays (i.e. tensors). Consider the case $k=3$. We will have tensors of the form 
\begin{align*}
(\alpha_{i+j+l-3})_{\substack{1 \leq i \leq i_1 \\ 1 \leq j \leq j_1 \\ 1 \leq l \leq l_1}}
\end{align*}
and we will need to determine how many $\mathbf{f} = (f_1 , \ldots , f_{j_1})^T$ and $\mathbf{g} = (g_1 , \ldots , g_{l_1})^T$ there are such that
\begin{align*}
\sum_{j=1}^{j_1} \sum_{l=1}^{l_1} \alpha_{i+j+l-3} f_j g_l = 0
\end{align*}
for all $1 \leq i \leq i_1$. This is analogous to (\ref{statement, intro, motivation, H_(l+1,m+1) (alpha) f = 0 requirement}). \\

Let us now consider another extension. In Theorem \ref{main theorem, variance divisor function intervals}, we consider the variance of the divisor function, which is essentially the second moment. We could consider higher moments, such as the third, and if we are aiming to obtain an exact evaluation then it would be sufficient to obtain an exact formula for
\begin{align*}
\sum_{A \in \mathcal{M}_n} \bigg( \sum_{\substack{B \in \mathcal{M}_n \\ \degree (B-A) < h }} d(B) \bigg)^3 .
\end{align*}
With the variance, we needed to determine how many $\mathbfe{\alpha} , \mathbfe{\beta}$ there are such that $H_{l+1 , m+1}(\mathbfe{\alpha})$ and $H_{l+1 , m+1}(\mathbfe{\beta})$ have certain given ranks, and $\mathbfe{\alpha} + \mathbfe{\beta} = \mathbfe{0}$. Of course, the last condition means the matrices are essentially identical, making the problem slightly simpler. However, for the third moment, we will need to determine how many $\mathbfe{\alpha} , \mathbfe{\beta} , \mathbfe{\gamma}$ there are such that $H_{l+1 , m+1}(\mathbfe{\alpha})$, $H_{l+1 , m+1}(\mathbfe{\beta})$, and $H_{l+1 , m+1}(\mathbfe{\gamma})$ have certain given ranks, and $\mathbfe{\alpha} + \mathbfe{\beta} + \mathbfe{\gamma} = \mathbfe{0}$. This last condition makes the problem more difficult. We do indicate in Remark \ref{remark, hankel matrices, kernel structure, higher moments remark} how we can reduce this problem to special cases of $\mathbfe{\alpha} , \mathbfe{\beta} , \mathbfe{\gamma}$. We are effectively taking two Hankel matrices for which we understand their rank, and we wish to understand the rank of their sum. Related problems have been considered for Toeplitz-plus-Hankel matrices over the complex numbers \cite{Heinig2002_KernStructToeplitzPlusHankelMatr, Rojo2008_NewAlgeToeplitzPlusHankelMatrAppl} (compared to Hankel-plus-Hankel matrices, which is what we are interested in). \\

If we can obtain a result such as Theorem \ref{theorem, divisor correlations d(KQ+N)d(N)}, or at least a strong approximation, for the case $k < \degree Q -1$, then this would allow us to obtain lower order terms in the asymptotic expansion of the fourth moment of Dirichlet $L$-functions in function fields (the average would be over characters of prime modulus, and this would be analogous to Young's results \cite{FourthMomDirLFunc_Young_2011}). We discuss what is required for this in Remark \ref{remark, extending FF Dir L func Corr to k < deg Q -1}, after having established the necessary results and notation in Section \ref{Section, Hankel matrices}. For now, we can give the following indication: We must understand how many $\mathbfe{\alpha} = (\alpha_0 , \alpha_1 , \ldots , \alpha_{\degree Q +k})$ there are that satisfy the following three conditions: 
\begin{itemize}
\item The square matrix $H_{\frac{\degree Q +k}{2} +1 , \frac{\degree Q +k}{2} +1} (\mathbfe{\alpha})$ has rank $r_1$; \\

\item The square matrix $H_{\frac{n}{2} +1 , \frac{n}{2} +1} (\mathbfe{\alpha}')$ has rank $r_2$; \\

\item The matrix $H_{\degree Q +1 , k +1} (\mathbfe{\alpha})$ has $(q_0 , q_1 , \ldots , q_{\degree Q} )^T$ in its kernel; 
\end{itemize}
where $\mathbfe{\alpha}'$ is defined to be the subsequence $(\alpha_0 , \alpha_1 , \ldots , \alpha_{n})$, the integers $r_1 , r_2$ are fixed, and we are assuming for simplicity that $\degree Q + k$ and $n$ are even. Also, $q_0 , q_1 , \ldots , q_{\degree Q}$ are defined by $Q = q_0 + q_1 T + \ldots + q_{\degree Q} T^{\degree Q}$. Working with any two of the above conditions is possible given the results we establish later. However, the difficulty lies in working with all three. \\

If we were to work with higher moments of Dirichlet $L$-functions then we would consider correlations of higher divisor functions (such as $d_3$), and so we would need to work with Hankel tensors instead of Hankel matrices. Specifically, we would have a tensor analogue for the three conditions above. Of course, moments higher than the fourth are notoriously difficult, and no rigorous results have been obtained. However, the above is interesting as it provides an alternative approach to the problem.\\

Finally, we can consider arithmetic functions other than the divisor function. For example, the number of ways we can express a polynomial in $\mathcal{A}$ as a sum of two squares. As with the divisor function, this involves multiplication, and our approach of Hankel matrices can be applied.

\section{Hankel Matrices over $\mathbb{F}_q$} \label{Section, Hankel matrices}

\subsection{Introduction}

While we are concerned with Hankel matrices over finite fields, Hankel matrices over the complex numbers have received considerably more attention. Heinig and Rost \cite{HeinigRost1984_AlgMethToeplitzMatrOperat} provide a detailed account of the results that have been established for the complex setting. While fewer in number, there are publications specifically on Hankel matrices over finite fields as well \cite{Garcia-ArmasGhorpadeRam2011_RelativePrimePolyNonsingHankelMatrFinField, Meshulam1993_SpaceHankelMatrFinField}. About half of the results we provide are completely original; while the rest are either finite field analogies of results in \cite{HeinigRost1984_AlgMethToeplitzMatrOperat} or generalisations of results in \cite{Garcia-ArmasGhorpadeRam2011_RelativePrimePolyNonsingHankelMatrFinField}, but the proofs are often different with the intention of being more intuitive. Wherever possible, we will adhere to the notation established in \cite{HeinigRost1984_AlgMethToeplitzMatrOperat, Garcia-ArmasGhorpadeRam2011_RelativePrimePolyNonsingHankelMatrFinField}, and when this is not possible we make clear what the differences are. \\

As mentioned previously, an $l \times m$ Hankel matrix over $\mathbb{F}_q$ is a matrix of the form
\begin{align}
\begin{split} \label{definition, Hankel matrices}
(\alpha_{i+j-2})_{\substack{1 \leq i \leq l \\ 1 \leq j \leq m}}
= \begin{pNiceMatrix}
\alpha_0 & \alpha_1 & \alpha_2 & \Cdots & \Cdots & \Cdots & \alpha_{m-1} \\
\alpha_1 & \alpha_2 &  &  &  &  & \Vdots \\
\alpha_2 &  &  &  &  &  & \Vdots \\
\Vdots &  &  &  &  &  & \Vdots \\
\Vdots &  &  &  &  &  & \alpha_{l+m-4} \\
\Vdots &  &  &  &  & \alpha_{l+m-4} & \alpha_{l+m-3} \\
\alpha_{l-1}  & \Cdots & \Cdots & \Cdots & \alpha_{l+m-4} & \alpha_{l+m-3} & \alpha_{l+m-2}
\end{pNiceMatrix} ,
\end{split}
\end{align}
where $\alpha_0 , \ldots , \alpha_{l+m-2} \in \mathbb{F}_q$. As we can see, all the entries on a given skew-diagonal take the same value. We index our entries from zero, and we later see this is necessary in order to be consistent with the indexing of coefficients in a polynomial, which also begins at zero. \\

Define the $n \times n$ counter-identity matrix, $J_n$, to be the matrix with zero entries everywhere except for the main skew-diagonal going from bottom-left to top-right. If $H$ is an $l \times m$ Hankel matrix, then $J_l H$ and $H J_m$ are $l \times m$ Toeplitz matrices. Similarly, if $T$ is an $l \times m$ Toeplitz matrix, then $J_l T$ and $T J_m$ are $l \times m$ Hankel matrices. Thus, we can see that Hankel matrices are inextricably linked to the more well known Toeplitz matrices. Although, we will focus on the former. We denote the set of all $l \times m$ Hankel matrices in $\mathbb{F}_q [T]$ by $\mathscr{H}_{l,m}$. \\

It is natural to consider the finite sequence $( \alpha_0 , \alpha_1 , \ldots , \alpha_{l+m-2}) \in \mathbb{F}_q^{l+m-1}$ that is associated to the matrix (\ref{definition, Hankel matrices}). Generally, for $\mathbfe{\alpha} := (\alpha_0 , \ldots , \alpha_n ) \in \mathbb{F}_q^{n+1}$ and $l+m-2 = n$ with $l,m \geq 1$, we define the matrices
\begin{align}
\begin{split} \label{definition, Hankel matrices associated to a given sequence}
H_{l , m} (\mathbfe{\alpha})  
:= \begin{pNiceMatrix}
\alpha_0 & \alpha_1 & \alpha_2 & \Cdots & \Cdots & \Cdots & \alpha_{m-1} \\
\alpha_1 & \alpha_2 &  &  &  &  & \Vdots \\
\alpha_2 &  &  &  &  &  & \Vdots \\
\Vdots &  &  &  &  &  & \Vdots \\
\Vdots &  &  &  & &  & \alpha_{l+m-4} \\
\Vdots &  &  &  &  & \alpha_{l+m-4} & \alpha_{l+m-3} \\
\alpha_{l-1} & \Cdots & \Cdots & \Cdots & \alpha_{l+m-4} & \alpha_{l+m-3} & \alpha_{l+m-2}
\end{pNiceMatrix} .
\end{split}
\end{align}
That is, we associate $n+1$ number of matrices with $\mathbfe{\alpha}$. As we will later see, there is a crucial relationship between the kernels of these matrices. We note that we can extend the above definition to the case where $l+m-2 = n' < n$, in which case we have $n' +1$ number of matrices, and the last $n-n'$ entries of $\mathbfe{\alpha}$ do not appear in any of them. \\

Throughout this paper, for an integer $n \geq 0$, we will always define $n_1 := \lfloor \frac{n+2}{2} \rfloor$ and $n_2 := \lfloor \frac{n+3}{2} \rfloor$. \\

Now, let $\mathbfe{\alpha} = (\alpha_0 , \ldots , \alpha_{n}) \in \mathbb{F}_{q}^{n+1}$, and consider the square matrices
\begin{align*}
H_{1,1} (\mathbfe{\alpha}) , H_{2,2} (\mathbfe{\alpha}) , \ldots , H_{n_1 , n_1} (\mathbfe{\alpha}) .
\end{align*}
Note that all of these matrices have at most $n+1$ skew-diagonals and so, given the length of the sequence  $\mathbfe{\alpha}$, they are well defined. Intuitively, they are all the square Hankel matrices that can be obtained from $\mathbfe{\alpha}$ or its truncations; and they are all top-left submatrices of $H_{n_1 , n_2} ( \mathbfe{\alpha})$. Now, if at least one of these matrices has non-zero determinant, then define $\rho (\mathbfe{\alpha})$ to be the largest $l \in \{ 1 ,2 , \ldots , n_1 \}$ with the property that $\det H_{l , l } (\mathbfe{\alpha}) \neq 0$. If all of these matrices have determinant equal to zero, then define $\rho (\mathbfe{\alpha})$ to be zero. \\

Now consider the matrix $H_{n_1 , n_2} (\mathbfe{\alpha})$, which is square if $n$ is even and almost square if $n$ is odd. Note that it has exactly $n+1$ skew-diagonals, and so each entry in $\mathbfe{\alpha}$ appears in the matrix. We define 
\begin{align*}
\pi (\mathbfe{\alpha})
:= \rank H_{n_1 , n_2} (\mathbfe{\alpha}) - \rho (\mathbfe{\alpha}) .
\end{align*}

We will later see that $\rho (\mathbfe{\alpha})$ and $\pi (\mathbfe{\alpha})$ will help us in understanding the rank and kernel of not just $H_{n_1 , n_2} (\mathbfe{\alpha})$ but all $H_{l , m} (\mathbfe{\alpha})$ for $l+m-2 =n$. \\

Thus, it is appropriate to extend the definitions of $\rho (\mathbfe{\alpha})$ and $\pi (\mathbfe{\alpha})$ to the matrices associated to $\mathbfe{\alpha}$: For $l+m-2 = n$, we define $\rho \big( H_{l,m} (\mathbfe{\alpha}) \big) = \rho (\mathbfe{\alpha})$ and $\pi \big( H_{l,m} (\mathbfe{\alpha}) \big) = \pi (\mathbfe{\alpha})$. In particular, these properties are dependent only on the underlying $\mathbfe{\alpha}$ and not on the shape of the matrix. Note that this definition does not apply to matrices with $l+m-2 < n$; in these cases, we will need to work with an appropriate truncation $\mathbfe{\alpha}'$ of $\mathbfe{\alpha}$ and work with $\rho (\mathbfe{\alpha}')$ and $\pi (\mathbfe{\alpha}')$ instead. \\

There are a couple of remarks we must make before continuing. 

\begin{remark}
In \cite{Garcia-ArmasGhorpadeRam2011_RelativePrimePolyNonsingHankelMatrFinField} they give the same definition for $\rho$, although they use the letter $\delta$ instead and it applies only to square Hankel matrices. We use the letter $\rho$ because it is consistent with the notation established in \cite[Subsection 5.6]{HeinigRost1984_AlgMethToeplitzMatrOperat}. Here they define the $(\rho , \pi )$-characteristic of a Hankel matrix $H$ (denoted by $\HankelChar H$) to be $\big( \rho (H) , \pi (H) \big)$. Technically, they give a different definition; although, the results we establish later allow us to see that it is equivalent to the definition we give. The benefit of our definition is that it can be given before introducing results on the kernel structure of Hankel matrices. 
\end{remark}

\begin{remark} \label{remark, bounds for rho_1 and pi_1}
Suppose $\mathbfe{\alpha} = (\alpha_0 , \ldots , \alpha_n) \in \mathbb{F}_q^{n+1}$. By definition, $\rho (\mathbfe{\alpha}) $ can take values in $\{ 0 , 1, \ldots , n_1 \}$. Given that $\pi (\mathbfe{\alpha}) := \rank H_{n_1 , n_2 } (\mathbfe{\alpha}) - \rho (\mathbfe{\alpha}) \leq n_1 - \rho (\mathbfe{\alpha})$, we can see that $\pi (\mathbfe{\alpha})$ can take values in $\{ 0 , 1 , \ldots , n_1 - \rho (\mathbfe{\alpha}) \}$. \\

In fact, these are all attainable if $n$ is odd, but if $n$ is even then $\pi (\mathbfe{\alpha})$ can only attain values in $\big\{ 0 , 1 , \ldots , \max \{ 0 , n_1 - \rho (\mathbfe{\alpha}) -1 \} \big\}$. That is, we cannot have $\pi (\mathbfe{\alpha}) = n_1 - \rho (\mathbfe{\alpha})$ when $\rho (\mathbfe{\alpha}) \neq n_1$. Indeed, for a contradiction, suppose we do have $\pi (\mathbfe{\alpha}) = n_1 - \rho (\mathbfe{\alpha})$ and $\rho (\mathbfe{\alpha}) \neq n_1$. Then,
\begin{align*}
\rank H_{n_1 , n_1 } (\mathbfe{\alpha})
= \rank H_{n_1 , n_2 } (\mathbfe{\alpha})
= \rho (\mathbfe{\alpha}) + \pi (\mathbfe{\alpha})
= n_1 ,
\end{align*}
where the first equality uses the fact that $n_1 = n_2$ (since $n$ is even). Thus, $H_{n_1 , n_1 } (\mathbfe{\alpha})$ has full rank and so $\rho (\mathbfe{\alpha}) = n_1$ by definition, which obviously contradicts $\rho (\mathbfe{\alpha}) \neq n_1$. When $n$ is odd, we have $n_1 \neq n_2$ and so the first relation above does not hold, meaning we do not encounter this contradiction.
\end{remark}

We now make several definitions for sets of finite sequences and Hankel matrices.

\begin{definition}
We define,
\begin{align*}
\mathscr{L}_n (r) := &\{ \mathbfe{\alpha} \in \mathbb{F}_q^{n+1} : \rank H_{n_1 , n_1 } (\mathbfe{\alpha}) = r \} , \\
\mathscr{L}_n (r , \rho_1 , \pi_1 ) := &\{ \mathbfe{\alpha} \in \mathscr{L}_n (r) : \rho (\mathbfe{\alpha}) = \rho_1 , \pi (\mathbfe{\alpha}) = \pi_1 \} .
\end{align*}
Of course, by definition of $\pi (\mathbfe{\alpha})$, we must have $\rho_1 + \pi_1 = r$, and so at times we may write $\mathscr{L}_n (\rho_1 + \pi_1 , \rho_1 , \pi_1 )$ or $\mathscr{L}_n (r, \rho_1 , r - \rho_1 )$, depending on what parameters we are using. We also define, for $h=0 , \ldots , n+1$,
\begin{align*}
\mathscr{L}_n^h := &\{ \mathbfe{\alpha} = (\alpha_0 , \ldots , \alpha_n ) \in \mathbb{F}_q^{n+1} : \alpha_0 , \ldots ,\alpha_{h-1} = 0 \} , \\
\mathscr{L}_n^h (r) := &\{ \mathbfe{\alpha} = (\alpha_0 , \ldots , \alpha_n ) \in \mathscr{L}_n (r) : \alpha_0 , \ldots ,\alpha_{h-1} = 0 \} , \\
\mathscr{L}_n^h (r , \rho_1 , \pi_1 ) := &\{ \mathbfe{\alpha} = (\alpha_0 , \ldots , \alpha_n ) \in \mathscr{L}_n (r, \rho_1 , \pi_1) : \alpha_0 , \ldots ,\alpha_{h-1} = 0 \} .
\end{align*}
In the above, we have three sets of parameters. The first relates to the length of the sequence $\mathbfe{\alpha}$, the second relates to the rank of the associated square (or nearly square) Hankel matrix, and the third relates to entries equal to zero at the start of the sequence. Note that when $h=0$ we have $\mathscr{L}_n^h = \mathbb{F}_q^{n+1}$, $\mathscr{L}_n^h (r) = \mathscr{L}_n (r)$, and $\mathscr{L}_n^h (r, \rho_1 , \pi_1 ) = \mathscr{L}_n (r, \rho_1 , \pi_1 )$. \\

For Hankel matrices, we make the following definitions:
\begin{align*}
\mathscr{H}_{l,m} (r) := & \{ H \in \mathscr{H}_{l,m}: \rank H = r \} , \\
\mathscr{H}_{l,m}^h := & \{ H = (\alpha_{i+j-2})_{\substack{1 \leq i \leq l \\ 1 \leq j \leq m}} \in \mathscr{H}_{l,m} : \alpha_ 0 , \ldots , \alpha_{h-1} = 0 \} , \\
\mathscr{H}_{l,m}^h (r) := & \{ H = (\alpha_{i+j-2})_{\substack{1 \leq i \leq l \\ 1 \leq j \leq m}} \in \mathscr{H}_{l,m} (r) : \alpha_ 0 , \ldots , \alpha_{h-1} = 0 \} ,
\end{align*}
for $h = 0 , \ldots , l+m-1$. Note that when $h=0$ we have $\mathscr{H}_{l,m}^h = \mathscr{H}_{l,m}$ and $\mathscr{H}_{l,m}^h (r) = \mathscr{H}_{l,m} (r)$. Note also that the parameter $r$ appearing in $\mathscr{L}_n (r)$ is not analogous to the parameter $r$ appearing in $\mathscr{H}_{l,m} (r)$. For example, if $\mathbfe{\alpha} \in \mathscr{L}_n (r)$ and $l+m-2 = n$, then we do not necessarily have $H_{l,m} (\mathbfe{\alpha}) \in \mathscr{H}_{l,m} (r)$; indeed, if $l<r$ then we have $H_{l,m} (\mathbfe{\alpha}) \in \mathscr{H}_{l,m} (l)$.
\end{definition}

As we will see later, the number of zeros that appear at the start of our matrix is important for the variance of the divisor function over intervals. The definitions above incorporate various parameters, which are not all considered in \cite{HeinigRost1984_AlgMethToeplitzMatrOperat} or \cite{Garcia-ArmasGhorpadeRam2011_RelativePrimePolyNonsingHankelMatrFinField}. Thus, our notation is different. \\

\begin{remark} \label{remark, bounds on h parameter for Hankel matrices}
Consider $\mathscr{L}_n^h (r, \rho_1 , \pi_1 )$, and suppose $\rho_1 \neq 0$. We must have that $h \leq \rho_1 -1$. Otherwise, for any $\mathbfe{\alpha} \in \mathscr{L}_n^h (r, \rho_1 , \pi_1 )$ we would have that $H_{\rho_1 , \rho_1} (\mathbfe{\alpha})$ is strictly lower skew-triangular\footnote{We say a square matrix is lower/upper skew-triangular if all entries above/below the main skew-diagonal are zero, and it is strictly so if the entries on the main skew diagonal are also zero.} and thus contradicting that $\det H_{\rho_1 , \rho_1} (\mathbfe{\alpha}) \neq 0$. \\

Now suppose that $\rho_1 = 0$. Then, $h \leq n+1-r$. This can be seen from the following reasoning. Since $\rho_1 = 0$, we have
\begin{align*}
\det H_{1,1} (\mathbfe{\alpha}) = 0 , \hspace{1em} \ldots , \hspace{1em} \det H_{n_1 , n_1} (\mathbfe{\alpha}) = 0 ,
\end{align*}
and so, by induction (as we later demonstrate), we have that $H_{n_1 , n_1} (\mathbfe{\alpha})$ is strictly lower-skew triangular. Thus, the matrix $H_{n_1 , n_2} (\mathbfe{\alpha})$ is of the form
\begin{align*}
\begin{pNiceMatrix}
0 & \Cdots & \Cdots & \Cdots & 0 \\
0 & \Cdots & \Cdots & 0 & \alpha_{n_1 } \\
0 & \Cdots & 0 & \alpha_{n_1 } & \alpha_{n_1 +1} \\
\Vdots & \Iddots & \Iddots & \Iddots & \Vdots \\
0 & \alpha_{n_1 } & \alpha_{n_1 +1} & \Cdots & \alpha_n
\end{pNiceMatrix}
\hspace{2em} \text{ or } \hspace{2em}
\begin{pNiceMatrix}
0 & \Cdots & \Cdots & 0 & \alpha_{n_1 } \\
0 & \Cdots & 0 & \alpha_{n_1 } & \alpha_{n_1 +1} \\
\Vdots & \Iddots & \Iddots & \Iddots & \Vdots \\
0 & \alpha_{n_1 } & \alpha_{n_1 +1} & \Cdots & \alpha_n
\end{pNiceMatrix}
\end{align*}
if $n$ is even or odd, respectively. Given that $\rank H_{n_1 , n_2} (\mathbfe{\alpha}) = r$, we can see that $H_{n_1 , n_2} (\mathbfe{\alpha})$ must be of the form
\begin{align*}
\begin{pNiceMatrix}
0 & \Cdots & \Cdots & \Cdots & \Cdots & \Cdots & 0 \\
\Vdots &  &  &  &  &  & \Vdots \\
0 & \Cdots & \Cdots & \Cdots & \Cdots & \Cdots & 0 \\
0 & \Cdots & \Cdots & \Cdots & \Cdots & 0 & \alpha_{n+1-r} \\
0 & \Cdots & \Cdots & \Cdots & 0 & \alpha_{n+1-r} & \alpha_{n+2-r} \\
\Vdots &  &  & \Iddots & \Iddots & \Iddots & \Vdots \\
0 & \Cdots & 0 & \alpha_{n+1-r} & \alpha_{n+2-r} & \Cdots & \alpha_n
\end{pNiceMatrix} 
\end{align*}
with $\alpha_{n+1-r} \neq 0$. (If $n$ is odd and $r=n_1$, then there should be no rows of zeros at the top, and the matrix above should be interpreted as such). In particular, this forces $h \leq n+1-r$, as required.
\end{remark}

\begin{definition}[Quasi-regular]
Suppose we have $\mathbfe{\alpha} = (\alpha_0 , \alpha_1 , \ldots , \alpha_n ) \in \mathbb{F}_q^{n+1}$. We say that $\mathbfe{\alpha}$ is quasi-regular if $\pi (\mathbfe{\alpha}) = 0$. For any integers $l,m \geq 1$ with $l+m-2=n$ we say $H_{l,m} (\mathbfe{\alpha})$ is quasi-regular if $\mathbfe{\alpha}$ is quasi-regular. Quasi-regularity is defined in Definition 5.8 of \cite{HeinigRost1984_AlgMethToeplitzMatrOperat}.
\end{definition}

Before proceeding, we make a couple of remarks on notation. Let $M$ be an $l \times m$ matrix and let $l_1 , l_2 , m_1 , m_2$ satisfy $l_1 + l_2 \leq l$ and $m_1 + m_2 \leq m$. Then, we define $M[l_1 , -l_2 ; m_1 , -m_2]$ to be the submatrix of $M$ consisting of the first $l_1$ and last $l_2$ rows, and the first $m_1$ and last $m_2$ columns. In the special cases when one or more of $l_1 , l_2 , m_1 , m_2$ are zero, we may not include them. For example, $M[l_1 ; -m_2]$ should be taken to be $M[l_1 , 0 ; 0 , -m_2 ]$. \\

There will be times where we will use the matrix
\begin{align*}
\begin{pNiceMatrix}
0 & \Cdots & \Cdots & \Cdots & \Cdots & 0 \\
\Vdots &  &  &  &  & \Vdots \\
0 & \Cdots & \Cdots & \Cdots & \Cdots & 0 \\
0 & \Cdots & \Cdots & \Cdots & 0 & \alpha_{(n+1)-i} \\
\Vdots &  &  & \Iddots & \Iddots & \Vdots \\
0 & \Cdots & 0 & \alpha_{(n+1)-i} & \Cdots & \alpha_n
\end{pNiceMatrix} ,
\end{align*}
perhaps with different letters and indexing. If $i$ is equal to the number of rows of the matrix above, then there should be no rows of zeros at the top of the matrix. In that case, the matrix should be interpreted as such even if there are rows of zeros indicated. Similarly, if $i$ is equal to the number of columns, then the matrix above should be interpreted as having no columns of zeros on the left. This is to avoid unnecessary technicalities when we are working with a range of values of $i$.

\subsection{The $(\rho , \pi )$-form of a Hankel Matrix}\label{subsection, rho pi form of Hankel matrix}

We are now able to introduce the $(\rho , \pi )$-form of a Hankel matrix. Generally, we apply a series of row operations to transform the matrix into one whose structure demonstrates the $(\rho , \pi )$-characteristic of the original matrix. It allows us to understand the ranks and kernels of Hankel matrices more easily, as we will see in Subsections \ref{subsection, matrices of given size and rank} and \ref{subsection, kernel structure}. This form was used for square Hankel matrices in \cite{Garcia-ArmasGhorpadeRam2011_RelativePrimePolyNonsingHankelMatrFinField}, although no terminology for this was given there, or elsewhere, as far as we are aware. Thus, we have introduced the terminology of ``$(\rho , \pi )$-form". We require a few results before giving the definition of $(\rho , \pi )$-form. \\

\begin{lemma} \label{lemma, Hankel rho_1 = 0 form}
Suppose $\mathbfe{\alpha} = (\alpha_0 , \ldots , \alpha_n )$ with $\rho (\mathbfe{\alpha}) = 0$ and let
\begin{align*}
\pi_1 
\in \begin{cases}
\{ 0 , 1 , \ldots , n_1 \} &\text{ if $n$ is odd,} \\
\{ 0 , 1 , \ldots , n_1 - 1 \} &\text{ if $n$ is even} 
\end{cases}
\end{align*}
(See Remark \ref{remark, bounds for rho_1 and pi_1} regarding the values that $\pi_1$ can take). Then, $\pi (\mathbfe{\alpha}) = \pi_1$ if and only if
\begin{align*}
\alpha_i
\in \begin{cases}
\{ 0 \} &\text{ for $i < (n+1) - \pi_1$} \\
\mathbb{F}_q^* &\text{ for $i = (n+1) - \pi_1$} \\
\mathbb{F}_q &\text{ for $i > (n+1) - \pi_1$} .
\end{cases}
\end{align*}
\end{lemma}

\begin{proof}
We begin with the forward implication. Let $H := H_{n_1 , n_2} (\mathbfe{\alpha})$. Since $\rho (\mathbfe{\alpha}) = 0$, we have $\det H [i;i] = 0$ for $i=1 , \ldots n_1$. When $i=1$ this gives
\begin{align*}
0 
= \det H [1;1] 
= \alpha_0 .
\end{align*}
When $i=2$ it gives
\begin{align*}
0 
= \det H [2;2] 
= \det \begin{pmatrix} 0 & \alpha_1 \\ \alpha_1 & \alpha_2 \end{pmatrix} 
= - {\alpha_1 }^2 ,
\end{align*}
meaning $\alpha_1 = 0$. Proceeding as above in an inductive manner, we deduce that $\alpha_0 , \ldots , \alpha_{n_1 -1} = 0$. Thus, if $n$ is even or odd, we have
\begin{align*}
H
= \begin{pNiceMatrix}
0 & \Cdots & \Cdots & \Cdots & 0 \\
0 & \Cdots & \Cdots & 0 & \alpha_{n_1} \\
0 & \Cdots & 0 & \alpha_{n_1} & \alpha_{n_1 +1} \\
\Vdots & \Iddots & \Iddots & \Iddots & \Vdots \\
0 & \alpha_{n_1 } & \alpha_{n_1 +1} & \Cdots & \alpha_n
\end{pNiceMatrix}
\hspace{2em} \text{ or } \hspace{2em}
H =
\begin{pNiceMatrix}
0 & \Cdots & \Cdots & 0 & \alpha_{n_1} \\
0 & \Cdots & 0 & \alpha_{n_1} & \alpha_{n_1 +1} \\
\Vdots & \Iddots & \Iddots & \Iddots & \Vdots \\
0 & \alpha_{n_1 } & \alpha_{n_1 +1} & \Cdots & \alpha_n
\end{pNiceMatrix} ,
\end{align*}
respectively. Now, suppose $i$ is the largest element in the set $\{ 1 , 2 , \ldots , (n+1) - n_1 \}$ satisfying $\alpha_{(n+1) - i} \neq 0$ (such an $i$ must exist, unless we are working with the zero matrix, in which case we take $i=0$). Then, 
\begin{align*}
H
= \begin{pNiceMatrix}
0 & \Cdots & \Cdots & \Cdots & \Cdots & 0 \\
\Vdots &  &  &  &  & \Vdots \\
0 & \Cdots & \Cdots & \Cdots & \Cdots & 0 \\
0 & \Cdots & \Cdots & \Cdots & 0 & \alpha_{(n+1)-i} \\
\Vdots &  &  & \Iddots & \Iddots & \Vdots \\
0 & \Cdots & 0 & \alpha_{(n+1)-i} & \Cdots & \alpha_n
\end{pNiceMatrix} .
\end{align*}
Recalling that $\alpha_{(n+1)-i} \neq 0$, we can clearly see that $\rank H = i$. Since $\rank H = \rho (\mathbfe{\alpha}) + \pi (\mathbfe{\alpha}) = \pi_1$, we see that $i=\pi_1$. In particular, $\alpha_{(n+1) - \pi_1} \neq 0$; while for $i < (n+1) - \pi_1$ we have $\alpha_i = 0$, and for $i > (n+1) - \pi_1$ we have $\alpha_i \in \mathbb{F}_q$. This concludes the forward implication. \\

The backward implication follows easily given some of the reasoning that we have established above.
\end{proof}

\begin{remark} \label{remark, Hankel rho_1 = 0 form, visualisation}
Let $l+m -2 = n$. It is helpful to visualise what the matrix $H_{l,m} (\mathbfe{\alpha})$ looks like given $\rho (\mathbfe{\alpha}) = 0$ and $\pi (\mathbfe{\alpha}) = \pi_1$. For presentational purposes, we use $1$ to denote an entry in $\mathbb{F}_q^*$, we use $*$ to denote an entry in $\mathbb{F}_q$, and $0$ denotes $0$ as usual. For $l < \pi_1$, we have
\begin{align*}
H_{l,m} (\mathbfe{\alpha})
= \begin{pNiceMatrix}
0 & \Cdots & \Cdots & \Cdots & \Cdots & 0 & 1 & * & \Cdots & * \\
0 & \Cdots & \Cdots & \Cdots & 0 & 1 & * & \Cdots & \Cdots & * \\
\Vdots &  &  & \Iddots & \Iddots & \Iddots &  &  &  & \Vdots \\
0 & \Cdots & 0 & 1 & * & \Cdots & \Cdots & \Cdots & \Cdots & *
\end{pNiceMatrix} .
\end{align*}
This has full row rank. We can describe the kernel in some manner, although it is more helpful when $l,m \geq \pi_1$. In this case, we have
\begin{align*}
H_{l,m} (\mathbfe{\alpha})
= \begin{pNiceMatrix}
0 & \Cdots & \Cdots & \Cdots & \Cdots & \Cdots & 0 \\
\Vdots &  &  &  &  &  & \Vdots \\
0 & \Cdots & \Cdots & \Cdots & \Cdots & \Cdots & 0 \\
0 & \Cdots & \Cdots & \Cdots & \Cdots & 0 & 1 \\
0 & \Cdots & \Cdots & \Cdots & 0 & 1 & * \\
\Vdots &  &  & \Iddots & \Iddots & \Iddots & \Vdots \\
0 & \Cdots & 0 & 1 & * & \Cdots & * 
\end{pNiceMatrix} ,
\end{align*}
where there are exactly $\pi_1$ number of $1$s. This has rank equal to $\pi_1$, we can clearly see that any vector in the kernel must have zeros in its last $\pi_1$ positions. For $m < \pi_1$, we have
\begin{align*}
H_{l,m} (\mathbfe{\alpha})
= \begin{pNiceMatrix}
0 & 0 & \Cdots & 0 \\
\Vdots & \Vdots &  & \Vdots \\
\Vdots & \Vdots &  & 0 \\
\Vdots & \Vdots & \Iddots & 1 \\
\Vdots & 0 & \Iddots & * \\
0 & 1 & \Iddots & \Vdots \\
1 & * &  & \Vdots \\
* & \Vdots &  & \Vdots \\
\Vdots & \Vdots &  & \Vdots \\
* & * & \Cdots & * \\
\end{pNiceMatrix} .
\end{align*}
This has full column rank and so the kernel is trivial.
\end{remark}

\begin{lemma} \label{lemma, (rho,pi) form when 1 rho < n_1}
Suppose $\mathbfe{\alpha} = (\alpha_0 , \ldots , \alpha_n )$ with $\rho (\mathbfe{\alpha}) = \rho_1 \in \{ 1 , 2 , \ldots , n_1 -1 \}$ and
\begin{align*}
\pi (\mathbfe{\alpha})
= \pi_1
\in \begin{cases}
\{ 0 , 1 , \ldots , n_1 - \rho_1 \} &\text{ if $n$ is odd,} \\
\{ 0 , 1 , \ldots , n_1 - \rho_1 - 1 \} &\text{ if $n$ is even} 
\end{cases}
\end{align*}
(See Remark \ref{remark, bounds for rho_1 and pi_1} regarding the values that $\pi_1$ can take). Suppose $l+m-2 = n$ with $l > \rho_1$, and let $H := H_{l , m} (\mathbfe{\alpha})$. Define $\mathbf{x} = (x_0 , \ldots , x_{\rho_1 -1})^T$ to be the vector that satisfies
\begin{align*}
H [\rho_1 , \rho_1 ] \mathbf{x}
= \begin{pmatrix}
\alpha_{\rho_1} \\ \alpha_{\rho_1 +1} \\ \vdots \\ \alpha_{2\rho_1 - 1}
\end{pmatrix}.
\end{align*}
Let $R_i$ be the $i$-th row of $H_{l,m} (\mathbfe{\alpha})$. If we apply the row operations
\begin{align*}
R_i \longrightarrow R_i - (x_0 , \ldots , x_{\rho_1 -1}) \begin{pmatrix} R_{i - \rho_1 } \\ \vdots \\ R_{i-1} \end{pmatrix} 
	= R_i - x_0 R_{i - \rho_1 } - \ldots - x_{\rho_1 -1} R_{i-1} 
\end{align*}
for $i = n_1 , n_1 -1 , \ldots , \rho_1 +1$ in that order, then we obtain a matrix
\begin{align} \label{statement, (rho,pi) form when 1 rho < n_1, Hankel above Hankel form}
\begin{pmatrix}
H_{\rho_1 , m} (\mathbfe{\alpha}' )\\
\hline 
H_{l - \rho_1 , m} (\mathbfe{\beta})
\end{pmatrix} ,
\end{align}
where
\begin{align*}
\mathbfe{\alpha}' = &(\alpha_0 , \ldots , \alpha_{\rho_1 + m - 2} ) \\
\mathbfe{\beta} = &(\beta_{\rho_1} , \ldots , \beta_{n})
\end{align*}
and
\begin{align} \label{statement, (rho,pi) form when 1 rho < n_1, beta_i conditions}
\beta_i
\in \begin{cases}
\{ 0 \} &\text{ if $i < (n+1) - \pi_1$} \\
\mathbb{F}_q^* &\text{ if $i = (n+1) - \pi_1$} \\
\mathbb{F}_q &\text{ if $i > (n+1) - \pi_1$} .
\end{cases}
\end{align}
Furthermore, the sequence $\mathbfe{\beta}$ is independent of the specific values taken by $l,m$ (as long as $l > \rho_1$).
\end{lemma}

\begin{remark} \label{remark, lemma (rho,pi) form when 1 rho < n_1 clarification}
In order to keep the lemma above succinct, we avoided various explanatory remarks. We give them here, for clarity. \\

The matrix $H [\rho_1 , \rho_1 ]$ is the largest top-left submatrix of $H$ that is invertible, by definition of $\rho_1$. It is independent of the specific values taken by $l,m$. The vector $(\alpha_{\rho_1} , \alpha_{\rho_1 +1} , \ldots , \alpha_{2\rho_1 - 1})^T$ is simply the column of entries directly to the right of $H [\rho_1 , \rho_1 ]$ in $H$, which is also equal to the transpose of the row of entries directly below $H [\rho_1 , \rho_1 ]$. \\

The rows operations that we apply start at the last row and end at the row just below the submatrix $H [\rho_1 , \rho_1 ]$. \\

The lemma states that after the row operations are applied, we are left with the matrix
\begin{align} \label{statement, (rho,pi) form when 1 rho < n_1 remark, Hankel over Hankel explicit}
\begin{pmatrix}
H_{\rho_1 , m} (\mathbfe{\alpha}' )\\
\hline 
H_{n_1 - \rho_1 , m} (\mathbfe{\beta})
\end{pmatrix}
= \begin{pNiceMatrix}
\alpha_0 & \alpha_1 & \Cdots & \Cdots & \Cdots & \alpha_{m-1} \\
\alpha_1 &  &  &  &  & \Vdots \\
\Vdots &  &  &  &  & \Vdots \\
\Vdots &  &  &  &  & \alpha_{m+\rho_1 - 3} \\
\alpha_{\rho_1-1}  & \Cdots & \Cdots & \Cdots & \alpha_{m+\rho_1 - 3} & \alpha_{m+\rho_1 - 2} \\
\hline
\beta_{\rho_1} & \beta_{\rho_1 +1} & \Cdots & \Cdots & \Cdots & \beta_{m+\rho_1 -1} \\
\beta_{\rho_1 +1} &  &  &  &  & \Vdots \\
\Vdots &  &  &  &  & \Vdots \\
\Vdots &  &  &  &  & \beta_{n-1} \\
\beta_{l-1} & \Cdots & \Cdots & \Cdots & \beta_{n-1} & \beta_{n} 
\end{pNiceMatrix}
\end{align}
Of course, the lemma states that some of the $\beta_i$ are zero, but we do not demonstrate this above so that we can instead see that the indexing is preserved: An entry in the $i$-th skew-diagonal is always either $\alpha_i$ or $\beta_i$.
\end{remark}

\begin{remark} \label{remark, Hankel rho_1 in 1 , ... , n_1 -1 form, visualisation}
As in Remark \ref{remark, Hankel rho_1 = 0 form, visualisation}, it is helpful to visualise what the matrix
\begin{align*}
H'
:= \begin{pmatrix}
H_{\rho_1 , m} (\mathbfe{\alpha}' ) \\
\hline 
H_{l - \rho_1 , m} (\mathbfe{\beta})
\end{pmatrix}
\end{align*}
looks like. Again, for presentational purposes, we use $1$ to denote an entry in $\mathbb{F}_q^*$, we use $*$ to denote an entry in $\mathbb{F}_q$, and $0$ denotes $0$ as usual. If $l \leq \rho_1$, then the lemma above does not apply. We simply remark that in this case $H$ has full row rank, which follows from the fact that $H [ \rho_1 , \rho_1 ]$ is invertible and thus has full rank. \\

Now, suppose $l,m \geq \rho_1 + \pi_1$. Given (\ref{statement, (rho,pi) form when 1 rho < n_1, Hankel above Hankel form}) and (\ref{statement, (rho,pi) form when 1 rho < n_1, beta_i conditions}), we can see that
\begin{align} \label{remark, (rho,pi) form when 1 rho < n_1 lemma, l,m geq rho_1 + pi_1 form}
H'
= \begin{pmatrix}
H [\rho_1 , \rho_1] & \vline & H [\rho_1 , -(m-\rho_1)] \\
\hline
\mathbf{0} & \vline & 
\begin{NiceMatrix}
0 & \Cdots & \Cdots & \Cdots & \Cdots & \Cdots & 0 \\
\Vdots &  &  &  &  &  & \Vdots \\
0 & \Cdots & \Cdots & \Cdots & \Cdots & \Cdots & 0 \\
0 & \Cdots & \Cdots & \Cdots & \Cdots & 0 & 1 \\
0 & \Cdots & \Cdots & \Cdots & 0 & 1 & * \\
\Vdots &  &  & \Iddots & \Iddots & \Iddots & \Vdots \\
0 & \Cdots & 0 & 1 & * & \Cdots & * 
\end{NiceMatrix}
\end{pmatrix} ,
\end{align}
where there are exactly $\pi_1$ number of $1$s in the bottom-right submatrix. One reason that this is helpful is that the bottom two submatrices imply that any vector in the kernel of $H'$ must have zeros in its last $\pi_1$ entries. Given that row operations do not affect the kernel, the same can be said for $H$. Furthermore, we can see that the rank of $H'$ (which is equal to the rank of $H$) is equal to the number of rows of $H [\rho_1 , \rho_1 ]$ (which is invertible) added to the number of $1$s in the bottom-right submatrix. That is, the rank of $H'$ (and $H$) is $\rho_1 + \pi_1 = r$. \\

Now, we wish to consider the case $\rho_1 < l < \rho_1 + \pi_1$ (which requires $\pi_1 \geq 2$). We can do this by repeatedly removing a row and adding a column to (\ref{remark, (rho,pi) form when 1 rho < n_1 lemma, l,m geq rho_1 + pi_1 form}), while maintaining that the bottom two matrices form a Hankel matrix and that the top two matrices form a Hankel matrix. This is permissible because, as stated in Lemma \ref{lemma, (rho,pi) form when 1 rho < n_1}, the sequence $\mathbfe{\beta}$ is independent of the values of $l$ and $m$. We can then see that when $\rho_1 < l < \rho_1 + \pi_1$ we have a matrix of the form
\begin{align} \label{remark, (rho,pi) form when 1 rho < n_1 lemma, rho_1 leq l < rho_1 + pi_1 form}
H'
= \begin{pmatrix}
H [\rho_1 , \rho_1] & \vline & H [\rho_1 , -(m-\rho_1)] \\
\hline
\mathbf{0} & \vline & 
\begin{NiceMatrix}
0 & \Cdots & \Cdots & \Cdots & \Cdots & 0 & 1 & * & \Cdots & * \\
0 & \Cdots & \Cdots & \Cdots & 0 & 1 & * & \Cdots & \Cdots & * \\
\Vdots &  &  & \Iddots & \Iddots & \Iddots &  &  &  & \Vdots \\
0 & \Cdots & 0 & 1 & * & \Cdots & \Cdots & \Cdots & \Cdots & *
\end{NiceMatrix}
\end{pmatrix}
\end{align}
Note that the $1$s still appear to the right of $H [\rho_1 , \rho_1]$. Given this fact, and the invertibility of $H [\rho_1 , \rho_1]$, we see that $H'$ (and $H$) has full row rank. \\

For the case $m < \rho_1 + \pi_1$, we can again take (\ref{remark, (rho,pi) form when 1 rho < n_1 lemma, l,m geq rho_1 + pi_1 form}), but this time we repeatedly remove a column and add a row. Similar to above, we can see that we will have full column rank. In particular, the kernel will be trivial.
\end{remark}

Remarks \ref{remark, Hankel rho_1 = 0 form, visualisation} and \ref{remark, Hankel rho_1 in 1 , ... , n_1 -1 form, visualisation} effectively establish the following Corollary.

\begin{corollary} \label{corollary, rank of recatnagular Hankel matrix given r of sequence}
Let $\mathbfe{\alpha} \in \mathscr{L}_n (r)$, and let $l+m-2 = n$. Then,
\begin{align*}
\rank H_{l,m} (\mathbfe{\alpha})
= \begin{cases}
r &\text{ if $\min \{ l , m \} \geq r$,} \\
\min \{ l,m \} &\text{ if $\min \{ l , m \} < r$.} 
\end{cases}
\end{align*}
Of course, if we are working with particular values of $\rho (\mathbfe{\alpha})$ and $\pi (\mathbfe{\alpha})$, then we can replace $r$ with $\rho (\mathbfe{\alpha}) + \pi (\mathbfe{\alpha})$.
\end{corollary}

We give a further, final remark in order to demonstrate the usefulness of the $(\rho , \pi )$-form.

\begin{remark}
Lemma \ref{lemma, Hankel rho_1 = 0 form} allows us to easily determine the number of $\mathbfe{\alpha} \in \mathbb{F}_q^{n+1}$ with $\rho (\mathbfe{\alpha}) = 0$ and $\pi (\mathbfe{\alpha}) = \pi_1$. Regarding Lemma \ref{lemma, (rho,pi) form when 1 rho < n_1}, we can easily count the number of possible values that the matrix $H_{l - \rho_1 , m} (\mathbfe{\beta})$ could take. It is not immediately obvious what  the number of values the matrix $H_{\rho_1 , m} (\mathbfe{\alpha}' )$ could take, but we are working with a smaller matrix now, and so this suggests using an inductive argument, which is what we do in Subsection \ref{subsection, matrices of given size and rank}.
\end{remark}

We now proceed to prove Lemma \ref{lemma, (rho,pi) form when 1 rho < n_1}. 

\begin{proof}[Proof of Lemma \ref{lemma, (rho,pi) form when 1 rho < n_1}]
Given that row operations are only applied to rows $\rho_1 +1$ to $n_1$, it is clear that we do indeed have $H_{\rho_1 , m} (\mathbfe{\alpha}' )$ in the top submatrix of (\ref{statement, (rho,pi) form when 1 rho < n_1, Hankel above Hankel form}). \\

If a row operation is applied to an entry $\alpha_i$ that is found on the $i$-th skew diagonal, then it is mapped to $\alpha_i - x_0 \alpha_{i - \rho_1 } - \ldots - x_{\rho_1 -1} \alpha_{i-1}$. This is independent of what position on the $i$-th skew diagonal the entry is found (but, of course, it must be on a row that has a row operation applied to it). It is also independent of what the particular values of $l,m$ are (as long as $l > \rho_1$, as given in the lemma). The former demonstrates that the bottom submatrix of (\ref{statement, (rho,pi) form when 1 rho < n_1, Hankel above Hankel form}) is indeed a Hankel matrix; while the latter demonstrates that $\mathbfe{\beta}$ is independent of the specific values taken by $l,m$. \\

Thus, all that remains to be proven is (\ref{statement, (rho,pi) form when 1 rho < n_1, beta_i conditions}); and, since $\mathbfe{\beta}$ is independent of the specific values taken by $l,m$, it suffices to work with the case $l=n_1$ and $m=n_2$. To this end, after the row operations are applied, we have the matrix
\begin{align*}
H'
= \begin{pmatrix}
\alpha_0 & \alpha_1 & \cdots & \cdots & \alpha_{\rho_1-1} & \vline & \alpha_{\rho_1} & \alpha_{\rho_1 + 1} & \cdots & \cdots & \alpha_{n_2 -1} \\
\alpha_1 &  &  &  & \vdots & \vline & \alpha_{\rho_1 + 1} &  &  &  & \vdots \\
\vdots &  &  &  & \vdots & \vline & \vdots &  &  &  & \vdots \\
\vdots &  &  &  & \alpha_{2\rho_1 - 3} & \vline & \vdots &  &  &  & \alpha_{n_2 + \rho_1 - 1} \\
\alpha_{\rho_1-1}  & \cdots & \cdots & \alpha_{2\rho_1 - 3} & \alpha_{2\rho_1 - 2} & \vline & \alpha_{2 \rho_1-1}  & \cdots & \cdots & \alpha_{n_2 + \rho_1 - 1} & \alpha_{n_2 + \rho_1 - 2} \\ 
\hline
\beta_{\rho_1} & \beta_{\rho_1 + 1} & \cdots & \cdots & \cdots & \cdots & \cdots & \cdots & \cdots & \cdots & \beta_{n_2 + \rho_1 - 1} \\
\beta_{\rho_1 +1} &  &  &  &  &  &  &  &  &  & \vdots \\
\vdots &  &  &  &  &  &  &  &  &  & \vdots \\
\vdots &  &  &  &  &  &  &  &  &  & \beta_{n-1} \\
\beta_{n_1 -1} & \cdots & \cdots & \cdots & \cdots & \cdots & \cdots & \cdots & \cdots & \beta_{n-1} & \beta_{n} 
\end{pmatrix} .
\end{align*}
This is the similar to (\ref{statement, (rho,pi) form when 1 rho < n_1 remark, Hankel over Hankel explicit}), but here we indicate the top-left $\rho_1 \times \rho_1$ submatrix, which is the largest invertible top-left submatrix of $H'$ (by definition of $\rho_1$). \\

We recall that the row operations that we applied do not change the rank of $H$ or any of its top-left submatrices. We now consider the following top-left submatrices of $H'$:
\begin{align*}
H' [\rho_1 +1 | \rho_1 +1 ] , \hspace{1em}
H' [\rho_1 +2 | \rho_1 +2 ] , \hspace{0.5em}
\ldots \hspace{0.5em} , 
H' [n_1 | n_1 ] .
\end{align*}
For the first, we have
\begin{align*}
H' [\rho_1 +1 | \rho_1 +1 ] 
= &\begin{pmatrix}
\alpha_0 & \alpha_1 & \cdots & \cdots & \alpha_{\rho_1-1} & \vline & \alpha_{\rho_1} \\
\alpha_1 &  &  &  & \vdots & \vline & \alpha_{\rho_1 + 1} \\
\vdots &  &  &  & \vdots & \vline & \vdots \\
\vdots &  &  &  & \alpha_{2\rho_1 - 3} & \vline & \vdots \\
\alpha_{\rho_1-1}  & \cdots & \cdots & \alpha_{2\rho_1 - 3} & \alpha_{2\rho_1 - 2} & \vline & \alpha_{2 \rho_1-1} \\ 
\hline
\beta_{\rho_1} & \cdots & \cdots & \beta_{2 \rho_1-2} & \beta_{2 \rho_1-1} & \vline & \beta_{2 \rho_1} 
\end{pmatrix} \\
& \vspace{0.5em} \\
= &\begin{pmatrix}
\alpha_0 & \alpha_1 & \cdots & \cdots & \alpha_{\rho_1-1} & \vline & \alpha_{\rho_1} \\
\alpha_1 &  &  &  & \vdots & \vline & \alpha_{\rho_1 + 1} \\
\vdots &  &  &  & \vdots & \vline & \vdots \\
\vdots &  &  &  & \alpha_{2\rho_1 - 3} & \vline & \vdots \\
\alpha_{\rho_1-1} & \cdots & \cdots & \alpha_{2\rho_1 - 3} & \alpha_{2\rho_1 - 2} & \vline & \alpha_{2 \rho_1-1} \\ 
\hline
0 & \cdots & \cdots & \cdots & 0 & \vline & \beta_{2 \rho_1} 
\end{pmatrix} ,
\end{align*}
where the second equality follows from the definition of $\mathbf{x}$ and the row operation that we applied to row $\rho_1 +1$. Now, by the definition of $\rho_1$, we have that $\det H' [\rho_1 +1 | \rho_1 +1 ]  = 0$; while, by the form of $H' [\rho_1 +1 | \rho_1 +1 ] $ above, we can see that $\det H' [\rho_1 +1 | \rho_1 +1 ] = \beta_{2 \rho_1} \cdot \det H [\rho_1 | \rho_1 ] $. Given that $\det H [\rho_1 | \rho_1 ] \neq 0$ (by definition of $\rho_1$), we must have that $\beta_{2 \rho_1} = 0$. \\

Now consider $H' [\rho_1 +2 | \rho_1 +2 ]$. We have
\begin{align*}
H' [\rho_1 +2 | \rho_1 +2 ] 
= &\begin{pmatrix}
\alpha_0 & \alpha_1 & \cdots & \cdots & \alpha_{\rho_1-1} & \vline & \alpha_{\rho_1} & \alpha_{\rho_1 + 1} \\
\alpha_1 &  &  &  & \vdots & \vline & \alpha_{\rho_1 + 1} & \alpha_{\rho_1 + 2} \\
\vdots &  &  &  & \vdots & \vline & \vdots & \vdots \\
\vdots &  &  &  & \alpha_{2\rho_1 - 3} & \vline & \vdots & \vdots \\
\alpha_{\rho_1-1}  & \cdots & \cdots & \alpha_{2\rho_1 - 3} & \alpha_{2\rho_1 - 2} & \vline & \alpha_{2 \rho_1-1} & \alpha_{2 \rho_1} \\ 
\hline
0 & \cdots & \cdots & \cdots & 0 & \vline & 0 & \beta_{2 \rho_1 +1} \\
0 & \cdots & \cdots & \cdots & 0 & \vline & \beta_{2 \rho_1 +1} & \beta_{2 \rho_1 +2}
\end{pmatrix} .
\end{align*}
By similar reasoning as above, we have
\begin{align*}
0
= \det H' [\rho_1 +2 | \rho_1 +2 ] 
=  \det \begin{pmatrix} 0 & \beta_{2 \rho_1 +1} \\ \beta_{2 \rho_1 +1} & \beta_{2 \rho_1 +2} \end{pmatrix}
	\hspace{0.5em} \cdot \hspace{0.5em} \det H [\rho_1 | \rho_1 ]
= - {\beta_{2 \rho_1 +1}}^2 \cdot \det H [\rho_1 | \rho_1 ] .
\end{align*}
Given that $\det H [\rho_1 | \rho_1 ] \neq 0$, we must have that $\beta_{2 \rho_1 +1} = 0$. \\

Proceeding as above in an inductive manner, we see that $\beta_{\rho_1} , \beta_{\rho_1 +1} , \ldots , \beta_{n_1 + \rho_1 -1} = 0$. That is,
\begin{align} 
\begin{split} \label{H'[n_1 , n_1 ] after using rho(alpha)=rho_1}
&H'[n_1 , n_1 ] \\
&= \begin{pmatrix}
\alpha_0 & \alpha_1 & \cdots & \cdots & \alpha_{\rho_1-1} & \vline & \alpha_{\rho_1} & \alpha_{\rho_1 + 1} & \cdots & \cdots & \alpha_{n_1 -1} \\
\alpha_1 &  &  &  & \vdots & \vline & \alpha_{\rho_1 + 1} &  &  &  & \vdots \\
\vdots &  &  &  & \vdots & \vline & \vdots &  &  &  & \vdots \\
\vdots &  &  &  & \alpha_{2\rho_1 - 3} & \vline & \vdots &  &  &  & \alpha_{n_1 + \rho_1 - 1} \\
\alpha_{\rho_1-1}  & \cdots & \cdots & \alpha_{2\rho_1 - 3} & \alpha_{2\rho_1 - 2} & \vline & \alpha_{2 \rho_1-1}  & \cdots & \cdots & \alpha_{n_1 + \rho_1 - 1} & \alpha_{n_1 + \rho_1 - 2} \\ 
\hline
0 & \cdots & \cdots & \cdots & 0 & \vline & 0 & \cdots & \cdots & \cdots & 0 \\
0 & \cdots & \cdots & \cdots & 0 & \vline & 0 & \cdots & \cdots & 0 & \beta_{n_1 + \rho_1}  \\
0 & \cdots & \cdots & \cdots & 0 & \vline & 0 & \cdots & 0 & \beta_{n_1 + \rho_1} & \beta_{n_1 + \rho_1 +1} \\
\vdots &  &  &  & \vdots & \vline & \vdots & \udots & \udots & \udots & \vdots \\
0 & \cdots & \cdots & \cdots & 0 & \vline & 0 & \beta_{n_1 + \rho_1} & \beta_{n_1 + \rho_1 +1} & \cdots & \beta_{2n_1 -2} 
\end{pmatrix} .
\end{split}
\end{align}
Now, if $n$ is even, then $n_2 = n_1$ and $H'= H'[n_1 , n_1 ]$, whereas if $n$ is odd, then $n_2 = n_1 +1$ and we have an additional column:
\begin{align} 
\begin{split} \label{H' for odd n after using rho(alpha)=rho_1}
&H' \\
&= \begin{pmatrix}
\alpha_0 & \alpha_1 & \cdots & \cdots & \alpha_{\rho_1-1} & \vline & \alpha_{\rho_1} & \alpha_{\rho_1 + 1} & \cdots & \cdots & \alpha_{n_1 -1} & \alpha_{n_1}\\
\alpha_1 &  &  &  & \vdots & \vline & \alpha_{\rho_1 + 1} &  &  &  &  & \vdots \\
\vdots &  &  &  & \vdots & \vline & \vdots &  &  &  &  & \vdots \\
\vdots &  &  &  & \alpha_{2\rho_1 - 3} & \vline & \vdots &  &  &  & \alpha_{n_1 + \rho_1 - 1} & \alpha_{n_1 + \rho_1 -2} \\
\alpha_{\rho_1-1}  & \cdots & \cdots & \alpha_{2\rho_1 - 3} & \alpha_{2\rho_1 - 2} & \vline & \alpha_{2 \rho_1-1}  & \cdots & \cdots & \alpha_{n_1 + \rho_1 - 1} & \alpha_{n_1 + \rho_1 - 2} & \alpha_{n_1 + \rho_1 - 1} \\ 
\hline
0 & \cdots & \cdots & \cdots & 0 & \vline & 0 & \cdots & \cdots & \cdots & 0 & \beta_{n_1 + \rho_1} \\
0 & \cdots & \cdots & \cdots & 0 & \vline & 0 & \cdots & \cdots & 0 & \beta_{n_1 + \rho_1} & \beta_{n_1 + \rho_1 +1}\\
0 & \cdots & \cdots & \cdots & 0 & \vline & 0 & \cdots & 0 & \beta_{n_1 + \rho_1} & \beta_{n_1 + \rho_1 +1} & \beta_{n_1 + \rho_1 +2} \\
\vdots &  &  &  & \vdots & \vline & \vdots & \udots & \udots & \udots & \vdots & \vdots \\
0 & \cdots & \cdots & \cdots & 0 & \vline & 0 & \beta_{n_1 + \rho_1} & \beta_{n_1 + \rho_1 +1} & \cdots & \beta_{2n_1 -2} & \beta_{n_1 + n_2 -2} 
\end{pmatrix} .
\end{split}
\end{align}
In either case, in the last $n_1 - \rho_1$ rows, all the entries are zero except for the last $n_2 - \rho_1 -1$ skew diagonals which may or may not be zero. We now consider the first such skew-diagonal that is non-zero: Suppose $i$ is the largest element in the set $ \{ 1 , 2 , \ldots , (n+1) - (n_1 + \rho_1 ) \}$ that satisfies $\beta_{(n+1)-i} \neq 0$ (if no such $i$ exists, then we take $i=0$, and only a slight adaptation of the following reasoning is required). Then, the bottom-right quadrant of $H'$ (bounded by the vertical and horizontal lines in (\ref{H'[n_1 , n_1 ] after using rho(alpha)=rho_1}) and (\ref{H' for odd n after using rho(alpha)=rho_1})) is of the form
\begin{align*}
\begin{pNiceMatrix}
0 & \Cdots & \Cdots & \Cdots & \Cdots & 0 \\
\Vdots &  &  &  &  & \Vdots \\
0 & \Cdots & \Cdots & \Cdots & \Cdots & 0 \\
0 & \Cdots & \Cdots & \Cdots & 0 & \beta_{(n+1)-i} \\
\Vdots &  &  & \Iddots & \Iddots & \Vdots \\
0 & \Cdots & 0 & \beta_{(n+1)-i} & \ldots & \beta_n
\end{pNiceMatrix} .
\end{align*}
Given that $\beta_{(n+1)-i} \neq 0$, we can see that the last $i$ columns appearing in this quadrant are linearly independent, and that the rank of this quadrant (matrix) is $i$. Recall also that $H' [\rho_1 , \rho_1 ]$ has full column rank. So, given the form of $H'$, we can see that the first $\rho_1$ columns and the last $i$ columns of $H'$ form a basis for its column space. In particular, 
\begin{align*}
\rank H'
= \rho_1 + i .
\end{align*}
Since $\rank H' = \rank H = \rho_1 + \pi_1$, we have $i = \pi_1$. This proves (\ref{statement, (rho,pi) form when 1 rho < n_1, beta_i conditions}) as required.
\end{proof}

Lemmas \ref{lemma, Hankel rho_1 = 0 form} and \ref{lemma, (rho,pi) form when 1 rho < n_1} extend upon \cite[Section 5]{Garcia-ArmasGhorpadeRam2011_RelativePrimePolyNonsingHankelMatrFinField}, where they prove similar lemmas but for square Hankel matrices only. They use column operations instead of row operations. We chose the latter in order to preserve the kernel. We will now formally give the definition of $(\rho , \pi )$-form.

\begin{definition}[The $(\rho , \pi )$-form]
Let $\mathbfe{\alpha} \in \mathbb{F}_q^{n+1}$ and let $l+m-2 = n$. Consider $H := H_{l,m} (\mathbfe{\alpha})$. If $\rho (\mathbfe{\alpha}) = 0$, then we define the $(\rho , \pi )$-form of $H$ to be itself. \\

If $\rho (\mathbfe{\alpha}) = \rho_1 \in \{ 1 , 2 , \ldots , n_1 -1 \}$ and $l \leq \rho_1$, then we also define the $(\rho , \pi )$-form of $H$ to be itself. Whereas, if $l > \rho_1$, then we define the $(\rho , \pi )$-form of $H$ to be the matrix that we obtain after applying the row operations 
\begin{align*}
R_i \longrightarrow R_i - (x_0 , \ldots , x_{\rho_1 -1}) \begin{pmatrix} R_{i - \rho_1 } \\ \vdots \\ R_{i-1} \end{pmatrix} 
	= R_i - x_0 R_{i - \rho_1 } - \ldots - x_{\rho_1 -1} R_{i-1} 
\end{align*}
for $i = n_1 , n_1 -1 , \ldots , \rho_1 +1$ in that order; where $R_i$ is the $i$-th row of $H$, and $\mathbf{x} = (x_0 , \ldots , x_{\rho_1 -1})^T$ is the vector that satisfies
\begin{align*}
H [\rho_1 , \rho_1 ] \mathbf{x}
= \begin{pmatrix}
\alpha_{\rho_1} \\ \alpha_{\rho_1 +1} \\ \vdots \\ \alpha_{2\rho_1 - 1}
\end{pmatrix}. \\
\end{align*}

If $\rho (\mathbfe{\alpha}) = n_1$, then we define the $(\rho , \pi )$-form of $H$ to be itself. 
\end{definition}

\subsection{Matrices of a Given Size and Rank} \label{subsection, matrices of given size and rank}

In \cite[Section 5]{Garcia-ArmasGhorpadeRam2011_RelativePrimePolyNonsingHankelMatrFinField}, they determine the number of square Hankel matrices of a given size, rank, and $(\rho , \pi )$-form. Our results in this section generalise upon this by determining the size of sets of the form $\mathscr{L}_n^h (r , \rho_1 , \pi_1)$, $\mathscr{L}_n^h (r)$, and $\mathscr{H}_{l,m}^h (r)$. That is, we consider rectangular (not just square) Hankel matrices, the associated sequences, and the condition on the number of zeros at the start of those sequences.

\begin{theorem} \label{theorem, matrices of given rank and size}
Let $n \geq 0$ and $0 \leq h \leq n+1$, and consider $\mathscr{L}_{n}^{h} (r , \rho_1 , \pi_1 )$. \\

\textbf{Claim 1:} Suppose $\rho_1 = 0$. By Remarks \ref{remark, bounds for rho_1 and pi_1} and \ref{remark, bounds on h parameter for Hankel matrices}, in order for $\mathscr{L}_{n}^{h} (r , \rho_1 , \pi_1 ) = \mathscr{L}_{n}^{h} (r , 0 , r )$ to be non-empty, we require that
\begin{align*}
0 \leq r \leq n_1 -1 &\text{ if $n$ is even} , \\
0 \leq r \leq n_1 &\text{ if $n$ is odd} ;
\end{align*}
and
\begin{align*}
r \leq n-h+1. 
\end{align*}
Assuming these conditions are satisfied, we have
\begin{align*}
\lvert \mathscr{L}_{n}^{h} (r , 0 , r ) \rvert
= \begin{cases}
1 &\text{ if $r=0$,} \\
(q-1) q^{r-1} &\text{ if $r > 0$.} 
\end{cases} 
\end{align*}

\textbf{Claim 2:} Suppose that $\rho_1 \in \{ 1 , 2 , \ldots , n_1 - 1 \}$. By Remarks \ref{remark, bounds for rho_1 and pi_1} and \ref{remark, bounds on h parameter for Hankel matrices}, in order for $\mathscr{L}_{n}^{h} (r , \rho_1 , \pi_1 ) = \mathscr{L}_{n}^{h} (\rho_1 + \pi_1 , \rho_1 , \pi_1 )$ to be non-empty, we require that
\begin{align*}
0 \leq \pi_1 \leq n_1 - \rho_1 - 1 &\text{ if $n$ is even} , \\
0 \leq \pi_1 \leq n_1 - \rho_1 &\text{ if $n$ is odd} ;
\end{align*}
and
\begin{align*}
\rho_1 \geq h + 1 .
\end{align*} 
Assuming these conditions are satisfied, we have
\begin{align*}
\lvert \mathscr{L}_{n}^{h} (\rho_1 + \pi_1 , \rho_1 , \pi_1 ) \rvert
= \begin{cases}
(q-1) q^{2 \rho_1 - h - 1} &\text{ if $\pi_1 = 0$,} \\
(q-1)^2 q^{2 \rho_1 + \pi_1 - h -2}  &\text{ if $\pi_1 > 0$.} 
\end{cases} \\
\end{align*}

\textbf{Claim 3:} Now suppose that $\rho_1 = n_1$ (this implies $r=n_1$ and it requires $h+1 \leq n_1$). Then, 
\begin{align*}
\lvert \mathscr{L}_{n}^{h} (r , \rho_1 , \pi_1 ) \rvert
= \lvert \mathscr{L}_{n}^{h} (n_1 , n_1 , 0 ) \rvert
= (q-1) q^{n-h} \\
\end{align*}

\textbf{Claim 4:} Consider $\mathscr{L}_n^h (r )$. We have
\begin{align*}
\lvert \mathscr{L}_n^h (r ) \rvert
= \begin{cases}
1 &\text{ if $r=0$,} \\
(q-1) q^{r-1} &\text{ if $1 \leq r \leq \min \{ h , n-h+1 \} $,} \\
(q^2 -1) q^{2r-h-2} &\text{ if $h+1 \leq r \leq n_1 - 1$,} \\ 
q^{n-h+1} - q^{2n_1 - h -2} &\text{ if $r = n_1$ (which is only possible if $h+1 \leq n_1$).}
\end{cases}
\end{align*}
This accounts for all possible values of $r$ that allow $\mathscr{L}_n^h (r )$ to be non-empty. This is clear if $h \leq n-h+1$ (and so $\min \{ h , n-h+1 \} = h$). If instead $h > n-h+1$, then $h \geq \frac{n+2}{2} \geq n_1$, and so any $\mathbfe{\alpha} \in \mathscr{L}_n^h (r )$ will have $H_{n_1 , n_1 } (\mathbfe{\alpha})$ being strictly lower skew-triangular and thus $\rho (\mathbfe{\alpha}) = 0$; as stated in Claim 1, this requires $r \leq n-h+1$. \\

\textbf{Claim 5:} Let $l+m-2 = n$. If $r < \min \{ l , m \}$, then
\begin{align*}
\lvert \mathscr{H}_{l,m}^h (r) \rvert
= \begin{cases}
1 &\text{ if $r=0$,} \\
(q-1) q^{r-1} &\text{ if $1 \leq r \leq \min \{ h , n-h+1 \} $,} \\
(q^2 -1) q^{2r-h-2} &\text{ if $h+1 \leq r \leq n_1 - 1$.}
\end{cases}
\end{align*}
If $r = \min \{ l , m \}$, then
\begin{align*}
\lvert \mathscr{H}_{l,m}^h (r) \rvert
= &\big\lvert \mathscr{H}_{l,m}^h \big( \min \{ l , m \} \big) \big\rvert \\
= &\begin{cases}
q^{l+m-h-1} - q^{\min \{ l , m \} -1} &\text{ if $\min \{ l , m \} - 1 \leq \min \{ h , n-h+1 \}$,} \\
q^{l+m-h-1} - q^{2 \min \{ l , m \} -h-2} &\text{ if $\min \{ l , m \} - 1 \geq h+1$.}
\end{cases}
\end{align*}
Again, this accounts for all possible values of $r$ that allow $\mathscr{H}_{l,m}^h (r)$ to be non-empty.
\end{theorem}

\begin{proof}
\textbf{Claim 1:} When $r=0$, the only element in $\mathscr{L}_{n}^{h} (r , 0 , r )$ is a sequence of zeros. Thus, $\lvert \mathscr{L}_{n}^{h} (r , 0 , r ) \rvert = 1$. Suppose $r > 0$. Lemma \ref{lemma, Hankel rho_1 = 0 form} tells us $\mathbfe{\alpha} \in \mathscr{L}_{n}^{h} (r , 0 , r )$ if and only if
\begin{align*}
\mathbfe{\alpha}
= (0 , \ldots , 0, \alpha_{(n+1)-r } , \ldots , \alpha_n ) 
\end{align*}
with $\alpha_{(n+1)-r} \in \mathbb{F}_q^*$ and $\alpha_{(n+1)-r +1} , \ldots , \alpha_n \in \mathbb{F}_q$. Thus, we have $\lvert \mathscr{L}_{n}^{h} (r , 0 , r ) \rvert = (q-1) q^{r-1}$. \\

\textbf{Claim 2:} For $\pi_1 \geq 1$, there is a bijection between $\mathscr{L}_{n}^{h} (\rho_1 + \pi_1 , \rho_1 , \pi_1 )$ and 
\begin{align*}
\mathscr{L}_{2 \rho_1 -2}^{h} (\rho_1 )
\times \mathbb{F}_q
\times \{ 0 \}^{n - \pi_1 - 2\rho_1 +1}
\times \mathbb{F}_q^*
\times \mathbb{F}_q^{\pi_1 - 1} .
\end{align*}
Indeed, suppose $\mathbfe{\alpha} = (\alpha_0 , \alpha_1 , \ldots , \alpha_n ) \in \mathscr{L}_{n}^{h} (\rho_1 + \pi_1 , \rho_1 , \pi_1 )$ and consider $H := H_{n_1 , n_2} (\mathbfe{\alpha})$. Lemma \ref{lemma, (rho,pi) form when 1 rho < n_1} gives us the following information.
\begin{itemize}
\item The submatrix $H [\rho_1 , \rho_1]$ is invertible. In particular, 
\begin{align*}
\mathbf{\alpha}'
:= (\alpha_0 , \alpha_1 , \ldots , \alpha_{2\rho_1 -2})
\in \mathscr{L}_{2 \rho_1 -2}^{h} (\rho_1 , \rho_1 , 0 )
= \mathscr{L}_{2 \rho_1 -2}^{h} (\rho_1 ) . 
\end{align*}

\item The entry $\alpha_{2 \rho_1 -1}$ is free to take any value in $\mathbb{F}_q$. Note that the vector $\mathbf{x}$ is uniquely determined by $\mathbfe{\alpha}'$ and $\alpha_{2 \rho_1 -1}$. 

\item The entries $\beta_{2 \rho_1} , \ldots , \beta_{n - \pi_1}$, of which there are $n - \pi_1 - 2\rho_1 +1$ number of them, must all take the value $0$, and the invertibility of the row operations (which are uniquely determine by $\mathbf{x}$), means that the corresponding $\alpha_{2 \rho_1} , \ldots , \alpha_{n - \pi_1}$ can also only take a single value. 

\item Similarly, $\beta_{(n+1) - \pi_1}$ can take any value in $\mathbb{F}_q^*$, and so the corresponding $\alpha_{(n+1) - \pi_1}$ can take $(q-1)$ possible values. 

\item Similarly again, $\beta_{(n+2) - \pi_1} , \ldots , \beta_{n}$, of which there are $\pi_1 - 1$ number of them, can take any value in $\mathbb{F}_q$, and so the corresponding $\alpha_{(n+2) - \pi_1} , \ldots , \alpha_{n}$ can each take any value in $\mathbb{F}_q$.
\end{itemize}

By similar reasoning, when $\pi_1 = 0$ we have a bijection between $\mathscr{L}_{n}^{h} (\rho_1 + \pi_1 , \rho_1 , \pi_1 ) = \break \mathscr{L}_{n}^{h} (\rho_1 , \rho_1 , 0 )$ and 
\begin{align*}
\mathscr{L}_{2 \rho_1 -2}^{h} (\rho_1 )
\times \mathbb{F}_q
\times \{ 0 \}^{n - 2\rho_1 +1} .
\end{align*}
So, we have 
\begin{align} \label{statement, L_n^h (rho + pi , rho , pi ) in terms of L_(2 rho -2)^h (rho)}
\lvert \mathscr{L}_{n}^{h} (\rho_1 + \pi_1 , \rho_1 , \pi_1 ) \rvert
= \begin{cases}
\lvert \mathscr{L}_{2 \rho_1 -2}^{h} (\rho_1 ) \rvert \cdot (q-1) q^{\pi_1 } &\text{ if $\pi_1 \geq 1$,} \\
\lvert \mathscr{L}_{2 \rho_1 -2}^{h} (\rho_1 ) \rvert \cdot q &\text{ if $\pi_1 = 0$.} 
\end{cases}
\end{align}

Therefore, what we must understand are the sets $\mathscr{L}_{2 k -2}^{h} (k)$. We have that
\begin{align}
\begin{split} \label{statement L_(2k-2)^h (k), breakup condition on rank}
\lvert \mathscr{L}_{2 k -2}^{h} (k) \rvert
= &\lvert \mathscr{L}_{2 k -2}^{h} \rvert 
	- \sum_{i=0}^{k -1} \lvert \mathscr{L}_{2 k -2}^{h} (i) \rvert \\
= &q^{2 k - h -1} - 1
	- \sum_{i=1}^{k -1} \lvert \mathscr{L}_{2 k -2}^{h} (i) \rvert .
\end{split}
\end{align}
Let us now partition the sets $\mathscr{L}_{2 k -2}^{h} (i)$ above according to the $(\rho , \pi)$-form of the sequences they contain. Suppose first that $1 \leq i \leq h$ and let $\mathbfe{\alpha} \in \mathscr{L}_{2 k -2}^{h} (i)$. We must have that $\rho (\mathbfe{\alpha}) = 0$. Indeed, consider the matrix $H := H_{k,k} (\mathbfe{\alpha})$ that is associated to $\mathbfe{\alpha}$. It's rank is $i$, and so we must have that $\rho (\mathbfe{\alpha}) \leq i$. However, the fact that $i \leq h$ means that the following matrices are lower skew-triangular, and thus not invertible:
\begin{align*}
H[1,1] , H[2,2] , \ldots , H[i,i].
\end{align*}
Therefore, $\rho (\mathbfe{\alpha}) \not\in \{ 1 , 2 , \ldots , i \}$, and so we must have $\rho (\mathbfe{\alpha}) = 0$. Note this implies that $\pi (\mathbfe{\alpha}) = i - \rho (\mathbfe{\alpha}) = i$. Hence, by Claim 1, we have
\begin{align} \label{statement L_(2k-2)^h (i), i leq h, breakup condition on (rho , pi)-form}
\lvert \mathscr{L}_{2 k -2}^{h} (i) \rvert
= \lvert \mathscr{L}_{2 k -2}^{h} (i , 0 , i) \rvert
= (q-1) q^{i-1} . 
\end{align}

Now suppose that $h+1 \leq i \leq k-1$, and let $\mathbfe{\alpha} \in \mathscr{L}_{2 k -2}^{h} (i)$. By similar reasoning as above, we must have that $\rho (\mathbfe{\alpha}) = 0$ or $h+1 \leq \rho (\mathbfe{\alpha}) \leq i$. Hence, 
\begin{align}
\begin{split} \label{statement L_(2k-2)^h (i), h+1 leq i leq k-1, breakup condition on (rho , pi)-form}
\lvert \mathscr{L}_{2 k -2}^{h} (i) \rvert
= &\lvert \mathscr{L}_{2 k -2}^{h} (i , 0 , i) \rvert
	+ \sum_{j=h+1}^{i} \lvert \mathscr{L}_{2 k -2}^{h} (i , j , i-j) \rvert \\
= &(q-1) q^{i-1}
	+ \sum_{j=h+1}^{i} \lvert \mathscr{L}_{2 k -2}^{h} (i , j , i-j) \rvert .
	\end{split}
\end{align}
Substituting (\ref{statement L_(2k-2)^h (i), i leq h, breakup condition on (rho , pi)-form}) and (\ref{statement L_(2k-2)^h (i), h+1 leq i leq k-1, breakup condition on (rho , pi)-form}) into (\ref{statement L_(2k-2)^h (k), breakup condition on rank}), we obtain
\begin{align*}
\lvert \mathscr{L}_{2 k -2}^{h} (k ) \rvert
= &q^{2 k - h -1} - q^{k -1}
	- \sum_{i=h+1}^{k -1} \sum_{j=h+1}^{i} \lvert \mathscr{L}_{2 k -2}^{h} (i , j , i-j) \rvert \\
= &q^{2 \rho_1 - h -1} - q^{k -1}
	- \sum_{j=h+1}^{k -1} \sum_{i=j}^{k -1} \lvert \mathscr{L}_{2 k -2}^{h} (i , j , i-j) \rvert \\
= &q^{2 k - h -1} - q^{k -1}
	- q^{k} \sum_{j=h+1}^{k -1} \lvert \mathscr{L}_{2 j -2}^{h} (j) \rvert \cdot q^{-j} ,
\end{align*}
where the last lines applies (\ref{statement, L_n^h (rho + pi , rho , pi ) in terms of L_(2 rho -2)^h (rho)}). This is a recurrence relation. The initial condition is 
\begin{align*}
\lvert \mathscr{L}_{2 (h+1) -2}^{h} (h+1) \rvert
= (q-1) q^h ;
\end{align*}
Indeed, if $\mathbfe{\alpha} \in \mathscr{L}_{2 (h+1) -2}^{h} (h+1)$ then $H_{h+1 , h+1} (\mathbfe{\alpha})$ has all entries above the main skew-diagonal equal to $0$, and so to have rank equal to $h+1$ (i.e. full rank) we must have that the entries in the main skew-diagonal are in $\mathbb{F}_q^*$, while the entries in the last $h$ skew-diagonals are free to take any values in $\mathbb{F}_q$. Now, it can easily be verified that the solution to the recurrence relation is
\begin{align*}
\lvert \mathscr{L}_{2 k -2}^{h} (k ) \rvert
= (q-1) q^{2k-h-2} .
\end{align*}
Substituting this into (\ref{statement, L_n^h (rho + pi , rho , pi ) in terms of L_(2 rho -2)^h (rho)}) proves Case 3. \\

\textbf{Claims 3 and 4:} We begin with Claim 4. If $r=0$, the only element in $\mathscr{L}_n^h (r )$ is the sequence of zeros, thus proving this case. \\

Suppose instead that $1 \leq r \leq \min \{ h , n-h+1 \}$. As described in the theorem, if $\mathbfe{\alpha} \in \mathscr{L}_n^h (r )$ then $\rho (\mathbfe{\alpha}) = 0$, and so this case follows from Claim 1. \\

Now suppose that $h+1 \leq r \leq n_1 - 1$. If $\mathbfe{\alpha} \in \mathscr{L}_n^h (r )$ then $\rho (\mathbfe{\alpha}) = 0$ or $\rho (\mathbfe{\alpha}) \in \{ h+1 , h+2 , \ldots , r \}$. Hence, we have
\begin{align*}
\lvert \mathscr{L}_n^h (r ) \rvert
= & \lvert \mathscr{L}_n^h (r , 0 , r) \rvert
	+ \sum_{\rho_1 = h+1}^{r} \lvert \mathscr{L}_n^h (r , \rho_1 , r - \rho_1 ) \rvert \\
= &(q-1) q^{r-1}
	+ (q-1)^2 \sum_{\rho_1 = h+1}^{r-1} q^{\rho_1 + r - h - 2}
	\hspace{1em} + (q-1) q^{2r-h-1} \\ 
= & (q^2 -1) q^{2r-h-2} ,
\end{align*}
where the second equality uses Claim 2. \\

Finally, suppose that $r=n_1$. Then, 
\begin{align*}
\lvert \mathscr{L}_n^h (r ) \rvert
= &\lvert \mathscr{L}_n^h \rvert
	- \sum_{r=0}^{n_1 -1} \lvert \mathscr{L}_n^h (r ) \rvert \\
= &q^{n-h+1}
	- 1 - \sum_{r=1}^{h} (q-1) q^{r-1} - \sum_{r=h+1}^{n_1 -1} (q^2 -1) q^{2r-h-2} \\
= &q^{n-h+1} - q^{2n_1 - h -2}.
\end{align*}

For Claim 3, if $n$ is even, then we have $\mathscr{L}_n^h (n_1 , n_1 , 0 ) = \mathscr{L}_n^h (n_1 )$. So, by the last case of Claim 4, we have
\begin{align*}
\lvert \mathscr{L}_n^h (n_1 , n_1 , 0 )  \rvert
= q^{n-h+1} - q^{2n_1 - h -2}
= (q-1) q^{n-h} .
\end{align*}
Now suppose $n$ is odd, and let
\begin{align*}
\mathbfe{\alpha} = &(\alpha_0 , \alpha_1 , \ldots , \alpha_n ) \in \mathscr{L}_n^h (n_1 , n_1 , 0 ), \\
\mathbfe{\alpha}' := &(\alpha_0 , \alpha_1 , \ldots , \alpha_{n-1} ),
\end{align*}
 and $H := H_{n_1 , n_2 } (\mathbfe{\alpha})$. Since $\rho ( \mathbfe{\alpha }) = n_1$, we have that $H_{n_1 , n_1 } (\mathbfe{\alpha}') = H[n_1 ; n_1]$ has full rank. Therefore, $\mathbfe{\alpha}' \in \mathscr{L}_{n-1}^h (n_1 , n_1 , 0 ) = \mathscr{L}_{n-1}^h (n_1 )$, of which there are $(q-1) q^{n-h-1}$ possible values it could take (by the first case of Claim 3). Meanwhile, $\alpha_n$ is free to take any value in $\mathbb{F}_q$, of which there are $q$ possibilities. Thus, $\lvert \mathscr{L}_n^h (n_1 , n_1 , 0 ) \rvert = (q-1) q^{n-h}$, as required. \\
 
We remark that the difference between the odd case and the even case is that when $n$ is odd, the matrix $H_{n_1 , n_2} (\mathbfe{\alpha})$ is not quite square; the additional column allows the matrix to have full rank without necessarily having $\rho (\mathbfe{\alpha}) = n_1$. This is why Claim 3 gives the same result as the last case of Claim 4 only when $n$ is even, but not when $n$ is odd. \\
 
\textbf{Claim 5:} If $r < \min \{ l , m \}$, then Corollary \ref{corollary, rank of recatnagular Hankel matrix given r of sequence} implies that there is a bijection between $\mathscr{H}_{l,m}^h (r)$ and $\mathscr{L}_{n}^{h} (r)$. The result then follows by the first three cases of Claim 4. \\

Now suppose $r = \min \{ l , m \}$. If $\min \{ l , m \} - 1 \leq \min \{ h , n-h+1 \}$, then
\begin{align*}
\lvert \mathscr{H}_{l,m}^h (r) \rvert
= &\lvert \mathscr{H}_{l,m}^h \big( \min \{ l , m \} \big) \rvert \\
= &\lvert \mathscr{H}_{l,m}^h \rvert
	- \sum_{i=0}^{\min \{ l , m \} -1} \lvert \mathscr{H}_{l,m}^h (i) \rvert \\
= &q^{l+m-h-1}
	- 1 - (q-1) \sum_{i=1}^{\min \{ l , m \} -1} q^{i-1} \\
= &q^{l+m-h-1} - q^{\min \{ l , m \} -1} ,
\end{align*}
where the third equality uses the first part of Claim 5. Now suppose that $\min \{ l , m \} - 1 \geq h+1$. Note that this gives $h \leq \min \{ l , m \} -2 \leq n_1 -2$, and so $h < n-h+1$. Thus, we have
\begin{align*}
\lvert \mathscr{H}_{l,m}^h (r) \rvert
= &\lvert \mathscr{H}_{l,m}^h \big( \min \{ l , m \} \big) \rvert \\
= &\lvert \mathscr{H}_{l,m}^h \rvert
	- \sum_{i=0}^{\min \{ l , m \} -1} \lvert \mathscr{H}_{l,m}^h (i) \rvert \\
= &q^{l+m-h-1}
	- 1 
	- (q-1) \sum_{i=1}^{h} q^{i-1} 
	\hspace{1em} - (q^2 -1) \sum_{i=h+1}^{\min \{ l , m \} -1} q^{2i-h-2}  \\
= &q^{l+m-h-1} - q^{2 \min \{ l , m \} -h-2}.
\end{align*}
Again, the third equality uses the first part of Claim 5.
\end{proof}

\subsection{Kernel Structure} \label{subsection, kernel structure}

We will now investigate the kernel structure of Hankel matrices. We begin with an extension to Corollary \ref{corollary, rank of recatnagular Hankel matrix given r of sequence}.

\begin{corollary} \label{corollary, dim kernel of recatnagular Hankel matrix given r of sequence}
Suppose $\mathbfe{\alpha} \in \mathscr{L}_n^h (r)$. We have
\begin{align*}
\dim \kernel H_{l,m} (\mathbfe{\alpha})
= \begin{cases}
0 &\text{ if $1 \leq m \leq r$,} \\
\dim \kernel h_{l+1 , m-1} (\mathbfe{\alpha}) + 1 &\text{ if $r < m \leq n+2-r$,} \\
\dim \kernel h_{l+1 , m-1}(\mathbfe{\alpha}) + 2 &\text{ if $n+2-r < m \leq n+1$.} 
\end{cases}
\end{align*}
(For the case $r=0$ we must define $\dim \kernel H_{n+2,0} (\mathbfe{\alpha}) := 0$). Thus,
\begin{align*}
\dim \kernel H_{l,m} (\mathbfe{\alpha})
= \begin{cases}
0 &\text{ if $1 \leq m \leq r$,} \\
m-r &\text{ if $r < m \leq n+2-r$,} \\
2m - n -2 &\text{ if $n+2-r < m \leq n+1$.} 
\end{cases}
\end{align*}
\end{corollary}

\begin{proof}
The first statement follows from the second. The second statement follows directly from Corollary \ref{corollary, rank of recatnagular Hankel matrix given r of sequence} and the fact that the dimension of the kernel of a matrix is just the number of columns subtracted by the rank.
\end{proof}

\begin{remark}
For intuition it is helpful to understand the first result in Corollary \ref{corollary, dim kernel of recatnagular Hankel matrix given r of sequence} by making use of the $(\rho , \pi )$-form. We start with $m=1$, and add a column and remove a row incrementally (while maintaining that we have a Hankel matrix). For simplicity, assume $2 \leq r \leq n_1 -1$ and $\rho (\mathbfe{\alpha}) = \rho_1 \in \{ 1 , 2 , \ldots , r -1 \}$. When $m \leq r$ we have full column rank and hence the kernel is trivial. If $l,m > r$, then (\ref{remark, (rho,pi) form when 1 rho < n_1 lemma, l,m geq rho_1 + pi_1 form}) gives
\begin{align*}
H_{l,m} (\mathbfe{\alpha})
= \begin{pmatrix}
H [\rho_1 , \rho_1] & \vline & H [\rho_1 , -(m-\rho_1)] \\
\hline
\mathbf{0} & \vline & 
\begin{NiceMatrix}
0 & \Cdots & \Cdots & \Cdots & \Cdots & \Cdots & 0 \\
\Vdots &  &  &  &  &  & \Vdots \\
0 & \Cdots & \Cdots & \Cdots & \Cdots & \Cdots & 0 \\
0 & \Cdots & \Cdots & \Cdots & \Cdots & 0 & 1 \\
0 & \Cdots & \Cdots & \Cdots & 0 & 1 & * \\
\Vdots &  &  & \Iddots & \Iddots & \Iddots & \Vdots \\
0 & \Cdots & 0 & 1 & * & \Cdots & * 
\end{NiceMatrix}
\end{pmatrix} ,
\end{align*}
where there are $r - \rho_1$ number of $1$s at the bottom-right (recall $1$ represents an element in $\mathbb{F}_q^*$ while $*$ represents an element in $\mathbb{F}_q$). The rank of this matrix is $r$. Now, adding a column and removing a row maintains this form until $l=r$; that is until $m = n+2-r$. In particular the rank remains the same, but the number of columns increases by $1$ each time; thus, the dimension of the kernel increases by $1$ each time. If we now take $2 \leq l \leq r$, that is $n+2-r \leq m \leq n$, then (\ref{remark, (rho,pi) form when 1 rho < n_1 lemma, rho_1 leq l < rho_1 + pi_1 form}) gives
\begin{align*}
H_{l,m} (\mathbfe{\alpha})
= \begin{pmatrix}
H [\rho_1 , \rho_1] & \vline & H [\rho_1 , -(m-\rho_1)] \\
\hline
\mathbf{0} & \vline & 
\begin{NiceMatrix}
0 & \Cdots & \Cdots & \Cdots & \Cdots & 0 & 1 & * & \Cdots & * \\
0 & \Cdots & \Cdots & \Cdots & 0 & 1 & * & \Cdots & \Cdots & * \\
\Vdots &  &  & \Iddots & \Iddots & \Iddots &  &  &  & \Vdots \\
0 & \Cdots & 0 & 1 & * & \Cdots & \Cdots & \Cdots & \Cdots & *
\end{NiceMatrix}
\end{pmatrix} .
\end{align*}
Removing a row will decrease the rank by $1$. If we also add a column then the effect is to increase the dimension of the kernel by $2$.
\end{remark} 

\begin{definition}[Characteristic Degrees] \label{definition, characteristic degrees of sequence and associated Hankel matrices}
Suppose $\mathbfe{\alpha} \in \mathscr{L}_n^h (r)$. The characteristic degrees of $\mathbfe{\alpha}$ are defined to be $r$ and $n+2-r$. We extend this definition to any Hankel matrix $H_{l,m} (\mathbfe{\alpha})$ associated to $\mathbfe{\alpha}$.
\end{definition}

The characteristic degrees are just the boundaries for the cases in Corollary \ref{corollary, dim kernel of recatnagular Hankel matrix given r of sequence}. Note that we always have $r \leq n+2-r$, with equality occurring if $n$ is even and $r= \frac{n+2}{2}$. Corollary \ref{corollary, dim kernel of recatnagular Hankel matrix given r of sequence} is given in \cite{HeinigRost1984_AlgMethToeplitzMatrOperat} as Proposition 5.4, although it is stated differently and the proof is different. Definition \ref{definition, characteristic degrees of sequence and associated Hankel matrices} is also given in \cite{HeinigRost1984_AlgMethToeplitzMatrOperat} as Definition 5.3. \\

Now, in what follows, it will be necessary to view vectors in $\mathbb{F}_q^{k+1}$ (for any integer $k \geq 0$) as polynomials in $\mathcal{A} := \mathbb{F}_q [T]$. A vector $(v_0 , v_1 , \ldots, v_k)^T$ should be considered the same as the polynomial $v_0 + v_1 T + \ldots v_k T^k$ and vice versa. Clearly, a vector has a unique polynomial associated with it. However, a polynomial does not have a unique vector associated with it. For example, $\mathbf{v}_1 := (v_0 , v_1 , \ldots, v_k)^T$ and $\mathbf{v}_2 := (v_0 , v_1 , \ldots, v_k , 0)^T$ are different vectors but they are associated with the same polynomial. In order to avoid confusion and to ensure everything is well defined, we will make it clear what vector space we are working with, and its dimension will inform us of the number of zeros that should appear at the end of the vector. It should be noted that in \cite{HeinigRost1984_AlgMethToeplitzMatrOperat} the polynomial associated with $\mathbf{v}_2$ is said to have a root at infinity (associated with the $0$ in the last entry of the vector), thus distinguishing it from the polynomial associated with $\mathbf{v}_1$. However, we will not employ this. Finally, it is helpful to keep in mind that a polynomial of degree $k$ has $k+1$ coefficients, and so any vector associated to it must be in at least $(k+1)$-dimensional space. \\

In this subsection we prove the following four theorems and their associated corollaries. The proofs of the theorems are provided at the end of this subsection.

\begin{theorem} \label{theorem, all Hankel matrices have characteristic polynomial kernel}
Let $\mathbfe{\alpha} \in \mathscr{L}_n (r , \rho_1 , \pi_1 )$, where $n > 0$. Denote the characteristic degrees by $c_1 := r$ and $c_2 := n-r+2$. In what follows, $m$ is given and $l$ should be taken such that $l+m-2=n$. \\

There exist coprime polynomials $A_1 , A_2 \in \mathcal{A}$ with
\begin{align*}
\degree A_1 = \rho_1 , \\
\degree A_2 \leq c_2 ,
\end{align*}
such that
\begin{align*}
\kernel H_{l,m} (\mathbfe{\alpha}) 
= \begin{cases}
\{ \mathbf{0} \} &\text{ if $1 \leq m \leq  c_1$,} \\
\Big\{ B_1 A_1 : \substack{B_1 \in \mathcal{A} \\ \degree B_1 \leq m - c_1 - 1} \Big\} \subseteq \mathbb{F}_q^m &\text{ if $c_1 + 1 \leq m \leq c_2$. } \\
\bigg\{ B_1 A_1 + B_2 A_2 : \substack{B_1 , B_2 \in \mathcal{A} \\ \degree B_1 \leq m - c_1 - 1 \\ \degree B_2 \leq m - c_2 - 1} \bigg\} \subseteq \mathbb{F}_q^m &\text{ if $c_2 +1 \leq m \leq n+1$.}
\end{cases}
\end{align*}
If $\rho_1$ is not equal to $r=c_1$, then $\degree A_2$ is necessarily equal to $c_2$. \\

If $r= 0,1$, then this can be simplified to
\begin{align*}
\kernel H_{l,m} (\mathbfe{\alpha}) 
= \begin{cases}
\{ \mathbf{0} \} &\text{ if $1 \leq m \leq  c_1$,} \\
\Big\{ B_1 A_1 : \substack{B_1 \in \mathcal{A} \\ \degree B_1 \leq m - c_1 - 1} \Big\} \subseteq \mathbb{F}_q^m &\text{ if $c_1 + 1 \leq m \leq n+1$, } \\
\end{cases}
\end{align*}
but we still define 
\begin{align} \label{statement, all Hankel matrices have characteristic polynomial kernel therem proof, A_2 def when r leq 1}
A_2
:= \begin{cases}
0 &\text{ if $\mathbfe{\alpha} \in \mathscr{L}_n (0 , 0 , 0 )$ (i.e. $\mathbfe{\alpha}= \mathbf{0}$),} \\
1 &\text{ if $\mathbfe{\alpha} \in \mathscr{L}_n (1 , 1 , 0 )$,} \\
T^{n+1} &\text{ if $\mathbfe{\alpha} \in \mathscr{L}_n (1 , 0 , 1 )$.}
\end{cases}
\end{align}
\end{theorem}

This leads us to the following definition.

\begin{definition}[Characteristic Polynomials] \label{definition, characteristic polynomials}
In Theorem \ref{theorem, all Hankel matrices have characteristic polynomial kernel}, we define the polynomials $A_1 , A_2$ to be the characteristic polynomials of the sequence $\mathbfe{\alpha}$. Of course, when $r=0,1$, the polynomial $A_2$ is not required. However, it is sometimes helpful to define the second characteristic polynomial as is done in the theorem. For example, in Theorem \ref{theorem, kernel structure subsection, char polys of extended sequences} we take an extension $\mathbfe{\alpha}' := (\mathbfe{\alpha} \mid \alpha_{n+1})$ and express the characteristic polynomials of $\mathbfe{\alpha}'$ in terms of the characteristic polynomials of $\mathbfe{\alpha}$, and thus it is natural and easier to have two characteristic polynomials for both sequences. \\

Now, suppose $c_1 \neq c_2$. We can see that $A_1$ is unique up to multiplication by a unit in $\mathbb{F}_q^*$, but unless otherwise stated $A_1$ should be taken to be monic. For $r \geq 2$, the polynomial $A_2$ should be taken to be monic unless otherwise stated. However, even then it is not unique: We can multiply it by a unit in $\mathbb{F}_q^*$ and add $B_2 A_2$ to it, for any $\degree B_2 \leq c_2 - c_1$. Thus, if we state that $A_2$ is the second characteristic polynomial, it is with the understanding that it is generally not unique. Note that all possibilities for $A_2$ are equivalent modulo $A_1$; and in particular if $\rho_1$ is equal to $r=c_1$ (that is, the sequence $\mathbfe{\alpha}$ is quasi-regular), then we can choose $A_2$ to be monic and have degree less than $\degree A_1$. \\

The case when $c_1 = c_2$ occurs when $n$ is even and $c_1 = c_2 = n_1$. In this case we have
\begin{align*}
\kernel H_{n_1 , n_1 } (\mathbfe{\alpha}) = &\{ \mathbf{0} \} , \\
\kernel H_{n_1 -1 , n_1 +1 } (\mathbfe{\alpha}) = &\{ B A + B' A' : B , B' \in \mathbb{F}_q \} ,
\end{align*}
for some $A,A' \in \mathcal{A}$ with at least one having degree equal to $n_1$. As both these polynomials first appear in the same matrix, it is not immediately obvious how to define the first characteristic polynomial and how to define the second. However, this can be addressed in the following manner. We let $A_2$ be the polynomial that is of smaller degree between $A,A'$, and multiplied by an element in $\mathbb{F}_q^*$ so that it is monic; and we let $A_1$ be the polynomial of higher degree, multiplied so that it is monic. If both $A,A'$ have the same degree, then we take $A_2$ to be the smallest monic representative of $A$ modulo $A'$; and we take $A_1$ to be $A'$, again multiplied so that it is monic. There is more than one possibility for the specific values that $A,A'$ can take, with all possibilities spanning $\kernel H_{n_1 -1 , n_1 +1 } (\mathbfe{\alpha})$ as above; but regardless of which possibility we have, the value of $A_2$ is the same. In cases where we do not have $c_1 = c_2 = n_1$, this uniqueness would apply to $A_1$, not $A_2$; however, the definition we have here is consistent with the degree bounds on $A_1 , A_2$ given in the theorem for the other cases.
\end{definition}

It should be noted that the characteristic degrees and characteristic polynomials of a sequence $\mathbfe{\alpha}$ completely determine the kernel structure. However, the characteristic polynomials alone do not, as we will see in Theorem \ref{theorem, kernel structure subsection, char polys of extended sequences} where a sequence $\mathbfe{\alpha} = (\alpha_0 , \alpha_1 , \ldots , \alpha_n )$ and a certain extension $\mathbfe{\alpha}' = (\alpha_0 , \alpha_1 , \ldots , \alpha_n , \alpha_{n+1} )$ can have the same characteristic polynomials (but different characteristic degrees). \\

The following corollary is easily deduced from Theorem \ref{theorem, all Hankel matrices have characteristic polynomial kernel}.

\begin{corollary}
Suppose $\mathbfe{\alpha} \in \mathscr{L}_n (r , \rho_1 , \pi_1 )$ and $H:= H_{l,m} (\mathbfe{\alpha})$ where $l+m-2=n$. We have already established that if $m \leq r$, then the kernel of $H$ is trivial. \\

If $r < m \leq n+2-r$ and $\mathbfe{\alpha}$ is not quasi-regular (that is, $\pi_1 \neq 0$), then there are no vectors in the kernel of $H$ of the form $(v_0 , v_1 , \ldots , v_{m-1} , 1)^T$, for some $v_0 , \ldots , v_{m-1} \in \mathbb{F}_q$; that is, none of the polynomials in the kernel are monic and of degree $m$. Whereas, if $\mathbfe{\alpha}$ is quasi-regular (that is, $\pi_1 = 0$), then exactly $\frac{1}{q}$ of the vectors in the kernel of $H$ of the form $(v_0 , v_1 , \ldots , v_{m-1} , 1)^T$, for some $v_0 , \ldots , v_{m-1} \in \mathbb{F}_q$; that is, $\frac{1}{q}$ of the polynomials in the kernel are monic and of degree $m$. \\

If $n+2-r < r \leq n+1$, regardless of the value of $\pi_1$, exactly $\frac{1}{q}$ of the vectors in the kernel of $H$ of the form $(v_0 , v_1 , \ldots , v_{m-1} , 1)^T$, for some $v_0 , \ldots , v_{m-1} \in \mathbb{F}_q$; that is, $\frac{1}{q}$ of the polynomials in the kernel are monic and of degree $m$.
\end{corollary}

The following can be viewed as a converse to Theorem \ref{theorem, all Hankel matrices have characteristic polynomial kernel}.

\begin{theorem} \label{theorem, all corpime characteristic polynomial have a Hankel matrix}

\textbf{Claim 1:} Suppose we have $A_1 \in \mathcal{A} \backslash \{ 0 \}$ with $\rho_1 := \degree A_1 \leq 1$, and let $n \geq \rho_1$. Then, there exists a sequence $\mathbfe{\alpha} \in \mathscr{L}_n (\rho_1 , \rho_1 , 0)$ with first characteristic polynomial equal to $A_1$. If $\rho_1 = 0$, then there also exists a sequence $\mathbfe{\alpha} \in \mathscr{L}_n (0 , 0 , 1)$ with first characteristic polynomial equal to $A_1$. The second characteristic polynomials will be as in (\ref{statement, all Hankel matrices have characteristic polynomial kernel therem proof, A_2 def when r leq 1}). \\

\textbf{Claim 2:} Suppose we have $A_1 \in \mathcal{A} \backslash \{ 0 \}$ with $\rho_1 := \degree A_1 \leq 1$, and $A_2 \in \mathcal{A}$ with $\degree A_2 \geq \degree A_1 +2$. Also, let $n \geq \degree A_2$ and $r := n+2- \degree A_2$. Then, there exists a sequence $\mathbfe{\alpha} \in \mathscr{L}_n (r , \rho_1 , r - \rho_1 )$ with characteristic polynomials equal to $A_1 , A_2$.\\

\textbf{Claim 3:} Suppose we have coprime $A_1 , A_2 \in \mathcal{A}$ with $r := \degree A_1 \geq 2$, and let $n \geq \max \{ r , \degree A_2 \} \break + r - 2$. Then, there exists a sequence $\mathbfe{\alpha} \in \mathscr{L}_n (r,r,0)$ with characteristic polynomials $A_1 , A_2$. Furthermore, $\mathbfe{\alpha}$ is unique up to multiplication by elements in $\mathbb{F}_q^*$. \\

\textbf{Claim 4:} Suppose we have coprime $A_1 , A_2 \in \mathcal{A}$ with $\degree A_2 > \degree A_1 \geq 2$, and let 
\begin{align*}
\rho_1 := &\degree A_1 \\
\pi_1 := &\degree A_2 - \degree A_1 \\
r := &\rho_1 + \pi_1 \\
n := &\degree A_2 + r -2 .
\end{align*}
Then, there exists a sequence $\mathbfe{\alpha} \in \mathscr{L}_n (r , \rho_1 , \pi_1 )$ with characteristic polynomials $A_1 , A_2$. Furthermore, $\mathbfe{\alpha}$ is unique up to multiplication by elements in $\mathbb{F}_q^*$.
\end{theorem}

This theorem demonstrates the extent to which we can take any coprime polynomials $A_1 , A_2$ and integer $n$ such that there is a sequence $\mathbfe{\alpha} = (\alpha_0 , \alpha_1 , \ldots , \alpha_n )$ with characteristic polynomials equal to $A_1 , A_2$. Claims 1 and 2 address the cases where $\rho_1 \leq 1$. This is not difficult and it is included for completeness. Claim 3 considers the case where $\mathbfe{\alpha}$ is quasi-regular, and we can see by Theorem \ref{theorem, all Hankel matrices have characteristic polynomial kernel} that this allows for the possibility that $\degree A_2 \leq \degree A_2$. On the other hand, if $\mathbfe{\alpha}$ is not quasi-regular, then by Theorem \ref{theorem, all Hankel matrices have characteristic polynomial kernel} we must have $\degree A_2 > \degree A_1$, and this is the case that Claim 4 considers. \\

With regards to the definition of $r$ in Claim 2, this follows from the fact that if $A_2$ is to be the second characteristic polynomial, then we need the second characteristic degree of $\mathbfe{\alpha}$, which is $n+2 - r$, to be equal to $\degree A_2$. With regards to the bounds on $n$, for Claim 3 we note that the characteristic degrees are $r$ and $n+2-r$, and since the latter must be at least as large as the former, we obtain the requirement that $n \geq r+r-2$. Furthermore, by Theorem \ref{theorem, all Hankel matrices have characteristic polynomial kernel}, we must have that $\degree A_2 \leq n+2-r$, and thus $n \geq \degree A_2 + r -2$. For Claim 4, by Theorem \ref{theorem, all Hankel matrices have characteristic polynomial kernel} we must have that $\degree A_2 = n+2-r$; that is, $n = \degree A_2 + r -2$. The values given for $r, \rho_1 , \pi_1$ are also required by Theorem \ref{theorem, all Hankel matrices have characteristic polynomial kernel}.

\begin{theorem} \label {Theorem, Hankel matrices incorporate Euclidean algorithm}
Suppose
\begin{align*}
\mathbfe{\alpha} = (\alpha_0 , \alpha_1 , \ldots , \alpha_n ) \in \mathscr{L}_n (r , r , 0 )
\end{align*}
with $r \geq 2$ (note that $\mathbfe{\alpha}$ is quasi-regular). Let $A_1 , A_2$ be the characteristic polynomials. We necessarily have
\begin{align*}
d_1 := \degree A_1 = r ,
\end{align*}
and we can choose $A_2$ such that
\begin{align*}
d_2 := \degree A_2 < \degree A_1 .
\end{align*}
Now, if $d_2 \geq 1$, then let $A_3$ be the unique polynomial satisfying
\begin{align*}
A_1 = R_2 A_2 + A_3 \hspace{1em} \text{ and } \hspace{1em} d_3 := \degree A_3 < \degree A_2 ,
\end{align*}
for some polynomial $R_2$. \\

\textbf{Case 1:} If $d_2 \geq 2$, then 
\begin{align*}
\mathbfe{\alpha}^{(2)}
:= (\alpha_0 , \alpha_1 , \ldots , \alpha_{d_1 + d_2 -2})
\end{align*}
is in $\mathscr{L}_{d_1 + d_2 -2}  (d_2 , d_2 , 0 )$ and has characteristic polynomials $A_2 , A_3$. \\

\textbf{Case 2:} If $d_2 = 1$, then 
\begin{align*}
\mathbfe{\alpha}^{(2)}
:= (\alpha_0 , \alpha_1 , \ldots , \alpha_{d_1 })
\end{align*}
is in $\mathscr{L}_{d_1 }  (2 , 1 , 1 )$ and has characteristic polynomials $A_2 , A_1$ (note that the order is important). \\

\textbf{Case 3:} If $d_2 = 0$, then 
\begin{align*}
\mathbfe{\alpha}^{(2)}
:= (\alpha_0 , \alpha_1 , \ldots , \alpha_{d_1 })
\end{align*}
is in $\mathscr{L}_{d_1 }  (2 , 0 , 2 )$ and has characteristic polynomials $A_2 , A_1$ (note that the order is important). \\

Furthermore, $\mathbfe{\alpha}$ is the unique sequence in $\mathscr{L}_n (r , r , 0 )$ that has characteristic polynomials $A_1 , A_2$ and gives the above properties for $\mathbfe{\alpha}^{(2)}$. \\

Suppose now that
\begin{align*}
\mathbfe{\alpha} = (\alpha_0 , \alpha_1 , \ldots , \alpha_n ) \in \mathscr{L}_n (r , \rho_1 , \pi_1 ) ,
\end{align*}
where $r \geq 2$ and $\pi_1 \geq 1$; and let $A_1 , A_2$ be the characteristic polynomials. Then
\begin{align*}
\mathbfe{\alpha}^{(1)}
:= (\alpha_0 , \alpha_1 , \ldots , \alpha_{n - \pi_1 })
\end{align*}
is in $\mathscr{L}_{n - \pi_1 }  (\rho_1 , \rho_1 , 0 )$ and has characteristic polynomials $A_1 , A_2$. A similar rusult holds for the cases $r \leq 1$, but the second characteristic polynomial of $\mathbfe{\alpha}^{(1)}$ will not be that of $\mathbfe{\alpha}$, but it will be defined as in Theorem \ref{theorem, all Hankel matrices have characteristic polynomial kernel}. \\
\end{theorem} 

The theorem above demonstrates the manifestation of the Euclidean algorithm in Hankel matrices, and this is made clearer in the corollaries below. The final claim in the theorem is given in order to demonstrate that even if a sequence is not quasi-regular (which is required for the main part of the theorem) a truncation can be taken that is quasi-regular and has the same characteristic polynomials.

\begin{corollary} \label{corollary, Hankel matrices incorporate Euclidean algorithm theorem, full algorithm presented}
Let
\begin{align*}
\mathbfe{\alpha} = (\alpha_0 , \alpha_1 , \ldots , \alpha_n ) \in \mathscr{L}_n (r , r , 0 ) 
\end{align*}
where $r \geq 2$. Let the characteristic polynomials of $\mathbfe{\alpha}$ be $A_1 , A_2$, where we choose $A_2$ such that $\degree A_2 < \degree A_1$ (which is possible since $\mathbfe{\alpha}$ is quasi-regular) and note that degree $A_1 = r$. Define $d_1 := \degree A_1$ and $d_2 := \degree A_2$. Suppose the Euclidean algorithm gives
\begin{align*}
A_1 = &R_2 A_2 + A_3 			&&d_3 := \degree A_3 < \degree A_2 , \\
A_2 = &R_3 A_3 + A_4 			&&d_4 := \degree A_4 < \degree A_3 , \\
&\vdots 					&&\vdots \\
A_t = &R_{t+1} A_{t+1} + A_{t+2} 		&&d_{t+2} := \degree A_{t+2} < \degree A_{t+1} ,
\end{align*}
where $t$ is such that $\degree A_{t} \geq 2 > \degree A_{t+1} \geq 0$. Finally, let $\mathbfe{\alpha}^{(1)} := \mathbfe{\alpha}$; if $t \geq 2$ then let
\begin{align*}
\mathbfe{\alpha}^{(2)} := &(\alpha_0 , \alpha_1 , \ldots , \alpha_{d_1 + d_2 - 2} ) , \\
\mathbfe{\alpha}^{(3)} := &(\alpha_0 , \alpha_1 , \ldots , \alpha_{d_2 + d_3 - 2} ) , \\
&\vdots \\
\mathbfe{\alpha}^{(t)} := &(\alpha_0 , \alpha_1 , \ldots , \alpha_{d_{t-1} + d_{t} - 2} ) ; 
\end{align*}
and for all $t \geq 1$, let
\begin{align*}
\mathbfe{\alpha}^{(t+1)} := &(\alpha_0 , \alpha_1 , \ldots , \alpha_{d_{t}} ) .
\end{align*}
Then,
\begin{align}
\begin{split} \label{statement, Hankel matrices incorporate Euclidean algorithm theorem, full algorithm presented corollary, alpha^(1) ... alpha^(t) results}
&\text{$\mathbfe{\alpha}^{(1)} \in \mathscr{L}_{n} (r , r , 0)$ and has characteristic polynomials $A_1 , A_2$,} \\
&\text{$\mathbfe{\alpha}^{(2)} \in \mathscr{L}_{d_1 + d_2 - 2} (d_2 , d_2 , 0)$ and has characteristic polynomials $A_2 , A_3$,} \\
&\text{$\mathbfe{\alpha}^{(3)} \in \mathscr{L}_{d_2 + d_3 - 2} (d_3 , d_3 , 0)$ and has characteristic polynomials $A_3 , A_4$,} \\
&\vdots \\
&\text{$\mathbfe{\alpha}^{(t)} \in \mathscr{L}_{d_{t-1} + d_t -2} (d_{t} , d_{t} , 0)$ and has characteristic polynomials $A_{t} , A_{t+1}$;}
\end{split}
\end{align}
and
\begin{align*}
\mathbfe{\alpha}^{(t+1)}
\in \begin{cases}
\mathscr{L}_{d_{t}} (2 , 1 , 1) &\text{ if $d_{t+1} =1$,} \\
\mathscr{L}_{d_{t}} (2 , 0 , 2) &\text{ if $d_{t+1} =0$,} 
\end{cases}
\hspace{1em} \text{ and has characteristic polynomials $A_{t+1} , A_{t}$.} 
\end{align*}
Furthermore, for any given $1 \leq i \leq t$, the sequence $\mathbfe{\alpha}^{(i)}$ is the unique extension of $\mathbfe{\alpha}^{(t+1)}$ satisfying the associated conditions in (\ref{statement, Hankel matrices incorporate Euclidean algorithm theorem, full algorithm presented corollary, alpha^(1) ... alpha^(t) results}).
\end{corollary}

\begin{proof}
This follows by successive applications of Theorem \ref{Theorem, Hankel matrices incorporate Euclidean algorithm}.
\end{proof}

\begin{corollary} \label{corollary, h-1 zeros at start of sequence iff degree diff in last Euclidean alg is h}
Suppose $\mathbfe{\alpha} = (\alpha_0 , \alpha_1 , \ldots , \alpha_n ) \in \mathscr{L}_n (r,r,0)$ with $r \geq 2$, and let the characteristic polynomials be $A_1 , A_2$. Note that $\degree A_1 = r = 2$, and we choose $A_2$ such that $\degree A_2 < \degree A_1$. Let $A_1 , A_2 , \ldots , A_s , 1$ be the polynomials we obtain by applying the Euclidean algorithm to $A_1 , A_2$, and let $d_1 , d_2 , \ldots , d_s , 0$ be their respective degrees. Then, there are exactly $h := \degree A_{s} -1$ consecutive zeros at the beginning of the sequence $\mathbfe{\alpha}$. 
\end{corollary}

\begin{proof}
We begin with the case where $\degree A_{s} = 1$. We must show that the first term of $\mathbfe{\alpha}$ is non-zero. Note that since $\degree A_1 \geq 2$, we must have $s \geq 2$. Let us define $\mathbfe{\alpha}^{(s)}$ as in Corollary \ref{corollary, Hankel matrices incorporate Euclidean algorithm theorem, full algorithm presented}. That is, $\mathbfe{\alpha}^{(s)} = (\alpha_0 , \alpha_1 , \ldots , \alpha_{d_{s-1}}) \in \mathscr{L}_{d_{s-1}} (2,1,1)$ and has characteristic polynomials $A_s$ and $A_{s-1}$. In particular, the kernel of the matrix $(\alpha_0 , \alpha_1 , \ldots , \alpha_{d_{s-1}})$ contains the polynomials
\begin{align} \label{statement, h-1 zeros at start of sequence iff degree diff in last Euclidean alg is h corollary proof, deg A_s = 1, A_s polys in kernel }
A_s \hspace{1em} , \hspace{1em} T A_s \hspace{1em} , \hspace{1em} \ldots \hspace{1em} , \hspace{1em} T^{d_{s-1} -2} A_s 
\end{align}
and
\begin{align} \label{statement, h-1 zeros at start of sequence iff degree diff in last Euclidean alg is h corollary proof, deg A_s = 1, A_(s-1) poly in kernel }
A_{s-1}.
\end{align}
Suppose for a contradiction that $\alpha_0 = 0$. Then, (\ref{statement, h-1 zeros at start of sequence iff degree diff in last Euclidean alg is h corollary proof, deg A_s = 1, A_s polys in kernel }) implies that $\alpha_1 , \ldots , \alpha_{d_{s-1} -1} = 0$, and then (\ref{statement, h-1 zeros at start of sequence iff degree diff in last Euclidean alg is h corollary proof, deg A_s = 1, A_(s-1) poly in kernel }) implies also that $\alpha_{d_{s-1}} = 0$. Thus, $\mathbfe{\alpha}^{(s)} = \mathbf{0}$, contradicting that $\mathbfe{\alpha}^{(s)} \in \mathscr{L}_{d_{s-1}} (2,1,1)$. \\

We now consider the case where $\degree A_s \geq 2$, and we define $A_{s+1} =1$. We define $\mathbfe{\alpha}^{(s+1)}$ as in Corollary \ref{corollary, Hankel matrices incorporate Euclidean algorithm theorem, full algorithm presented}. That is, $\mathbfe{\alpha}^{(s+1)} = (\alpha_0 , \alpha_1 , \ldots , \alpha_{d_{s}}) \in \mathscr{L}_{d_{s-1}} (2,1,1)$ and has characteristic polynomials $A_{s+1}$ and $A_{s}$. In particular, the kernel of the matrix $(\alpha_0 , \alpha_1 , \ldots , \alpha_{d_{s}})$ contains the polynomials
\begin{align} \label{statement, h-1 zeros at start of sequence iff degree diff in last Euclidean alg is h corollary proof, deg A_s > 1, A_(s+1) polys in kernel }
A_{s+1} \hspace{1em} , \hspace{1em} T A_{s+1} \hspace{1em} , \hspace{1em} \ldots \hspace{1em} , \hspace{1em} T^{d_{s} -2} A_{s+1} 
\end{align}
and
\begin{align} \label{statement, h-1 zeros at start of sequence iff degree diff in last Euclidean alg is h corollary proof, deg A_s > 1, A_s poly in kernel }
A_{s}.
\end{align}
Since $A_{s+1} = 1$, we deduce from (\ref{statement, h-1 zeros at start of sequence iff degree diff in last Euclidean alg is h corollary proof, deg A_s > 1, A_(s+1) polys in kernel }) that the first $d_s - 1$ entries of $\mathbfe{\alpha}^{(s+1)}$ are $0$. We now need only show that $\alpha_{d_s -1}$ is non-zero, which follows easily by a contradiction argument: If it were zero, then (\ref{statement, h-1 zeros at start of sequence iff degree diff in last Euclidean alg is h corollary proof, deg A_s > 1, A_s poly in kernel }) would imply that $\alpha_{d_s}$ is also zero, meaning $\mathbfe{\alpha}^{(s+1)} = \mathbf{0}$, which contradicts that $\mathbfe{\alpha}^{(s+1)} \in \mathscr{L}_{d_{s-1}} (2,1,1)$.
\end{proof}

The following theorem demonstrates how the characteristic polynomials of a sequence change if we increase the length of the sequence. The intuition behind this result is made clear in the proof.

\begin{theorem} \label{theorem, kernel structure subsection, char polys of extended sequences}
Let 
\begin{align*}
\mathbfe{\alpha} = (\alpha_0 , \alpha_1 , \ldots , \alpha_n ) \in \mathbb{F}_q^{n+1} ,
\end{align*}
and let $c_1 , c_2$ be the characteristic degrees and $A_1 , A_2$ be the characteristic polynomials. Now let
\begin{align*}
\mathbfe{\alpha}' := (\alpha_0 , \alpha_1 , \ldots , \alpha_n , \alpha_{n+1} ) \in \mathbb{F}_q^{n+2}. 
\end{align*}
be an extension of $\mathbfe{\alpha}$, and denote the characteristic degrees by $c_1' , c_2'$ and the characteristic polynomials by $A_1' , A_2'$. In what follows we define $n_1' := \lfloor \frac{(n+1)+2}{2} \rfloor$ and $n_2' := \lfloor \frac{(n+1)+3}{2} \rfloor$. \\

\textbf{Claim 1:} Suppose that $\mathbfe{\alpha} \in \mathscr{L}_n (r,r,0)$ where $0 \leq r \leq n_1 -1$; and so $c_1 =r$ and $c_2 = n+2-r$, and $A_1 \in \mathcal{A}_{r}$ and $A_2 \in \mathcal{A}_{<r}$. \\

There is one value of $\alpha_{n+1}$ such that $\mathbfe{\alpha}' \in \mathscr{L}_{n+1} (r,r,0)$. In which case, we have $c_1' = c_1 = r$ and $c_2' = c_2 +1 = n+3-r$, and $A_1' = A_1$ and $A_2'= A_2$. \\

There are $q-1$ values of $\alpha_{n+1}$ such that $\mathbfe{\alpha}' \in \mathscr{L}_{n+1} (r+1,r,1)$. In which case, we have $c_1' = c_1 +1 = r+1$ and $c_2' = c_2 = n+2-r$, and $A_1' = A_1$ and $A_2' = \beta A_2 + T^{c_2 - c_1} A_1$ for some $\beta \in \mathbb{F}_q^*$. There is a one-to-one correspondence between $\alpha_{n+1}$ and $\beta$. \\

\textbf{Claim 2:} Suppose that $\mathbfe{\alpha} \in \mathscr{L}_n (r, \rho_1, \pi_1 )$ where $\pi_1 \geq 1$ and $0 \leq r \leq n_1 -1$ (and, by definition, $r=\rho_1 + \pi_1$); and so $c_1 = r$ and $c_2 = n+2-r$, and $A_1 \in \mathcal{A}_{\rho_1 }$ and $A_2 \in \mathcal{A}_{c_2}$. \\

For any value that $\alpha_{n+1}$ takes in $\mathbb{F}_q$, we have $\mathbfe{\alpha}' \in \mathscr{L}_{n+1} (r+1, \rho_1, \pi_1 +1 )$. We have $c_1' = c_1 +1 = r+1$ and $c_2' = c_2 = n+2-r$, and $A_1' = A_1$ and $A_2' = \beta T^{c_2 - c_1} A_1 + A_2$ for some $\beta \in \mathbb{F}_q$. There is a one-to-one correspondence between $\alpha_{n+1}$ and $\beta$. \\

\textbf{Claim 3:} Suppose $n$ is even and $\mathbfe{\alpha} \in \mathscr{L}_n (n_1, n_1, 0 )$; and so $c_1' , c_2' = n_1$, and $A_1 \in \mathcal{A}_{n_1}$ and $A_2 \in \mathcal{A}_{<n_1}$. For any value that $\alpha_{n+1}$ takes, we have $\mathbfe{\alpha} \in \mathscr{L}_{n+1} (n_1, n_1, 0 )$. We also have $c_1' = c_1 = n_1$ and $c_2' = c_2 +1 = n+3 - n_1 = n_1 +1$, and $A_1' = \beta A_2 + A_1$ and $A_2' = A_2$. There is a one-to-one correspondence between $\alpha_{n+1}$ and $\beta$. \\

Suppose $n$ is odd and $\mathbfe{\alpha} \in \mathscr{L}_n (n_1, n_1, 0 )$; and so $c_1' = n_1$ and $c_2' = n_1 +1$, and $A_1 \in \mathcal{A}_{n_1}$ and $A_2 \in \mathcal{A}_{<n_1}$. 
\begin{itemize}
\item There is one value of $\alpha_{n+1}$ such that $\mathbfe{\alpha}' \in \mathscr{L}_{n+1} (n_1, n_1, 0 )$; in which case $c_1' = c_1 = n_1$ and $c_2' = c_2 +1 = n_1 +2$, and $A_1' = A_1$ and $A_2' = A_2$. \\

\item There are $q-1$ values of $\alpha_{n+1}$ such that $\mathbfe{\alpha}' \in \mathscr{L}_{n+1} (n_1 +1, n_1 +1, 0 )$; in which case $c_1' = c_1 +1 = n_1 +1$ and $c_2' = c_2 = n_1 +1$, and $A_1' = \beta A_2 + T A_1$ and $A_2' = A_1$. There is a one-to-one correspondence between $\alpha_{n+1}$ and $\beta$. \\
\end{itemize}

Suppose $n$ is odd and $\mathbfe{\alpha} \in \mathscr{L}_n (n_1, \rho_1, \pi_1 )$, where $\pi_1 \geq 1$ and $0 \leq \rho_1 \leq n_1 -1$; and so $c_1' = n_1$ and $c_2' = n_1 +1$, and $A_1 \in \mathcal{A}_{\rho_1}$ and $A_2 \in \mathcal{A}_{n_1 +1}$. For any value that $\alpha_{n+1}$ takes, we have $\mathbfe{\alpha} \in \mathscr{L}_{n+1} (n_1 +1, n_1 +1, 0 )$; and $c_1' , c_2' = n_1 +1$, and $A_1' = \beta T A_1 + A_2$ and $A_2' = A_1$. There is a one-to-one correspondence between $\alpha_{n+1}$ and $\beta$.
\end{theorem}

We now proceed to prove our four theorems, but first we will need the following two lemmas.

\begin{lemma} \label{lemma, if poly in kernel of Hankel matrix then multiples of poly in wider matrix}
Suppose $\mathbfe{\alpha} = (\alpha_0 , \alpha_1 , \ldots , \alpha_n ) \in \mathbb{F}_q^{n+1}$, and that we have integers $l,m,k$ satisfying $l+m-2 = n$ and $l>k \geq 1$. \\

A vector
\begin{align*}
\mathbf{v} = (v_0 , \ldots , v_{m-1})^T
\end{align*}
is in the kernel of $H_{l,m} (\mathbfe{\alpha})$ if and only if the vectors
\begin{align*}
(\mathbf{v} \mid 0) = &(v_0 , \ldots , v_{m-1} , 0)^T , \\
(0 \mid \mathbf{v} ) = &(0 , v_0 , \ldots , v_{m-1})^T
\end{align*}
are in the kernel of $H_{l-1,m+1} (\mathbfe{\alpha})$. This can be extended, and expressed in terms of polynomials, to give the following result: \\

A polynomial $A \in \mathcal{A}$ with $\degree A \leq m-1$ is in the kernel of $H_{l,m} (\mathbfe{\alpha})$ if and only if $Y A$ is in the kernel of $H_{l-k,m+k} (\mathbfe{\alpha})$ for any $\degree Y \leq k$.
\end{lemma}

\begin{proof}
For the forward implication of the first claim, suppose $\mathbf{v}$ is in the kernel of $H_{l,m} (\mathbfe{\alpha})$, and let
\begin{align*}
\mathbfe{\alpha}' := (\alpha_0 , \alpha_1 , \ldots , \alpha_{n-2}) ,
\mathbfe{\alpha}'' := (\alpha_1 , \alpha_2 , \ldots , \alpha_{n-1}) .
\end{align*}
Due to the last entry being zero, we can see that $(\mathbf{v} \mid 0)$ is in the kernel of $H_{l-1,m+1} (\mathbfe{\alpha})$ if and only if $\mathbf{v}$ is in the kernel of $H_{l-1 , m} (\mathbfe{\alpha}')$. The latter is true, because $H_{l-1 , m} (\mathbfe{\alpha}')$ is the matrix we obtain by removing the last row from $H_{l,m} (\mathbfe{\alpha})$. \\

Similarly, $(0 \mid \mathbf{v})$ is in the kernel of $H_{l-1,m+1} (\mathbfe{\alpha})$ if and only if $\mathbf{v}$ is in the kernel of $H_{l-1 , m} (\mathbfe{\alpha}'')$. Again, the latter is true, because $H_{l-1 , m} (\mathbfe{\alpha}'')$ is the matrix we obtain by removing the first row from $H_{l,m} (\mathbfe{\alpha})$. \\

The backward implication of the first claim follows from what we have established above. \\

We now consider the second claim. The first claim tells us that a polynomial $A \in \mathcal{A} \in \mathbb{F}_q [T]$ with $\degree A \leq m-1$ is in the kernel of $H_{l,m} (\mathbfe{\alpha})$ if and only if $A$ and $T A$ are in the kernel of $H_{l-1,m+1} (\mathbfe{\alpha})$. \\

Successive applications of this tell us that this holds if and only if
\begin{align*}
A,  T A, \ldots, T^k A
\end{align*}
are in the kernel of $H_{l-k,m+k} (\mathbfe{\alpha})$ . \\

Using the fact that any polynomial in the kernel remains in the kernel after being multiplied by an element of $\mathbb{F}_q$, we can see that the above holds if and only if $Y \cdot A$ is in the kernel of $H_{l-k,m+k} (\mathbfe{\alpha})$ for any $\degree Y \leq k$.
\end{proof}

A related lemma is the following.

\begin{lemma} \label{lemma, shifted poly in kernel of Hankel matrix is in kernel of removed row matrix}
Let $\mathbfe{\alpha} = (\alpha_0 , \alpha_1 , \ldots , \alpha_n ) \in \mathbb{F}_q^{n+1}$, and define $\mathbfe{\alpha}' := (\alpha_0 , \alpha_1 , \ldots , \alpha_{n-1})$. Also, let $l+m-2 = n$ with $l \geq 2$. A vector
\begin{align*}
\mathbf{v} = (v_0 , \ldots , v_{m-1})^T
\end{align*}
is in the kernel of $H_{l,m} (\mathbfe{\alpha} )$, if and only if the vectors
\begin{align*}
(\mathbf{v} \mid 0) = &(v_0 , \ldots , v_{m-1} , 0)^T , \\
(0 \mid \mathbf{v} ) = &(0 , v_0 , \ldots , v_{m-1})^T
\end{align*}
are in the kernel of $H_{l-1 ,m} (\mathbfe{\alpha}' )$ (which is just the matrix $H_{l,m} (\mathbfe{\alpha} )$ after removing the last row).
\end{lemma}

The proof of this lemma is similar to the proof of Lemma \ref{lemma, if poly in kernel of Hankel matrix then multiples of poly in wider matrix}. \\

We now give the proofs of the four theorems, beginning with Theorem \ref{theorem, all Hankel matrices have characteristic polynomial kernel}.

\begin{proof}[Proof of Theorem \ref{theorem, all Hankel matrices have characteristic polynomial kernel}]
If $r=0$, then $\mathbfe{\alpha} = 0$ and so $\kernel H_{l,m} (\mathbfe{\alpha}) = \mathbb{F}_q^m$ for all $m$. Therefore, we can take any $A_1 \in \mathcal{A}$ with $\degree A_1 = 0$. \\

Now suppose $r=1$ and $\rho_1 = 1$. Then, $\alpha_0 \neq 0$, and so the matrix $H_{n+1,1} (\mathbfe{\alpha})$ has full column rank and its kernel is trivial. The $(\rho , \pi )$-form of $H_{n,2} (\mathbfe{\alpha})$ is
\begin{align*}
\begin{pmatrix}
\alpha_0 & \vline & \alpha_1 \\
\hline
0 &  & 0 \\
0 &  & 0 \\
\vdots &  & \vdots \\
0 &  & 0
\end{pmatrix} .
\end{align*}
So, we can see that $\kernel H_{n,2} (\mathbfe{\alpha}) = \{ \gamma A_1 : \gamma \in \mathbb{F}_q \}$ for some $A_1 \in \mathcal{A}$ with $\degree A_1 = \rho_1 = 1$. Lemma \ref{lemma, if poly in kernel of Hankel matrix then multiples of poly in wider matrix} tells us that
\begin{align*}
\Big\{ B_1 A_1 : \substack{B_1 \in \mathcal{A} \\ \degree B_1 \leq m - 2} \Big\}
\subseteq \kernel H_{l,m} (\mathbfe{\alpha})
\end{align*}
for $2 \leq m \leq n+1$. The dimension of the left side is $m-1$ (recall a polynomial of degree $m-2$ has $m-1$ coefficients); while Corollary \ref{corollary, dim kernel of recatnagular Hankel matrix given r of sequence} tells us that the right side has dimension $m-1$ as well. Therefore, we must have equality, as required. \\

If, instead, we have $r=1$ and $\rho_1 = 0$, then, $\mathbfe{\alpha} = (0 , \ldots , 0 , \alpha_n)$ with $\alpha_n \neq 0$, and so the matrix $H_{n+1,1} (\mathbfe{\alpha})$ has full column rank and its kernel is trivial. We have that
\begin{align*}
H_{n,2} (\mathbfe{\alpha})
= \begin{pmatrix}
0 & 0 \\
0 & 0 \\
\vdots & \vdots \\
0 & 0 \\
0 & \alpha_n 
\end{pmatrix} ,
\end{align*}
and so by similar means as above we have
\begin{align*}
\kernel H_{l,m} (\mathbfe{\alpha})
= \Big\{ B_1 A_1 : \substack{B_1 \in \mathcal{A} \\ \degree B_1 \leq m - 2} \Big\}
\end{align*}
where $\degree A_1 = 0 = \rho_1$, as required. \\

We now consider the case $r \geq 2$. We will work up to $m = c_2 +1$ first, before considering $m > c_2 +1$. \\

We will first address the special subcase when $c_1 = c_2$. This occurs when $n$ is even and $r = \frac{n+2}{2} = n_1$ (which also implies that $\rho_1 = r$). When $1 \leq m \leq c_1$, we have $\kernel H_{l,m} (\mathbfe{\alpha}) = \{ 0 \}$, by Corollary \ref{corollary, dim kernel of recatnagular Hankel matrix given r of sequence}. In this subcase, there are no $m$ satisfying $c_1 + 1 \leq m \leq c_2$. Suppose now that $m = c_2 +1 = n_1 +1$ and $l=c_2 -1=n_1 -1$. Corollary \ref{corollary, dim kernel of recatnagular Hankel matrix given r of sequence} tells us that $\dim \kernel H_{l,m} (\mathbfe{\alpha}) = 2$. Thus, there are polynomials $A_1 , A_2$ (neither being a multiple of the other) such that
\begin{align*}
\kernel H_{l,m} (\mathbfe{\alpha}) 
= \bigg\{ B_1 A_1 + B_2 A_2 : \substack{B_1 , B_2 \in \mathcal{A} \\ \degree B_1 \leq 0 \\ \degree B_2 \leq 0} \bigg\} .
\end{align*}
All that remains to be shown is that at least one of $A_1 , A_2$ have degree equal to $\rho_1 = r$ (without loss of generality, this will be $A_1$). To show this, suppose for a contradiction that $\degree A_1 , \degree A_2 < \rho_1 = r$. Then, the vectors associated to these polynomials are of the form 
\begin{align*}
\mathbf{v} = &(v_0 , v_1 , \ldots , v_{r-1} , 0), \\
\mathbf{w} = &(w_0 , w_1 , \ldots , w_{r-1} , 0) ;
\end{align*}
and so the vectors
\begin{align*}
\mathbf{v}' = &(v_0 , v_1 , \ldots , v_{r-1}) , \\
\mathbf{w}' = &(w_0 , w_1 , \ldots , w_{r-1}) 
\end{align*}
are in the kernel of $H_{l,m-1} (\mathbfe{\alpha}')$, where $\mathbfe{\alpha}' := (\alpha_0 , \alpha_1 , \ldots , \alpha_{n-1})$. In particular, $\break \dim \kernel H_{l,m-1} (\mathbfe{\alpha}') = 2$. From this, and the fact that $H_{l,m-1} (\mathbfe{\alpha}')$ is the matrix we obtain by removing the last row from $H_{l+1 , m-1} (\mathbfe{\alpha}) = H_{n_1 , n_1} (\mathbfe{\alpha})$, we deduce that the kernel of $H_{n_1 , n_1} (\mathbfe{\alpha})$ is at least one-dimensional. This contradicts that the kernel is trivial (since the matrix is invertible). \\

Suppose now that $c_1 \neq c_2$. When $1 \leq m \leq c_1 = r$, Corollary \ref{corollary, dim kernel of recatnagular Hankel matrix given r of sequence} tells us that $\kernel H_{l,m} (\mathbfe{\alpha}) = \{ \mathbf{0} \}$. \\

Now suppose that $m= c_1 +1 = r+1$. Corollary \ref{corollary, dim kernel of recatnagular Hankel matrix given r of sequence} tells us that the kernel of $H : = H_{l,m} (\mathbfe{\alpha})$ has dimension $1$, and so we let $\mathbf{v} \neq \mathbf{0}$ be a vector that spans the kernel. The $(\rho , \pi )$-form of $H$ is 
\begin{align*}
\begin{pmatrix}
H [\rho_1 , \rho_1] & \vline & H [\rho_1 , -(m-\rho_1)] \\
\hline
\mathbf{0} & \vline & 
\begin{NiceMatrix}
0 & \Cdots & \Cdots & \Cdots & 0 \\
\Vdots &  &  &  & \Vdots \\
0 & \Cdots & \Cdots & \Cdots & 0 \\
0 & \Cdots & \Cdots & 0 & 1 \\
0 & \Cdots & 0 & 1 & * \\
\Vdots & \Iddots & \Iddots & \Iddots & \Vdots \\
0 & 1 & * & \Cdots & * 
\end{NiceMatrix}
\end{pmatrix} ,
\end{align*}
where $1$ represents an element in $\mathbb{F}_q^*$ and $*$ represents an element in $\mathbb{F}_q$. The submatrix $H [\rho_1 , \rho_1]$ is invertible, and there are $\pi_1$ number of $1$s in the bottom-right submatrix. Of course, if $\rho_1 = 0$, then the top two submatrices and the bottom-left submatrix should disappear. If $\pi_1 = 0$, then the bottom-right submatrix should be a zero matrix. Regardless, the $(\rho , \pi )$-form shows us that $\mathbf{v}$ must have zeros in its last $\pi_1$ entries. That is,
\begin{align*}
\mathbf{v}
= \begin{pmatrix} v_0 \\ \vdots \\ v_{\rho_1 -1} \\ v_{\rho_1 } \\ 0 \\ \vdots \\ 0 \end{pmatrix} .
\end{align*}
We must also have that $v_{\rho_1 } \neq 0$. When $\rho_1 = 0$, this is clear. When $\rho_1 > 0$, because $\mathbf{v}$ is in the kernel of $H$, we can see that $(v_0 , \ldots , v_{\rho_1 -1} , v_{\rho_1 })^T$ is in the kernel of 
\begin{align*}
\begin{pmatrix}
H [\rho_1 , \rho_1 ] & \vline &
\begin{matrix}
\alpha_{\rho_1} \\
\vdots \\
\alpha_{2 \rho_1 -2} \\
\alpha_{2 \rho_1 -1}
\end{matrix}
\end{pmatrix} ;
\end{align*}
and so, if $v_{\rho_1 } = 0$, then $(v_0 , \ldots , v_{\rho_1 -1})^T \neq \mathbf{0}$ is in the kernel of $H [\rho_1 , \rho_1 ]$, contradicting that it is invertible. \\

In terms of polynomials, we have shown, for $m= c_1 +1 = r+1$, that
\begin{align*}
\kernel H_{l,m} (\mathbfe{\alpha}) 
= \Big\{ B_1 A_1 : \substack{B_1 \in \mathcal{A} \\ \degree B_1 \leq 0} \Big\} .
\end{align*}
As previously, Lemma \ref{lemma, if poly in kernel of Hankel matrix then multiples of poly in wider matrix} and Corollary \ref{corollary, dim kernel of recatnagular Hankel matrix given r of sequence} tell us, for $c_1 +1 \leq m \leq c_2$, that
\begin{align*}
\kernel H_{l,m} (\mathbfe{\alpha}) 
= \Big\{ B_1 A_1 : \substack{B_1 \in \mathcal{A} \\ \degree B_1 \leq m - c_1 -1} \Big\} .
\end{align*}

Now suppose that $m = c_2 +1$. Lemma \ref{lemma, if poly in kernel of Hankel matrix then multiples of poly in wider matrix} tells us that 
\begin{align*}
\Big\{ B_1 A_1 : \substack{B_1 \in \mathcal{A} \\ \degree B_1 \leq m - c_1 -1} \Big\}
\subseteq \kernel H_{l,m} (\mathbfe{\alpha}) .
\end{align*}
However, the left side has dimension $m - c_1 = n+3 - 2r$, while, by Corollary \ref{corollary, dim kernel of recatnagular Hankel matrix given r of sequence}, the right side has dimension $2m-n-2 = n+4 - 2r$. Thus, there is some polynomial $A_2 \in \mathcal{A}$ in the kernel of $H_{l,m} (\mathbfe{\alpha})$ with
\begin{align} \label{statement, Hankel kernel structure lemma, A_2 not in span of A_1}
A_2 \not\in \Big\{ B_1 A_1 : \substack{B_1 \in \mathcal{A} \\ \degree B_1 \leq m - c_1 -1} \Big\} .
\end{align}
Thus, we have 
\begin{align*}
\kernel H_{l,m} (\mathbfe{\alpha})
= \bigg\{ B_1 A_1 + B_2 A_2 : \substack{B_1 , B_2 \in \mathcal{A} \\ \degree B_1 \leq m - c_1 - 1 \\ \degree B_2 \leq 0} \bigg\} .
\end{align*}
The fact that $\degree A_2 \leq c_2$ (that is, $A_2$ has at most its first $c_2 +1$ coefficients being non-zero) follows from the fact that the kernel is in $(c_2 +1)$-dimensional space. \\

We will now show that if $\rho_1$ is not equal to $r=c_1$, then $\degree A_2$ is necessarily equal to $c_2$. Let $m = c_2 +1$, and let $\mathbf{v} = (a_0 , a_1 , \ldots , a_{c_2})^T$ be the vector associated with $A_2$. Note that the condition $\degree A_2 = c_2$ is equivalent to $a_{c_2} \neq 0$. \\

Suppose for a contradiction that $\rho_1 \neq r$ and $\degree A_2 \neq c_2$. This means that $\pi_1 \geq 1$ and $\mathbf{v} = (a_0 , a_1 , \ldots , a_{c_2 -1} , 0)^T$. The latter implies that $\mathbf{v}' := (a_0 , a_1 , \ldots , a_{c_2 -1})^T$ is in the kernel of $H_{l , m-1} (\mathbfe{\alpha}')$, where $\mathbfe{\alpha}' := (\alpha_0 , \alpha_1 , \ldots , \alpha_{n-1})$. In terms of polynomials, this means $A_2 \in H_{l , m-1} (\mathbfe{\alpha}')$. \\

Note that $H_{l , m-1} (\mathbfe{\alpha}')$ is the matrix we obtain by removing the last row from $H_{l+1 , m-1} (\mathbfe{\alpha})$. We have already established that
\begin{align*}
\kernel H_{l+1 , m-1} (\mathbfe{\alpha})
= \kernel H_{c_1 , c_2} (\mathbfe{\alpha})
= \Big\{ B_1 A_1 : \substack{B_1 \in \mathcal{A} \\ \degree B_1 \leq c_2 - c_1 - 1} \Big\} .
\end{align*}
Hence, an application of Lemma \ref{lemma, shifted poly in kernel of Hankel matrix is in kernel of removed row matrix} gives us that
\begin{align*}
\kernel H_{l , m-1} (\mathbfe{\alpha}')
= \Big\{ B_1 A_1 : \substack{B_1 \in \mathcal{A} \\ \degree B_1 \leq c_2 - c_1 } \Big\} .
\end{align*}
(Note that, because $\pi_1 \geq 1$ every vector in the kernel of $H_{l+1 , m-1} (\mathbfe{\alpha})$ has a zero in its last entry, which is a requirement for our application of Lemma \ref{lemma, shifted poly in kernel of Hankel matrix is in kernel of removed row matrix}). However, since we have established $A_2 \in H_{l , m-1} (\mathbfe{\alpha}')$, this implies that 
\begin{align*}
A_2 \in \Big\{ B_1 A_1 : \substack{B_1 \in \mathcal{A} \\ \degree B_1 \leq c_2 - c_1 } \Big\} ,
\end{align*}
which contradicts (\ref{statement, Hankel kernel structure lemma, A_2 not in span of A_1}). \\

Finally, it remains to consider when $m > c_2 + 1$. By Lemma \ref{lemma, if poly in kernel of Hankel matrix then multiples of poly in wider matrix}, we can see that
\begin{align} \label{statement, all Hankel matrices have characteristic polynomial kernel theorem proof, m>c_2 +1 subset result}
\bigg\{ B_1 A_1 + B_2 A_2 : \substack{B_1 , B_2 \in \mathcal{A} \\ \degree B_1 \leq m - c_1 - 1 \\ \degree B_2 \leq m - c_2 - 1} \bigg\} 
\subseteq \kernel H_{l,m} (\mathbfe{\alpha}) .
\end{align} 
We will show that 
\begin{align} \label{statement, all Hankel matrices have characteristic polynomial kernel theorem proof, Y_1 A_1 + Y_2 A_2 neq 0}
B_1 A_1 + B_2 A_2 \neq 0
\end{align}
for all $\degree B_1 \leq n+1 - c_1$ and $\degree B_2 \leq n+1 - c_2$. Thus, the left side of (\ref{statement, all Hankel matrices have characteristic polynomial kernel theorem proof, m>c_2 +1 subset result}) has dimension $2m - c_1 - c_2$, which is equal to the dimension of the right side by Corollary \ref{corollary, dim kernel of recatnagular Hankel matrix given r of sequence}, and thus we have equality. \\

Note that (\ref{statement, all Hankel matrices have characteristic polynomial kernel theorem proof, Y_1 A_1 + Y_2 A_2 neq 0}) also proves that $A_1 , A_2$ are coprime. Indeed, if they were not then we could write
\begin{align*}
A_1 = &C A_1' , \\
A_2 = &C A_2' ,
\end{align*}
where
\begin{align*}
\degree C \geq &1 , \\
\degree A_1' \leq &\degree A_1 -1 = \rho_1 -1 \leq c_1 -1 = n+1 -c_2, \\
\degree A_2' \leq &\degree A_2 -1 = c_2 -1 = n+1 - c_1 .
\end{align*}
In particular, taking $B_1 = A_2'$ and $B_2 = -A_1'$ would contradict (\ref{statement, all Hankel matrices have characteristic polynomial kernel theorem proof, Y_1 A_1 + Y_2 A_2 neq 0}). \\

To prove (\ref{statement, all Hankel matrices have characteristic polynomial kernel theorem proof, Y_1 A_1 + Y_2 A_2 neq 0}), suppose for a contradiction that 
\begin{align*}
B_1 A_1 + B_2 A_2
= 0
\end{align*}
with
\begin{align*}
\degree B_1 \leq &m' - c_1 , \\
\degree B_2 \leq &m' - c_2 ,
\end{align*}
where $c_2 \leq m' \leq n+1$. Suppose further that $m'$ is minimal with this property; in particular, there is equality in at least one the inequalities above. Let us write
\begin{align*}
B_1 = &x_0 + x_1 T + \ldots + x_{m' - c_1} T^{m' - c_1} , \\
B_2 = &y_0 + y_1 T + \ldots + y_{m' - c_2} T^{m' - c_2} ;
\end{align*}
and we note that at least one of $x_{m' - c_1} , y_{m' - c_2}$ is non-zero due to there being at least one equality in the inequalities above. Since $B_1 A_1 + B_2 A_2 = 0$, we have that
\begin{align*}
&( x_{m' - c_1} T^{m' - c_1} A_1 + y_{m' - c_2} T^{m' - c_2} A_2 ) \\
= &- \Big( ( x_0 + x_1 T + \ldots + x_{m' - c_1 -1} T^{m' - c_1 -1}) A_1
	+ (y_0 + y_1 T + \ldots + y_{m' - c_2 -1} T^{m' - c_2 -1}) A_2 \Big) .
\end{align*}
(Note that this is non-zero because at least one of $x_{m' - c_1} , y_{m' - c_2}$ is non-zero, and because $A_2$ is not a multiple of $A_1$). By (\ref{statement, all Hankel matrices have characteristic polynomial kernel theorem proof, m>c_2 +1 subset result}), the right side tells us that
\begin{align*}
T^{m' - c_2} Z
\in H_{l' , m'} (\mathbfe{\alpha}) ,
\end{align*}
where $l'$ is such that $l' + m' = n+2$ and 
\begin{align*}
Z
:= ( x_{m' - c_1} T^{c_2 - c_1} A_1 + y_{m' - c_2} A_2 ) .
\end{align*}
Again by (\ref{statement, all Hankel matrices have characteristic polynomial kernel theorem proof, m>c_2 +1 subset result}), we also have that
\begin{align*}
T^{s} Z
= ( x^{m' - c_1} T^{s + c_2 - c_1} A_1 + y^{m' - c_2} T^{s} A_2 )
\in H_{l' , m'} (\mathbfe{\alpha}) 
\end{align*}
for all $0 \leq s \leq m' - c_2 -1$. But then, using the second result in Lemma \ref{lemma, if poly in kernel of Hankel matrix then multiples of poly in wider matrix} for the first relation below, we have that
\begin{align*}
Z
\in H_{l' + (m' - c_2) , m' - (m' - c_2)} (\mathbfe{\alpha}) 
= H_{c_1 , c_2} (\mathbfe{\alpha}) 
= \Big\{ X A_1 : \substack{X \in \mathcal{A_1} \\ \degree X \leq c_2 - c_1 - 1} \Big\} .
\end{align*}
Recalling the definition of $Z$, we can see that this implies
\begin{align*}
A_2
\in \Big\{ X A_1 : \substack{X \in \mathcal{A} \\ \degree X \leq c_2 - c_1} \Big\} ,
\end{align*}
which contradicts (\ref{statement, Hankel kernel structure lemma, A_2 not in span of A_1}).
\end{proof}

We now prove Theorem \ref{Theorem, Hankel matrices incorporate Euclidean algorithm}, which will be required for the proof of Theorem \ref{theorem, all corpime characteristic polynomial have a Hankel matrix} afterwards. 

\begin{proof}[Proof of Theorem \ref{Theorem, Hankel matrices incorporate Euclidean algorithm}]

For what follows, we define $\mathbfe{\alpha}' := (\alpha_0 , \alpha_1 , \ldots , \alpha_{2 d_1})$. \\

Now, consider the matrix $H_{d_1 -1 , n+3 - d_1 } (\mathbfe{\alpha})$. It has kernel
\begin{align*}
\kernel H_{d_1 -1 , n+3 - d_1 } (\mathbfe{\alpha})
= \bigg\{ B_1 A_1 + B_2 A_2 : \substack{B_1 , B_2 \in \mathcal{A} \\ \degree B_1 \leq n+3 - 2 d_1 \\ \degree B_2 \leq 0} \bigg\} .
\end{align*}
Recall that $\degree A_1 = d_1$ and $\degree A_2 = d_2 < d_1$. In particular, the vectors associated to $A_1 , A_2$ in $H_{d_1 -1 , n+3 - d_1 } (\mathbfe{\alpha})$ have at least $n+2 - 2d_1$ zero entries at the end. Therefore, consider the matrix that we obtain by removing the last $n+2 - 2d_1$ columns of $H_{d_1 -1 , n+3 - d_1 } (\mathbfe{\alpha})$, which is $H_{d_1 -1 , d_1 +1} (\mathbfe{\alpha}')$. From the above we can see that
\begin{align*}
\bigg\{ B_1 A_1 + B_2 A_2 : \substack{B_1 , B_2 \in \mathcal{A} \\ \degree B_1 \leq 0 \\ \degree B_2 \leq 0} \bigg\}
\subseteq \kernel H_{d_1 -1 , d_1 +1} (\mathbfe{\alpha}') .
\end{align*}
Given that $H_{d_1 , d_1 } (\mathbfe{\alpha}') = H_{r , r } (\mathbfe{\alpha}')$ is invertible, we can see that the dimension of $\break \kernel H_{d_1 -1 , d_1 +1} (\mathbfe{\alpha}')$ is $2$ and so we must have equality above:
\begin{align} \label{statement, Theorem Hankel matrices incorporate Euclidean algorithm proof, kernel of H_(d_1 -1 , d_1 +1) (alpha')}
\kernel H_{d_1 -1 , d_1 +1} (\mathbfe{\alpha}')
= \bigg\{ B_1 A_1 + B_2 A_2 : \substack{B_1 , B_2 \in \mathcal{A} \\ \degree B_1 \leq 0 \\ \degree B_2 \leq 0} \bigg\} .
\end{align}
Now, we have $\degree A_2 = d_2 < d_1$, and so the polynomial associated to $A_2$ in $\kernel H_{d_1 -1 , d_1 +1} (\mathbfe{\alpha}')$ has $d_1 - d_2$ zeros entries at the end. What we do next depends on the size of $d_2$. \\

\textbf{Case 1:} If $d_2 \geq 2$, then $d_1 - d_2 < d_1 -1$. Thus, we can remove the last $d_1 - d_2$ columns from $\kernel H_{d_1 -1 , d_1 +1} (\mathbfe{\alpha}')$. This gives the matrix $\kernel H_{d_2 -1 , d_1 +1} (\mathbfe{\alpha}^{(2)} )$, and we can see that it has kernel
\begin{align*}
\kernel H_{d_2 -1 , d_1 +1} (\mathbfe{\alpha}^{(2)} )
= \bigg\{ B_1 A_1 + B_2 A_2 : \substack{B_1 , B_2 \in \mathcal{A} \\ \degree B_1 \leq 0 \\ \degree B_2 \leq d_1 - d_2} \bigg\} .
\end{align*}
The fact that the left side is a subset of the right side follows from (\ref{statement, Theorem Hankel matrices incorporate Euclidean algorithm proof, kernel of H_(d_1 -1 , d_1 +1) (alpha')}) and successive applications of Lemma \ref{lemma, shifted poly in kernel of Hankel matrix is in kernel of removed row matrix}. Equality then follows by noting that removing a row will increase the dimension by (at most) $1$. Also, we can see that $A_3$ is in the kernel of $H_{d_2 -1 , d_1 +1} (\mathbfe{\alpha}^{(2)} )$ and we can replace $A_1$ with $A_3$:
\begin{align*}
\kernel H_{d_2 -1 , d_1 +1} (\mathbfe{\alpha}^{(2)} )
= \bigg\{ B_2 A_2 + B_3 A_3: \substack{B_2 , B_3 \in \mathcal{A} \\ \degree B_2 \leq d_1 - d_2 \\ \degree B_3 \leq 0 } \bigg\} .
\end{align*}
We can now deduce that the characteristic polynomials of $\mathbfe{\alpha}^{(2)}$ are $A_2 , A_3$, and that $\mathbfe{\alpha}^{(2)} \in \mathscr{L}_{d_1 + d_2 -2} (d_2 , d_2 , 0)$. \\

\textbf{Case 2:} If $d_2 = 1$, then $d_1 - d_2 = d_1 -1$. Thus, we can only remove the last $d_1 - d_2 -1 = d_1 -2$ columns from $\kernel H_{d_1 -1 , d_1 +1} (\mathbfe{\alpha}')$. This gives the matrix $\kernel H_{d_2  , d_1 +1} (\mathbfe{\alpha}^{(2)} ) = \kernel H_{1  , d_1 +1} (\mathbfe{\alpha}^{(2)} )$ (recall that the definition of $\mathbfe{\alpha}^{(2)}$ differs between the cases), and we can see by similar means as in Case 1 that it has kernel
\begin{align*}
\kernel H_{1 , d_1 +1} (\mathbfe{\alpha}^{(2)} )
= \bigg\{ B_1 A_1 + B_2 A_2 : \substack{B_1 , B_2 \in \mathcal{A} \\ \degree B_1 \leq 0 \\ \degree B_2 \leq d_1 - d_2 -1} \bigg\} .
\end{align*}
Note that $B_2$ cannot have degree equal to $d_1 - d_2$, and so $A_3$ is not in the kernel above. Thus, we have that the characteristic polynomials are $A_2 , A_1$ (recall that the order matters). The fact that $B_2$ cannot have degree equal to $d_1 - d_2$ also tells us that $\pi (\mathbfe{\alpha}^{(2)} ) = 1$. We deduce that $\mathbfe{\alpha}^{(2)} \in \mathscr{L}_{d_1 } (2 , 1 , 1)$. \\

\textbf{Case 3:} This case is very similar to Case 2. \\

We now prove the uniqueness claim that is made in the theorem. We do so for Case 1, and the remaining cases are almost identical and only need to take into account the difference in definition of $\mathbfe{\alpha}^{(2)}$. To this end, consider the sequences
\begin{align*}
\mathbfe{\alpha}^{(2)} = &(\alpha_0 , \alpha_1 , \ldots , \alpha_{d_1 + d_2 -1}) \\
\mathbfe{\alpha}' = &(\alpha_0 , \alpha_1 , \ldots , \alpha_{d_1 + d_2 -1} , \alpha_{d_1 + d_2 } , \ldots , \alpha_{2 d_1} ) .
\end{align*}
Note that $\alpha_{d_1 + d_2 } , \ldots , \alpha_{2 d_1}$ form the last $d_1 - d_2 +1$ entries in the last column of the matrix $H_{d_1 -1 , d_1 +1} (\mathbfe{\alpha}')$, and by assumption we must have that $A_1$ is in the kernel of this matrix. Thus, for $d_1 + d_2 \leq i \leq 2 d_1$, the $i$-th row of $H_{d_1 -1 , d_1 +1} (\mathbfe{\alpha}')$ is orthogonal to the vector associated with $A_1$. Since this vector has non-zero final entry, we can express the last entry in the $i$-th row in terms of the previous entries. An inductive argument proves that the entries $\alpha_{d_1 + d_2 } , \ldots , \alpha_{2 d_1}$ can be uniquely determined in terms of the entries $\alpha_0 , \alpha_1 , \ldots , \alpha_{d_1 + d_2 -1}$. \\

Now consider the sequence 
\begin{align*}
\mathbfe{\alpha} = &(\alpha_0 , \alpha_1 , \ldots , \alpha_{2 d_1} , \alpha_{2 d_1 +1} , \ldots , \alpha_{n} ) ,
\end{align*}
and note that $\alpha_{2 d_1 +1} , \ldots , \alpha_{n}$ form the last $n - 2 d_1$ entries of the final row of $H_{d_1 -1 , n+3 - d_1 } (\mathbfe{\alpha})$. By the assumptions made in the uniqueness claim, we must have that the polynomials
\begin{align*}
A_1 , T A_1 , \ldots , T^{n - 2 d_1 -1} A_1
\end{align*}
are in the kernel of $H_{d_1 -1 , n+3 - d_1 } (\mathbfe{\alpha})$ (and thus orthogonal to its final row). Hence, a similar inductive argument as above will show that the entries $\alpha_{2 d_1 +1} , \ldots , \alpha_{n}$ can be uniquely determined in terms of the entries $\alpha_0 , \alpha_1 , \ldots , \alpha_{2 d_1}$, thus concluding the proof of the uniqueness claim. \\

All that remains is to prove the final claim made in the theorem, which we do for the case $r \geq 2$. We have that
\begin{align*}
\mathbfe{\alpha} = (\alpha_0 , \alpha_1 , \ldots , \alpha_n ) \in \mathscr{L}_n (r , \rho_1 , \pi_1 ) .
\end{align*}
By considering the $(\rho , \pi )$-form of the matrix $H_{n_1 , n_2} (\mathbfe{\alpha})$ and removing last $\pi_1$ columns, we can deduce that $\mathbfe{\alpha}^{(1)}$ is in $\mathscr{L}_{n - \pi_1 } ( \rho_1 , \rho_1 , 0)$. Given that $A_1$ is in the kernel of $H_{n+1-r , r+1} (\mathbfe{\alpha})$, and that the associated vector has $\pi_1$ number of zero entries at the end, we can see that $A_1$ is in the kernel of matrix obtained by removing the last $\pi_1$ columns of $H_{n+1-r , r+1} (\mathbfe{\alpha})$; that is, the matrix $H_{n+1-r , \rho_1 +1} (\mathbfe{\alpha}^{(1)} )$. This tells us that $A_1$ is the first characteristic polynomial. Similarly, since $A_2$ is in the kernel of $H_{r-1 , n+3-r} (\mathbfe{\alpha})$, we can see that it is in the kernel of the matrix obtained by removing the last $\pi_1$ rows of $H_{r-1 , n+3-r} (\mathbfe{\alpha})$; that is, the matrix $H_{\rho_1 -1 , n+3-r} (\mathbfe{\alpha}^{(1)} )$. This tells us that $A_2$ is the second characteristic polynomial.
\end{proof}

We now prove Theorem \ref{theorem, all corpime characteristic polynomial have a Hankel matrix}.

\begin{proof}[Proof of Theorem \ref{theorem, all corpime characteristic polynomial have a Hankel matrix}]
The construction of the sequences $\mathbfe{\alpha}$ in Claims 1 and 2 are not difficult, and so we proceed to Claims 3 and 4. \\

\textbf{Claim 3:} Let us write $A_2'$ to be the unique polynomial satisfying
\begin{align*}
A_2 = R A_1 + A_2' \hspace{3em} \degree A_2' < \degree A_1 .
\end{align*}
That is, $A_2'$ is the smallest representative of $A_2$ modulo $A_1$. Recalling our discussion on uniqueness in Definition \ref{definition, characteristic polynomials}, it is equivalent to prove that there exists a sequence $\mathbfe{\alpha} \in \mathscr{L}_n (r,r,0)$ with characteristic polynomials $A_1 , A_2'$. Now, let us define $d_1 := \degree A_1$ and $d_2 := \degree A_2'$, and we apply the Euclidean algorithm to $A_1, A_2'$:
\begin{align*}
A_1 = &R_2 A_2' + A_3 			&&d_3 := \degree A_3 < \degree A_2' , \\
A_2' = &R_3 A_3 + A_4 			&&d_4 := \degree A_4 < \degree A_3 , \\
A_3 = &R_4 A_4 + A_5 			&&d_5 := \degree A_5 < \degree A_4 , \\
&\vdots 					&&\vdots \\
A_t = &R_{t+1} A_{t+1} + A_{t+2} 		&&d_{t+2} := \degree A_{t+2} < \degree A_{t+1} ,
\end{align*}
where $t$ is such that $\degree A_{t} \geq 2 > \degree A_{t+1} \geq 0$. Also, let
\begin{align*}
\mathbfe{\alpha}^{(t+1)} := (\alpha_0 , \alpha_1 , \ldots , \alpha_{d_{t+1}} ) .
\end{align*}
Corollary \ref{corollary, Hankel matrices incorporate Euclidean algorithm theorem, full algorithm presented} tells us that our desired $\mathbfe{\alpha}$ exists if and only if we have
\begin{align*}
\mathbfe{\alpha}^{(t+1)}
\in \begin{cases}
\mathscr{L}_{d_{t}} (2 , 1 , 1) &\text{ if $d_{t+1} =1$,} \\
\mathscr{L}_{d_{t}} (2 , 0 , 2) &\text{ if $d_{t+1} =0$,} 
\end{cases}
\end{align*}
with characteristic polynomials $A_{t+1} , A_{t}$. \\

Suppose $d_{t+1} = 1$. Then, it is equivalent to find $\mathbfe{\alpha}^{(t+1)}$ such that
\begin{align} \label{statement, all corpime characteristic polynomial have a Hankel matrix theorem proof, claim 2, d_(t+1)=1, A_(t+2) mults in kernel}
A_{t+1} \hspace{1em} , \hspace{1em} T A_{t+1} \hspace{1em} , \hspace{1em} \ldots \hspace{1em} , \hspace{1em} T^{d_{t} -2} A_{t+1}
\in \kernel (\alpha_0 , \alpha_1 , \ldots , \alpha_{d_{t}} )
\end{align}
and
\begin{align} \label{statement, all corpime characteristic polynomial have a Hankel matrix theorem proof, claim 2, d_(t+1)=1, A_(t+1) in kernel}
A_{t}
\in \kernel (\alpha_0 , \alpha_1 , \ldots , \alpha_{d_{t}} ) .
\end{align}
To this end, we let $\alpha_0$ take any value in $\mathbb{F}_q$. Since $\degree A_{t+1} = 1$, we can see that (\ref{statement, all corpime characteristic polynomial have a Hankel matrix theorem proof, claim 2, d_(t+1)=1, A_(t+2) mults in kernel}) uniquely determines the values of $\alpha_1 , \ldots , \alpha_{d_{t} -1}$ in terms of $\alpha_0$. Similarly, since $\degree A_{t} = d_{t}$, we can see that (\ref{statement, all corpime characteristic polynomial have a Hankel matrix theorem proof, claim 2, d_(t+1)=1, A_(t+1) in kernel}) uniquely determines the value of $\alpha_{d_{t}}$ in terms of $\alpha_0$. Ultimately, we have shown that our desired $\mathbfe{\alpha}$ exists. The fact that it is unique up to multiplication by elements in $\mathbb{F}_q^*$ follows from the following four facts:
\begin{enumerate}
\item We could let $\alpha_0$ take any value in $\mathbb{F}_q$. \\

\item The entries $\alpha_1 , \ldots , \alpha_{d_{t}}$ can be expressed uniquely in terms of $\alpha_0$. In fact, each such $\alpha_i$ can be expressed as a linear function of $\alpha_0$ that passes through the origin. \\

\item The sequence $\mathbfe{\alpha}$ is uniquely determined by the sequence $\mathbfe{\alpha}^{(t+1)}$, which follows from the uniqueness claim at the end of Corollary \ref{corollary, Hankel matrices incorporate Euclidean algorithm theorem, full algorithm presented}. In fact, we can show that for all $i \leq n$ the entry $\alpha_i$ can be expressed as a linear function of $\alpha_0$ that passes through the origin. \\

\item Finally, we note from the above that if $\alpha_0 = 0$, then $\mathbfe{\alpha} = \mathbf{0}$. We dismiss this case as it does not give us $\mathbfe{\alpha}^{(t+1)} \in \mathscr{L}_{d_{t}} (2 , 1 , 1)$. \\
\end{enumerate}

Now suppose that we have $d_{t+1} = 0$ instead. We can use a similar argument as above. The main difference will be that $\alpha_0 , \ldots , \alpha_{d_{t} -2} = 0$, while $\alpha_{d_{t} -1}$ will be free to take any value in $\mathbb{F}_q^*$, and for all $d_{t} \leq i \leq n$ the entry $\alpha_i$ can be expressed as a linear function of $\alpha_{d_{t} -1}$ that passes through the origin. \\

\textbf{Claim 4:} We know by Claim 2 that there is a sequence $\mathbfe{\alpha}^{(1)} = (\alpha_0 , \alpha_1 , \ldots , \alpha_{n - \pi_1} ) \in \mathscr{L}_{n-\pi_1 } (\rho_1 , \rho_1 , 0)$ with characteristic polynomials $A_1 , A_2$. Note that
\begin{align} \label{statement, all corpime characteristic polynomial have a Hankel matrix theorem proof, claim 3, kernel H_(rho_1 -1 , n+3-r) (alpha^(1))}
\kernel H_{\rho_1 -1 , n+3 - r} (\mathbfe{\alpha}^{(1)})
= \bigg\{ B_1 A_1 + B_2 A_2 : \substack{B_1 , B_2 \in \mathcal{A} \\ \degree B_1 \leq n+2 - r - \rho_1 \\ \degree B_2 \leq 0} \bigg\} .
\end{align}
We now define the extension $\mathbfe{\alpha} = (\alpha_0 , \alpha_1 , \ldots , \alpha_n )$ of $\mathbfe{\alpha}^{(1)}$ by the following property:
\begin{align} \label{statement, all corpime characteristic polynomial have a Hankel matrix theorem proof, claim 3, kernel H_(r -1 , n+3-r) (alpha)}
\kernel H_{r -1 , n+3 - r} (\mathbfe{\alpha})
\subseteq \bigg\{ B_1 A_1 + B_2 A_2 : \substack{B_1 , B_2 \in \mathcal{A} \\ \degree B_1 \leq n+2 - 2r \\ \degree B_2 \leq 0} \bigg\} .
\end{align}
Note that removing the last $\pi_1$ rows of $H_{r -1 , n+3 - r} (\mathbfe{\alpha})$ will leave us with the matrix $\break H_{\rho_1 -1 , n+3 - r} (\mathbfe{\alpha}^{(1)})$ from (\ref{statement, all corpime characteristic polynomial have a Hankel matrix theorem proof, claim 3, kernel H_(rho_1 -1 , n+3-r) (alpha^(1))}). Hence, by successive applications of Lemma \ref{lemma, shifted poly in kernel of Hankel matrix is in kernel of removed row matrix}, regardless of the way we extended $\mathbfe{\alpha}^{(1)}$, we certainly have that
\begin{align*}
\kernel H_{r -1 , n+3 - r} (\mathbfe{\alpha})
\supseteq \Big\{ B_1 A_1 : \substack{B_1 \in \mathcal{A} \\ \degree B_1 \leq n+2 - 2r} \bigg\} .
\end{align*}
The requirement that $A_2$ is in the kernel above will actually uniquely determine the entries $\alpha_{n - \pi_1 +1} , \ldots , \alpha_{n}$ (which form the extended part of $\mathbfe{\alpha}$) in terms of $\alpha_0 , \ldots , \alpha_{n - \pi_1 }$. This follows from the fact that $\alpha_{n - \pi_1 +1} , \ldots , \alpha_{n}$ appear as the final entries in the last $\pi_1$ rows of $H_{r -1 , n+3 - r} (\mathbfe{\alpha})$, and the fact that $A_2$ has degree $n+2-r$. \\

We will now show that $\mathbfe{\alpha}$ is in $\mathscr{L}_n (r , \rho_1 , \pi_1 )$. This is a condition of Claim 3, but it also allows us to determine the dimension of the left side of (\ref{statement, all corpime characteristic polynomial have a Hankel matrix theorem proof, claim 3, kernel H_(r -1 , n+3-r) (alpha)}) by using Corollary \ref{corollary, dim kernel of recatnagular Hankel matrix given r of sequence}. By comparing this to the dimension of the right side of (\ref{statement, all corpime characteristic polynomial have a Hankel matrix theorem proof, claim 3, kernel H_(r -1 , n+3-r) (alpha)}), we see that we must have equality, and thus allowing us to deduce that $A_1 , A_2$ are the characteristic polynomials of $\mathbfe{\alpha}$. All that will be left to prove is that $\mathbfe{\alpha}$ is unique up to multiplication by  elements in $\mathbb{F}_q^*$. \\

Consider the sequence $\mathbfe{\alpha}' := (\alpha_0 , \alpha_1 , \ldots , \alpha_{n - \pi_1 +1} )$ and the associated matrix $H' := \break H_{\rho_1 , n+3 - r} (\mathbfe{\alpha}')$. Note that if we remove the last row from this matrix then we are left with the matrix $H := H_{\rho_1 -1 , n+3 - r} (\mathbfe{\alpha}^{(1)})$ from (\ref{statement, all corpime characteristic polynomial have a Hankel matrix theorem proof, claim 3, kernel H_(rho_1 -1 , n+3-r) (alpha^(1))}). Hence, by Lemma \ref{lemma, shifted poly in kernel of Hankel matrix is in kernel of removed row matrix} we have that
\begin{align*}
\Big\{ B_1 A_1 : \substack{B_1 \in \mathcal{A} \\ \degree B_1 \leq n+1 - r - \rho_1 } \Big\}
\subseteq \kernel H' .
\end{align*}
By our construction of $\mathbfe{\alpha}$,we also have that $A_2$ is in this kernel, and thus
\begin{align*}
\bigg\{ B_1 A_1 + B_2 A_2 : \substack{B_1 , B_2 \in \mathcal{A} \\ \degree B_1 \leq n+1 - r - \rho_1 \\ \degree B_2 \leq 0} \bigg\}
\subseteq \kernel H' .
\end{align*}
In fact, we must have equality above. This follows from the fact that the dimension of $\kernel H'$ is one less than that of $\kernel H$. Indeed, the number of columns remains the same, but the rank of $H'$ is one more than the rank of $H$, which follows by the fact that $H'$ has full row rank since $\rho (\mathbfe{\alpha}')  \geq \rho (\mathbfe{\alpha}^{(1)}) = \rho_1$. Now, because $T^{n+2 - r - \rho_1} A_1$ is not in the kernel of $H'$, we can see that $\pi (H') = 1$, and thus $\pi (\mathbfe{\alpha}') = 1$. By considering $(\rho , \pi )$-forms, we can see that this forces $\pi (\mathbfe{\alpha}') = \pi_1$ (indeed, extending the sequence by one entry will increase the $\pi$-characteristic by one), and thus $\mathbfe{\alpha} \in \mathscr{L}_n (r , \rho_1 , \pi_1 )$. \\

Finally, for the uniqueness claim, this follows from the uniqueness claim in Claim 2 and the way we formed our extension $\mathbfe{\alpha}$ of $\mathbfe{\alpha}^{(1)}$. Technically, we should also prove a converse result; that any sequence $\mathbfe{\alpha}$ with the properties in Claim 3 has a truncation $\mathbfe{\alpha}^{(1)} \in \mathscr{L}_{n-\pi_1 } (\rho_1 , \rho_1 , 0)$ with characteristic polynomials $A_1 , A_2$. This is to ensure that our construction above actually addresses all possibilities for $\mathbfe{\alpha}$. This is not difficult to prove, and we have done something similar in the proof of Theorem \ref{Theorem, Hankel matrices incorporate Euclidean algorithm}.
\end{proof}

Finally, we prove Theorem \ref{theorem, kernel structure subsection, char polys of extended sequences}.

\begin{proof}[Proof of Theorem \ref{theorem, kernel structure subsection, char polys of extended sequences}]

\textbf{Claim 1:} The cases $r \leq 1$ are considerably easier than the cases $r \geq 2$; we consider only the latter. \\

Let $H :=H_{n_1 , n_2} (\mathbfe{\alpha})$. The $(\rho , \pi )$-form of $H$ is
\begin{align*}
\begin{pmatrix}
H [r , r] & \vline & H [r , -(n_2 -r)] \\
\hline
\mathbf{0} & \vline & 
\begin{NiceMatrix}
0 & \Cdots & \Cdots & 0 \\
\Vdots &  &  & \Vdots \\
0 & \Cdots & \Cdots & 0 \\
0 & \Cdots & \Cdots & 0 
\end{NiceMatrix}
\end{pmatrix} .
\end{align*}
By recalling the definition of $(\rho , \pi )$-form, we can see that the $(\rho , \pi )$-form of $H' :=H_{n_1' , n_2'} (\mathbfe{\alpha}')$ is
\begin{align*}
\begin{pmatrix}
H [r , r] & \vline & H [r , -(n_2' -r)] \\
\hline
\mathbf{0} & \vline & 
\begin{NiceMatrix}
0 & \Cdots & \Cdots & 0 \\
\Vdots &  &  & \Vdots \\
0 & \Cdots & \Cdots & 0 \\
0 & \Cdots & 0 & \gamma 
\end{NiceMatrix}
\end{pmatrix} ,
\end{align*}
where $\gamma \in \mathbb{F}_q$, and there is a one-to-one correspondence between $\gamma$ and $\alpha_{n+1}$. If $\gamma = 0$, then we have $\mathbfe{\alpha}' \in \mathscr{L}_{n+1} (r,r,0)$, while if $\gamma \neq 0$, then we have $\mathbfe{\alpha}' \in \mathscr{L}_{n+1} (r+1,r,1)$. The claims on the characteristic degrees follow by definition. \\

Let us now consider the characteristic polynomials. Consider the case $\gamma = 0$ first. By definition, we have that $A_1$ spans the kernel of $H_{c_2 -1 , c_1 +1} (\mathbfe{\alpha})$. By comparing the $(\rho , \pi )$-form of $H_{c_2 -1 , c_1 +1} (\mathbfe{\alpha})$ to the $(\rho , \pi )$-form of $H_{c_2' -1 , c_1' +1} (\mathbfe{\alpha}') = H_{c_2  , c_1 +1} (\mathbfe{\alpha}')$ (the latter having an extra row of zeros at the bottom, compared to the former), we see that $A_1$ spans the kernel of $H_{c_2' -1 , c_1' +1} (\mathbfe{\alpha}')$. Thus, $A_1' = A_1$. \\

Regarding $A_2'$, by definition it is the polynomial in the kernel of $H_{c_1' -1 , c_2' +1} (\mathbfe{\alpha}')$ that is not a multiple of $A_1'$. Now, let $\mathbf{a}_2$ be the vector in the kernel of $H_{c_1 -1 , c_2 +1} (\mathbfe{\alpha})$ that is associated to $A_2$. We can see that $(\mathbf{a}_2 \mid 0)$ is in the kernel of $H_{c_1' -1 , c_2' +1} (\mathbfe{\alpha}') = H_{c_1 -1 , c_2 +2} (\mathbfe{\alpha}')$ (the latter matrix has an extra column compared to the former). In terms of polynomials, this means $A_2$ is in the kernel of $H_{c_1' -1 , c_2' +1} (\mathbfe{\alpha}')$; and since it is not a multiple of $A_1' = A_1$, we have that $A_2' = A_2$. \\

Now consider the case where $\gamma \neq 0$. Let us write $\mathbf{a}_1$ for the vector associated to $A_1$ in the kernel of $H_{c_2 -1 , c_1 +1} (\mathbfe{\alpha})$. Similar to above, we compare the $(\rho , \pi )$-form of $H_{c_2 -1 , c_1 +1} (\mathbfe{\alpha})$ to the $(\rho , \pi )$-form of $H_{c_2' -1 , c_1' +1} (\mathbfe{\alpha}') = H_{c_2 -1  , c_1 +2} (\mathbfe{\alpha}')$ (the latter having an additional column compared to the former, with a non-zero entry in the bottom-right), and we we see that $(\mathbf{a}_1 \mid 0)$ spans the kernel of $H_{c_2' -1 , c_1' +1} (\mathbfe{\alpha}')$. Thus, $A_1' = A_1$. \\

For $A_2'$, as above, it is the polynomial in the kernel of $H_{c_1' -1 , c_2' +1} (\mathbfe{\alpha}') = H_{c_1 , c_2 +1} (\mathbfe{\alpha}')$ that is not a multiple of $A_1'$. Since $A_1' = A_1$ is the first characteristic polynomial, we have 
\begin{align*}
\kernel H_{c_1' -1 , c_2' +1} (\mathbfe{\alpha}') 
\supseteq \Big\{ B_1 A_1' : \substack{B_1 \in \mathcal{A} \\ \degree B_1 \leq c_2' - c_1' } \Big\}
= \Big\{ B_1 A_1 : \substack{B_1 \in \mathcal{A} \\ \degree B_1 \leq c_2 - c_1 -1 } \Big\} .
\end{align*}
We also have
\begin{align} \label{lemma, kernel structure subsection, extending sequence by one, T^(c_1 - c_1) A_1 not in kernel}
T^{c_2 - c_1} A_1 \not\in \kernel H_{c_1' -1 , c_2' +1} (\mathbfe{\alpha}').
\end{align}
Otherwise, by Lemma \ref{lemma, if poly in kernel of Hankel matrix then multiples of poly in wider matrix}, we would have that $A_1$ is in the kernel of $H_{c_2' -2 , c_1'} (\mathbfe{\alpha}')$, which is a contradiction. Note also that $H_{c_1 -1 , c_2 +1} (\mathbfe{\alpha})$ is the matrix we obtain after removing the last row from $H_{c_1' -1 , c_2' +1} (\mathbfe{\alpha}') = H_{c_1 , c_2 +1} (\mathbfe{\alpha}')$. In particular, the kernel of the latter is a subspace of the former. That is,
\begin{align*}
\kernel H_{c_1' -1 , c_2' +1} (\mathbfe{\alpha}')
\subseteq \kernel H_{c_1 -1 , c_2 +1} (\mathbfe{\alpha})
= \bigg\{ B_1 A_1 + B_2 A_2 : \substack{B_1 , B_2 \in \mathcal{A} \\ \degree B_1 \leq c_2 - c_1 \\ \degree B_2 \leq 0 } \bigg\} .
\end{align*}
Thus, we have established that
\begin{align*}
\Big\{ B_1 A_1 : \substack{B_1 \in \mathcal{A} \\ \degree B_1 \leq c_2 - c_1 -1 } \Big\}
\subseteq \kernel H_{c_1' -1 , c_2' +1} (\mathbfe{\alpha}')
\subseteq \bigg\{ B_1 A_1 + B_2 A_2 : \substack{B_1 , B_2 \in \mathcal{A} \\ \degree B_1 \leq c_2 - c_1 \\ \degree B_2 \leq 0 } \bigg\} .
\end{align*}
Corollary \ref{corollary, dim kernel of recatnagular Hankel matrix given r of sequence} tells us that the dimensions of adjacent vector spaces above differ by $1$. Now, $A_2'$ is the vector in $\kernel H_{c_1' -1 , c_2' +1} (\mathbfe{\alpha}')$ that is not a multiple of $A_1' = A_1$, and recall that by Theorem \ref{theorem, all Hankel matrices have characteristic polynomial kernel} the degree of $A_2'$ must be equal to $c_2'$. Thus, we can deduce
\begin{align*}
A_2' = \beta A_2 + T^{c_2 - c_1} A_1
\end{align*}
for some $\beta \in \mathbb{F}_q^*$ ($\beta$ cannot be $0$ by (\ref{lemma, kernel structure subsection, extending sequence by one, T^(c_1 - c_1) A_1 not in kernel})). It is not difficult to see that if we change the value of $\alpha_{n+1}$ (which appears in the last row of $H_{c_1' -1 , c_2' +1} (\mathbfe{\alpha}')$), then the value of $\beta$ will have to change to ensure that the vector associated to $A_2'$ is orthogonal to the last row of $H_{c_1' -1 , c_2' +1} (\mathbfe{\alpha}')$, and thus in the kernel of $H_{c_1' -1 , c_2' +1} (\mathbfe{\alpha}')$.( Note that this makes use of the fact that $A_2$ is not in the kernel of $H_{c_1' -1 , c_2' +1} (\mathbfe{\alpha}')$; indeed, otherwise, $A_2$ would be orthogonal to the last row and altering the value of $\beta$ would have no effect). Hence, there is a one-to-one correspondence between $\alpha_{n+1}$ and $\beta$. \\

\textbf{Claim 2:} Let $H := H_{n_1 , n_2} (\mathbfe{\alpha})$. The $(\rho , \pi )$-form of $H$ is of the form
\begin{align} \label{statement, kernel structure subsection, sequence extension theorem, claim 2, H rho, pi matrix}
\begin{pmatrix}
H [\rho_1 , \rho_1] & \vline & H [\rho_1 , -(n_2 - \rho_1)] \\
\hline
\mathbf{0} & \vline & 
\begin{NiceMatrix}
0 & \Cdots &  &  &  &  & 0 \\
\Vdots &  &  &  &  &  & \Vdots \\
0 & \Cdots &  &  &  &  & 0 \\
0 & \Cdots &  &  &  & 0 & 1 \\
0 & \Cdots &  &  & 0 & 1 & * \\
\Vdots &  &  & \Iddots & \Iddots & \Iddots & \Vdots \\
0 & \Cdots & 0 & 1 & * & \Cdots & * 
\end{NiceMatrix}
\end{pmatrix} ,
\end{align}
where there are $\pi_1$ number of $1$s in the bottom right submatrix (recall, $1$ represents an element in $\mathbb{F}_q^*$, $0$ represents $0$ as usual, and $*$ represents an element in $\mathbb{F}_q$). Now consider the matrix $H' = H_{n_1' , n_2'} (\mathbfe{\alpha}')$, and note that $H'$ is the same matrix as $H$ but it has either an additional row or column at the end (depending on whether $n$ is even or odd). Thus, we can deduce that the $(\rho , \pi)$-form of $H'$ is the same as (\ref{statement, kernel structure subsection, sequence extension theorem, claim 2, H rho, pi matrix}) but with an additional row or column. This additional row or column contributes an additional $1$ in the bottom-right submatrix, and thus we have $\pi (\mathbfe{\alpha}') = \pi_1 +1$. Clearly $\rho (\mathbfe{\alpha}') = \rho_1$, and thus $\mathbfe{\alpha}' \in \mathscr{L}_n (r+1, \rho_1, \pi_1 +1 )$, as required. The claims on the characteristic degrees follow by definition. \\

Now, the proof that $A_1' = A_1$ is similar as in the second case of Claim 1. The proof for the second characteristic polynomial is also similar, but slightly different, and so we give an outline of the proof. \\

By definition of the first characteristic polynomial, we have that $A_1' = A_1$ spans the kernel of $H_{c_2' -1 , c_1' +1} (\mathbfe{\alpha})$. Thus, using Lemma \ref{lemma, if poly in kernel of Hankel matrix then multiples of poly in wider matrix} for the second relation below, we have that
\begin{align*}
\Big\{ B_1 A_1 : \substack{B_1 \in \mathcal{A} \\ \degree B_1 \leq c_2 - c_1 -1} \Big\}
= \Big\{ B_1 A_1' : \substack{B_1 \in \mathcal{A} \\ \degree B_1 \leq c_2' - c_1'} \Big\}
\subseteq \kernel H_{c_1'-1 , c_2'+1} (\mathbfe{\alpha}') ,
\end{align*}
but
\begin{align} \label{lemma, kernel structure subsection, extending sequence by one, pi_1 geq 1 case, T^(c_2 -c_1) A_1 not in kernel}
T^{c_2 - c_1} A_1
\not\in \kernel H_{c_1'-1 , c_2'+1} (\mathbfe{\alpha}') .
\end{align}
We also have that
\begin{align*}
\kernel H_{c_1'-1 , c_2'+1} (\mathbfe{\alpha}')
\subseteq \kernel H_{c_1'-2 , c_2'+1} (\mathbfe{\alpha})
= \kernel H_{c_1-1 , c_2+1} (\mathbfe{\alpha})
= \bigg\{ B_1 A_1 + B_2 A_2 : \substack{B_1 , B_2 \in \mathcal{A} \\ \degree B_1 \leq c_2 - c_1 \\ \degree B_2 \leq 0 } \bigg\} .
\end{align*}
Thus, we have
\begin{align*}
\Big\{ B_1 A_1 : \substack{B_1 \in \mathcal{A} \\ \degree B_1 \leq c_2 - c_1 -1} \Big\}
\subseteq \kernel H_{c_1'-1 , c_2'+1} (\mathbfe{\alpha}') 
\subseteq \bigg\{ B_1 A_1 + B_2 A_2 : \substack{B_1 , B_2 \in \mathcal{A} \\ \degree B_1 \leq c_2 - c_1 \\ \degree B_2 \leq 0 } \bigg\} ,
\end{align*}
and, by Corollary \ref{corollary, dim kernel of recatnagular Hankel matrix given r of sequence}, the dimensions of adjacent vector spaces above differ by $1$. Thus, by (\ref{lemma, kernel structure subsection, extending sequence by one, pi_1 geq 1 case, T^(c_2 -c_1) A_1 not in kernel}), and the fact that Theorem \ref{theorem, all Hankel matrices have characteristic polynomial kernel} tells us that $A_2'$ must be of degree $c_2'$, we have that
\begin{align*}
A_2' = \beta T^{c_2 - c_1} A_1 + A_2
\end{align*}
for some $\beta \in \mathbb{F}_q$. The one-to-one correspondence between $\alpha_{n+1}$ and $\beta$ follows by similar reasoning as in Claim 2. \\

\textbf{Claim 3:}  Consider the first case given in Claim 3. The fact that $\mathbfe{\alpha} \in \mathscr{L}_{n+1} (n_1, n_1, 0 )$, and the claims on the characteristic degrees, are not difficult to deduce. Thus, we restrict our attention to the characteristic polynomials. By definition, $A_1'$ is the polynomial that spans the kernel of $H_{n_1 , n_1 +1} (\mathbfe{\alpha}')$, and Theorem \ref{theorem, all Hankel matrices have characteristic polynomial kernel} tells us that $\degree A_1' = n_1$. Now, the matrix $H_{n_1 , n_1 +1} (\mathbfe{\alpha}')$ is the same as the matrix $\kernel H_{n_1 -1 , n_1 +1} (\mathbfe{\alpha})$ but with an additional row, and we know that
\begin{align*}
\kernel H_{n_1 -1 , n_1 +1} (\mathbfe{\alpha})
= \{ B_1 A_1 + B_2 A_2 : B_1 , B_2 \in \mathbb{F}_q \},
\end{align*}
and so 
\begin{align*}
\kernel H_{n_1 , n_1 +1} (\mathbfe{\alpha}')
\subseteq \{ B_1 A_1 + B_2 A_2 : B_1 , B_2 \in \mathbb{F}_q \}.
\end{align*}
Since we know $\degree A_1' = n_1$, and that $\degree A_1 = n_1$ and $\degree A_2 < n_1$, we must have that $A_1' = \beta A_2 + A_1$ for some $\beta \in \mathbb{F}_q$. The one-to-one correspondence between $\alpha_{n+1}$ and $\beta$ follows similarly as previously. The second characteristic polynomial is, by definition, the polynomial in the kernel of $H_{n_1 -1 , n_1 +2} (\mathbfe{\alpha}')$ that is not a multiple of $A_1'$. Similar to the previous claims, we note that $H_{n_1 -1 , n_1 +2} (\mathbfe{\alpha}')$ is the same as the matrix $H_{n_1 -1 , n_1 +1} (\mathbfe{\alpha}')$ but with an additional column. So, because $A_2$ is in the kernel of the latter, we can see that it must be in the kernel of the former. Hence $A_2' = A_2$. \\

The second case of Claim 3 is very similar to Claim 1. However, for the second bullet point, one should note that, because $\mathbfe{\alpha}$ is quasi-regular (unlike in the analogous result in Claim 1), we require that $\degree A_1' = n_1 +1$, and this is why we take $A_1' = \beta A_2 + T A_1$ (as opposed to $A_2' = \beta A_2 + T A_1 = \beta A_2 + T^{c_2 - c_1} A_1$ which would be completely analogous to Claim 1). However, this ``swapping'' of $A_1'$ and $A_2'$ is important and natural as it plays a role in the manifestation of the Euclidean algorithm that we saw in Theorem \ref{Theorem, Hankel matrices incorporate Euclidean algorithm} and its corollaries. \\

The third case of Claim 3 is very similar to Claim 2. Again, there is a similar ``swapping'' of $A_1'$ and $A_2'$.
\end{proof}

\begin{remark} \label{remark, hankel matrices, kernel structure, higher moments remark}
Recall that in Subsection \ref{subsection, intro, extensions} we considered the generalisation of Theorem \ref{main theorem, variance divisor function intervals} to higher moments such as the third, and we described that we would need to determine how many $\mathbfe{\alpha} , \mathbfe{\beta} , \mathbfe{\gamma}$ there are such that $H_{l+1 , m+1}(\mathbfe{\alpha})$, $H_{l+1 , m+1}(\mathbfe{\beta})$, and $H_{l+1 , m+1}(\mathbfe{\gamma})$ have certain given ranks, and $\mathbfe{\alpha} + \mathbfe{\beta} + \mathbfe{\gamma} = \mathbfe{0}$. Now that we have established various results on Hankel matrices, we can reduce this problem to certain special cases. \\

Let us write
\begin{align*}
\mathbfe{\alpha} = &(\alpha_0 , \alpha_1 , \ldots , \alpha_n ) \\
\mathbfe{\beta} = &(\beta_0 , \beta_1 , \ldots , \beta_n ) \\
\mathbfe{\gamma} = &(\gamma_0 , \gamma_1 , \ldots , \gamma_n ) 
\end{align*}
and 
\begin{align*}
\mathbfe{\alpha}^- = &(\alpha_0 , \alpha_1 , \ldots , \alpha_{n-1} ) \\
\mathbfe{\beta}^- = &(\beta_0 , \beta_1 , \ldots , \beta_{n-1} ) \\
\mathbfe{\gamma}^- = &(\gamma_0 , \gamma_1 , \ldots , \gamma_{n-1} ) 
\end{align*}
Suppose we know how many $\mathbfe{\alpha}^- , \mathbfe{\beta}^- , \mathbfe{\gamma}^-$ there are with certain given $(\rho , \pi )$-characteristics and $\mathbfe{\alpha}^- + \mathbfe{\beta}^- + \mathbfe{\gamma}^- = \mathbfe{0}$. If at least one of these sequences $\mathbfe{\alpha}^- , \mathbfe{\beta}^- , \mathbfe{\gamma}^-$ has non-zero $\pi$-characteristic (that is, not all are quasi-regular) then we can determine the $(\rho , \pi )$-characteristics of $\mathbfe{\alpha} , \mathbfe{\beta} , \mathbfe{\gamma}$, even when we impose the condition that $\mathbfe{\alpha} + \mathbfe{\beta} + \mathbfe{\gamma} = \mathbfe{0}$. \\

Indeed, without loss of generality, suppose that $\pi (\mathbfe{\gamma}^- ) \geq 1$. Theorem \ref{theorem, kernel structure subsection, char polys of extended sequences} allows us to determine the $(\rho , \pi )$-characteristics of $\mathbfe{\alpha} , \mathbfe{\beta}$ based on the $(\rho , \pi )$-characteristics of $\mathbfe{\alpha}^- , \mathbfe{\beta}^-$. However, if, for example, $\mathbfe{\alpha}^-$ is quasi-regular, then the $(\rho , \pi )$-characteristic of $\mathbfe{\alpha}$ would depend on $\alpha_n$. Theorem \ref{theorem, kernel structure subsection, char polys of extended sequences} tells us how many values $\alpha_n$ can take so that the $(\rho , \pi )$-characteristic of $\mathbfe{\alpha}$ takes a specific value. However, it does not tell us exactly what those values are, and this is why it is important that $\pi (\mathbfe{\gamma}^- ) \geq 1$: Regardless of the values of $\alpha_n$ and $\beta_n$, we know there exists $\gamma_n$ such that $\alpha_n + \beta_n + \gamma_n = 0$, and we still know what the $(\rho , \pi )$-characteristic of $\mathbfe{\gamma}$ is because it is independent of the value of $\gamma_n$ (which uses the fact that $\pi (\mathbfe{\gamma}^- ) \geq 1$). \\

Thus we have reduce our problem to the cases where $\mathbfe{\alpha}^- , \mathbfe{\beta}^- , \mathbfe{\gamma}^-$ are all quasi-regular.
\end{remark}

\section{The Variance of the Divisor Function} \label{section, variance of divisor function}

We now prove Theorem \ref{main theorem, variance divisor function intervals}. 

\begin{proof}[Proof of Theorem \ref{main theorem, variance divisor function intervals}]

In what follows, $n \geq 4$. \\

We have that
\begin{align} \label{statement, div var section, variance as Delta in terms of noncentred d(B)}
\frac{1}{q^n} \sum_{A \in \mathcal{M}_n} \lvert \Delta (A;h) \rvert^2
= \frac{1}{q^n} \sum_{A \in \mathcal{M}_n} \bigg( \sum_{B \in I (A;h)} d(B) \bigg)^2 - q^{2h} (n+1)^2 .
\end{align}
So, we will consider the first term on the right side.
\begin{align*}
&\sum_{A \in \mathcal{M}_n} \bigg( \sum_{B \in I (A;h)} d(B) \bigg)^2 \\
= &\sum_{A \in \mathcal{M}_n} \bigg( \sum_{\substack{B \in \mathcal{M}_n \\ \degree (B-A) < h}} d(B) \bigg)^2 \\
= &\sum_{A \in \mathcal{M}_n}
	\bigg( \sum_{\substack{B \in \mathcal{M}_n \\ \degree (B-A) < h}}  
	\sum_{\substack{0 \leq l,m \leq n \\ l+m=n}}
	\sum_{E \in \mathcal{M}_l} \sum_{F \in \mathcal{M}_m} \mathbbm{1}_{EF=B} \bigg)^2 \\
= &\sum_{A \in \mathcal{A}_{\leq n }}
	\bigg( \sum_{\substack{B \in \mathcal{M}_n \\ \degree (B-A) < h}}  
	\sum_{\substack{0 \leq l,m \leq n \\ l+m=n}}
	\sum_{E \in \mathcal{M}_l} \sum_{F \in \mathcal{M}_m} \mathbbm{1}_{EF=B} \bigg)^2 .
\end{align*}
For the last line, we changed the first summation range from $A \in \mathcal{M}_n$ to all $A \in \mathcal{A}$ with $\degree A \leq n$. This does not change the result because the conditions on $E$ and $F$ force $A$ to be in $\mathcal{M}_n$. Continuing, we have
\begin{align*}
&\sum_{A \in \mathcal{M}_n} \bigg( \sum_{B \in I (A;h)} d(B) \bigg)^2 \\
= &\sum_{A \in \mathcal{A}_{\leq n }}
	\bigg( \sum_{\substack{B \in \mathcal{M}_n \\ \degree (B-A) < h}}  
	\sum_{\substack{0 \leq l,m \leq n \\ l+m=n}}
	\sum_{E \in \mathcal{M}_l} \sum_{F \in \mathcal{M}_m} 
	\prod_{i=0}^{n} \mathbbm{1}_{\{ EF \}_i = b_i } \bigg)^2 .
\end{align*}
Here, for a polynomial $C$, we let $\{ C \}_i$ denote its $i$-ith coefficient, which is convenient when working with products of polynomials such as $EF$. We also  denote $i$-th coefficient of $A,B,E,F$ by $a_i , b_i , e_i , f_i$, respectively. \\

We will now express the sums over $A,B,E,F$ by sums over their coefficients. For example $\sum_{A \in \mathcal{A}_{\leq n }}$ will be expressed as $\sum_{a_0 , \ldots , a_n \in \mathbb{F}_q}$. We also note that if $A'$ satisfies $\degree (A-A') < h$, then (due to the sum over $B$) both $A$ and $A'$ give the same contribution. For a given $A$, there are $q^h$ such $A'$. Thus, we can multiply the whole expression by $q^h$ and consider only the $A$ that are of the form
\begin{align*}
A = 0 + 0 T + 0 T^2 + \ldots + 0 T^{h-1} + a_h T^h + a_{h+1} T^{h+1} + \ldots + a_{n} T^{n} .
\end{align*}
Furthermore, this means that $B$ is of the form
\begin{align*}
B = b_0 + b_1 T + b_2 T^2 + \ldots + b_{h-1} T^{h-1} + a_h T^h + a_{h+1} T^{h+1} + \ldots + a_{n} T^{n} .
\end{align*}
Thus, using the above and applying (\ref{statement, intro, exp identity sum}) to the terms $\mathbbm{1}_{\{ EF \}_i = b_i }$, we obtain
\begin{align}
\begin{split} \label{statement, div var section, variance as exp sums before alpha_i = 0 for i < h}
&\sum_{A \in \mathcal{M}_n} \bigg( \sum_{B \in I (A;h)} d(B) \bigg)^2 \\
= &\frac{q^{h}}{q^{2n+2}} \sum_{a_h , \ldots , a_n \in \mathbb{F}_q }
	\Bigg( \sum_{b_0 , \ldots , b_{h-1} \in \mathbb{F}_q}  
	\sum_{\substack{0 \leq l,m \leq n \\ l+m=n}}
	\sum_{\substack{e_0 , \ldots , e_{l-1} \in \mathbb{F}_q \\ e_l =1 }}
	\sum_{\substack{f_0 , \ldots , f_{m-1} \in \mathbb{F}_q \\ f_m =1 }} 
	\prod_{i=0}^{h-1} \sum_{\alpha_i \in \mathbb{F}_q }
	\psi \bigg( \alpha_i \Big( b_i - \sum_{\substack{0 \leq l_1 \leq l \\ 0 \leq m_1 \leq m \\ l_1 + m_1 = i }} e_{l_1} f_{m_1} \Big) \bigg) \\
	&\hspace{15em} \times \prod_{i=h}^{n} \sum_{\alpha_i \in \mathbb{F}_q } 
	\psi \bigg( \alpha_i \Big( a_i - \sum_{\substack{0 \leq l_1 \leq l \\ 0 \leq m_1 \leq m \\ l_1 + m_1 = i }} e_{l_1} f_{m_1} \Big) \bigg) \Bigg)^2 .
\end{split}
\end{align}
Now, consider the sum over one of the coefficients $b_i$, and apply (\ref{statement, intro, exp identity sum}). We obtain,
\begin{align*}
\frac{1}{q} \sum_{b_i \in \mathbb{F}_q } \psi ( \alpha_i b_i )
= \begin{cases}
1 &\text{ if $\alpha_i = 0$,} \\
0 &\text{ if $\alpha_i \in \mathbb{F}_q^*$.} 
\end{cases}
\end{align*}
Thus, we require $\alpha_i = 0$ in order to have a non-zero contribution to (\ref{statement, div var section, variance as exp sums before alpha_i = 0 for i < h}). Applying this for $i = 0 , \ldots , h-1$, we obtain
\begin{align}
\begin{split} \label{statement, div var section, variance as exp sums after alpha_i = 0 for i < h}
&\sum_{A \in \mathcal{M}_n} \bigg( \sum_{B \in I (A;h)} d(B) \bigg)^2 \\
= &\frac{q^{3h}}{q^{2n+2}} \sum_{a_h , \ldots , a_n \in \mathbb{F}_q }
	\Bigg( \sum_{\substack{0 \leq l,m \leq n \\ l+m=n}}
	\sum_{\substack{e_0 , \ldots , e_{l-1} \in \mathbb{F}_q \\ e_l =1 }}
	\sum_{\substack{f_0 , \ldots , f_{m-1} \in \mathbb{F}_q \\ f_m =1 }} 
	\prod_{i=h}^{n} \sum_{\alpha_i \in \mathbb{F}_q } 
	\psi \bigg( \alpha_i \Big( a_i - \sum_{\substack{0 \leq l_1 \leq l \\ 0 \leq m_1 \leq m \\ l_1 + m_1 = i }} e_{l_1} f_{m_1} \Big) \bigg) \Bigg)^2 \\
= &\frac{q^{3h}}{q^{2n+2}} \sum_{a_h , \ldots , a_n \in \mathbb{F}_q }
	\Bigg( \sum_{\substack{0 \leq l,m \leq n \\ l+m=n}}
	\sum_{\substack{e_0 , \ldots , e_{l-1} \in \mathbb{F}_q \\ e_l =1 }}
	\sum_{\substack{f_0 , \ldots , f_{m-1} \in \mathbb{F}_q \\ f_m =1 }} 
	\prod_{i=h}^{n} \sum_{\alpha_i \in \mathbb{F}_q } 
	\psi \bigg( \alpha_i \Big( a_i - \sum_{\substack{0 \leq l_1 \leq l \\ 0 \leq m_1 \leq m \\ l_1 + m_1 = i }} e_{l_1} f_{m_1} \Big) \bigg) \Bigg) \\
	&\hspace{6em} \times \Bigg( \sum_{\substack{0 \leq l' , m' \leq n \\ l' + m' = n}}
	\sum_{\substack{g_0 , \ldots , g_{l' -1} \in \mathbb{F}_q \\ g_{l'} =1 }}
	\sum_{\substack{h_0 , \ldots , h_{m' -1} \in \mathbb{F}_q \\ h_{m'} =1 }} 
	\prod_{i=h}^{n} \sum_{\beta_i \in \mathbb{F}_q } 
	\psi \bigg( \beta_i \Big( a_i - \sum_{\substack{0 \leq l_1' \leq l' \\ 0 \leq m_1' \leq m' \\ l_1' + m_1' = i }} g_{l_1'} h_{m_1'} \Big) \bigg) \Bigg) .
	\end{split}
\end{align}
We now consider the sum over one of the coefficients $a_i$. Unlike the $b_i$ which appeared within the largest parentheses, the $a_i$ appear outside. Thus, we must simultaneously consider the terms within each pair of parentheses whose product forms the square, and that is why we have written them explicitly in the last line above. Again, we apply (\ref{statement, intro, exp identity sum}) to obtain
\begin{align*}
\frac{1}{q} \sum_{a_i \in \mathbb{F}_q } \psi \big( (\alpha_i + \beta_i) a_i \big)
= \begin{cases}
1 &\text{ if $\alpha_i = - \beta_i$,} \\
0 &\text{ if $\alpha_i \neq - \beta_i$.} 
\end{cases}
\end{align*}
Applying this to (\ref{statement, div var section, variance as exp sums after alpha_i = 0 for i < h}) gives
\begin{align}
\label{statement, div var section, variance 2nd last step}
&\sum_{A \in \mathcal{M}_n} \bigg( \sum_{B \in I (A;h)} d(B) \bigg)^2 \notag \\
= &q^{2h-n-1} \sum_{\alpha_{h} , \alpha_{h+1} , \ldots , \alpha_{n} \in \mathbb{F}_q}
	\Bigg( \sum_{\substack{0 \leq l,m \leq n \\ l+m=n}}
	\sum_{\substack{e_0 , \ldots , e_{l-1} \in \mathbb{F}_q \\ e_l =1 }}
	\sum_{\substack{f_0 , \ldots , f_{m-1} \in \mathbb{F}_q \\ f_m =1 }} 
	\prod_{i=h}^{n} \psi \bigg( - \alpha_i \sum_{\substack{0 \leq l_1 \leq l \\ 0 \leq m_1 \leq m \\ l_1 + m_1 = i }} e_{l_1} f_{m_1} \bigg) \Bigg) \notag \\
	&\hspace{6em} \times \Bigg( \sum_{\substack{0 \leq l' , m' \leq n \\ l' + m' = n}}
	\sum_{\substack{g_0 , \ldots , g_{l' -1} \in \mathbb{F}_q \\ g_{l'} =1 }}
	\sum_{\substack{h_0 , \ldots , h_{m' -1} \in \mathbb{F}_q \\ h_{m'} =1 }} 
	\prod_{i=h}^{n} \psi \bigg( \alpha_i \sum_{\substack{0 \leq l_1' \leq l' \\ 0 \leq m_1' \leq m' \\ l_1' + m_1' = i }} g_{l_1'} h_{m_1'} \bigg) \Bigg) \notag \\
= &q^{2h-n-1} \sum_{\mathbfe{\alpha} \in \mathscr{L}_n^h }
	\Bigg( \sum_{\substack{0 \leq l,m \leq n \\ l+m=n}}
	\sum_{\substack{e_0 , \ldots , e_{l-1} \in \mathbb{F}_q \\ e_l =1 }}
	\sum_{\substack{f_0 , \ldots , f_{m-1} \in \mathbb{F}_q \\ f_m =1 }} 
	\prod_{i=h}^{n} \psi \bigg( - \alpha_i \sum_{\substack{0 \leq l_1 \leq l \\ 0 \leq m_1 \leq m \\ l_1 + m_1 = i }} e_{l_1} f_{m_1} \bigg) \Bigg) \notag \\
	&\hspace{6em} \times \Bigg( \sum_{\substack{0 \leq l' , m' \leq n \\ l' + m' = n}}
	\sum_{\substack{g_0 , \ldots , g_{l' -1} \in \mathbb{F}_q \\ g_{l'} =1 }}
	\sum_{\substack{h_0 , \ldots , h_{m' -1} \in \mathbb{F}_q \\ h_{m'} =1 }} 
	\prod_{i=h}^{n} \psi \bigg( \alpha_i \sum_{\substack{0 \leq l_1' \leq l' \\ 0 \leq m_1' \leq m' \\ l_1' + m_1' = i }} g_{l_1'} h_{m_1'} \bigg) \Bigg) \\
= &q^{2h-n-1} \sum_{\mathbfe{\alpha} \in \mathscr{L}_n^h }
	\Bigg( \sum_{\substack{0 \leq l,m \leq n \\ l+m=n}}
	\sum_{\substack{\mathbf{e} \in {\mathbb{F}_q}^l \times \{ 1 \} \\ \mathbf{f} \in {\mathbb{F}_q}^m \times \{ 1 \} }}
	\psi \bigg( - \mathbf{e}^T H_{l+1 , m+1} (\mathbfe{\alpha}) \mathbf{f} \bigg) \Bigg) \notag \\
	&\hspace{6em} \times \Bigg( \sum_{\substack{0 \leq l' , m' \leq n \\ l' + m' = n}}
	\sum_{\substack{\mathbf{g} \in {\mathbb{F}_q}^{l'} \times \{ 1 \} \\ \mathbf{h} \in {\mathbb{F}_q}^{m'} \times \{ 1 \} }}
	\psi \bigg( \mathbf{g}^T H_{l' + 1 , m' + 1} (\mathbfe{\alpha}) \mathbf{h} \bigg) \Bigg) \notag \\
= &q^{2h+n} \sum_{r=0, h+1 , h+2 , \ldots , n_1 -1} \sum_{\mathbfe{\alpha}^- \in \mathscr{L}_{n-1}^h (r,r,0) } (n+1-2r)^2 p^{-2r} \notag \\
= &(q-1) q^{h+n-1} \sum_{r=h+1}^{n_1 -1} (n+1-2r)^2
	\hspace{2em} + q^{2h+n} (n+1)^2 \notag 
\end{align}
For the second equality, $\mathbfe{\alpha} := (\alpha_0 , \alpha_1 , \ldots , \alpha_n)$; of course, since $\mathbfe{\alpha} \in \mathscr{L}_n^h$, we have $\alpha_0 , \alpha_1 , \ldots , \break \alpha_{h-1} = 0$, which is consistent with the previous line. For the third equality, we are writing $\mathbf{e} = (e_1 , e_2 , \ldots , e_l )^T$ and similarly for $\mathbf{f}, \mathbf{g}, \mathbf{h}$. The fourth equality uses Lemma \ref{lemma, variance section, evaluation of exp sums of vectors and matrices} below, and the fifth equality uses Claims 2 and 3 from Theorem \ref{theorem, matrices of given rank and size}. So, applying this to (\ref{statement, div var section, variance as Delta in terms of noncentred d(B)}) gives
\begin{align*}
\frac{1}{q^n} \sum_{A \in \mathcal{M}_n} \lvert \Delta (A;h) \rvert^2
= \begin{cases}
(q-1) q^{h-1} \frac{(n-2h-1)(n-2h)(n-2h+1)}{6} &\text{ for $h \leq n_1 -2$,} \\
0 &\text{ for $h \geq n_1 -1$.}
\end{cases}
\end{align*}

\end{proof}

\begin{lemma} \label{lemma, variance section, evaluation of exp sums of vectors and matrices}

In what follows we have
\begin{align*}
\mathbfe{\alpha} = (\alpha_0 , \alpha_1 , \ldots , \alpha_n ) \in \mathbb{F}_q^{n+1}
\end{align*}
and we define
\begin{align*}
\mathbfe{\alpha}^- := (\alpha_0 , \alpha_1 , \ldots , \alpha_{n-1} ) .
\end{align*}
The two claims below address all possible values that $\mathbfe{\alpha}$ could take.

\textbf{Claim 1:} Suppose $\mathbfe{\alpha}$ is such that $\mathbfe{\alpha}^{-} \in \mathscr{L}_{n-1} (r , \rho_1 , \pi_1 )$, where $0 \leq r \leq n_1 -1$ and $\pi_1 \geq 1$. For all $l,m \geq 0$ with $l+m=n$, we have
\begin{align} \label{statement, variance section, evaluation of exp sums of vectors and matrices lemma, pi geq 2 eq}
\sum_{\substack{\mathbf{e} \in {\mathbb{F}_q}^l \times \{ 1 \} \\ \mathbf{f} \in {\mathbb{F}_q}^m \times \{ 1 \} }}
	\psi \bigg( - \mathbf{e}^T H_{l+1 , m+1} (\mathbfe{\alpha}) \mathbf{f} \bigg)
= 0 .
\end{align}

Now suppose $n$ is odd and $\mathbfe{\alpha}$ is such that $\mathbfe{\alpha}^{-} \in \mathscr{L}_{n-1} (n_1 , 0 , 0 )$. For all $l,m \geq 0$ with $l+m=n$, we have
\begin{align*}
\sum_{\substack{\mathbf{e} \in {\mathbb{F}_q}^l \times \{ 1 \} \\ \mathbf{f} \in {\mathbb{F}_q}^m \times \{ 1 \} }}
	\psi \bigg( - \mathbf{e}^T H_{l+1 , m+1} (\mathbfe{\alpha}) \mathbf{f} \bigg)
= 0 .
\end{align*}

\textbf{Claim 2:} Suppose $\mathbfe{\alpha}$ is such that $\mathbfe{\alpha}^- \in \mathscr{L}_{n-1} (r,r,0)$, where $0 \leq r \leq n_1 -1$. We will fix $\mathbfe{\alpha}^-$ and consider all possible extensions $\mathbfe{\alpha}$; that is, we let $\alpha_n$ vary in $\mathbb{F}_q$. We have
\begin{align*}
&\sum_{\alpha_n \in \mathbb{F}_q }
	\Bigg( \sum_{\substack{0 \leq l,m \leq n \\ l+m=n}}
	\sum_{\substack{\mathbf{e} \in {\mathbb{F}_q}^l \times \{ 1 \} \\ \mathbf{f} \in {\mathbb{F}_q}^m \times \{ 1 \} }}
	\psi \big( - \mathbf{e}^T H_{l+1 , m+1} (\mathbfe{\alpha}) \mathbf{f} \big) \Bigg) \\
	&\hspace{6em} \times \Bigg( \sum_{\substack{0 \leq l' , m' \leq n \\ l' + m' = n}}
	\sum_{\substack{\mathbf{g} \in {\mathbb{F}_q}^{l'} \times \{ 1 \} \\ \mathbf{h} \in {\mathbb{F}_q}^{m'} \times \{ 1 \} }}
	\psi \big( \mathbf{g}^T H_{l' + 1 , m' + 1} (\mathbfe{\alpha}) \mathbf{h} \big) \Bigg) \\
= &q^{2n - 2r +1} (n+1-2r)^2 .
\end{align*}
\end{lemma}

\begin{proof}
In what follows, $e_i$ and $f_i$ are the $i$-th entries of $\mathbf{e}$ and $\mathbf{f}$, respectively. \\

\textbf{Claim 1:} Consider the first result in this claim. Since $\pi_1 \geq 1$, we can see from Theorem \ref{theorem, all Hankel matrices have characteristic polynomial kernel} tells us that for $m \leq n-r-1$ there is no monic polynomial of degree $m$ in the kernel of $H_{l , m+1} (\mathbfe{\alpha}^-)$. That is, there is no vector $\mathbf{f} \in {\mathbb{F}_q}^m \times \{ 1 \}$ such that $H_{l , m+1} (\mathbfe{\alpha}^-) \mathbf{f} = \mathbf{0}$. Therefore, for any $\mathbf{f} \in {\mathbb{F}_q}^m \times \{ 1 \}$ we can find a row, $R_i$, of $H_{l , m+1} (\mathbfe{\alpha})$ such that $R_i \mathbf{f} \neq 0$. Now, consider only the sum and terms involving $e_i$ on the left side of (\ref{statement, variance section, evaluation of exp sums of vectors and matrices lemma, pi geq 2 eq}); we have
\begin{align*}
\sum_{e_i \in \mathbb{F}_q} \psi ( - e_i R_i \mathbf{f} )
= 0 ,
\end{align*}
where we have used (\ref{statement, intro, exp identity sum}) and the fact that $R_i \mathbf{f} \neq 0$. Thus, the left side of (\ref{statement, variance section, evaluation of exp sums of vectors and matrices lemma, pi geq 2 eq}) is indeed $0$. \\

Let us now consider when $m \geq n-r$. Let $c_1^- := r$ and $c_2^- := n+1-r$. Theorem \ref{theorem, all Hankel matrices have characteristic polynomial kernel} tells us that 
\begin{align} \label{statement, variance section, evaluation of exp sums of vectors and matrices lemma proof, kernel smaller matr}
\kernel H_{l,m+1} (\mathbfe{\alpha}^-) 
= \bigg\{ B_1 A_1 + B_2 A_2 : \substack{B_1 , B_2 \in \mathcal{A} \\ \degree B_1 \leq m - c_1^- \\ \degree B_2 \leq m - c_2^- } \bigg\} ,
\end{align}
for some $A_1 \in \mathcal{A}_{\rho_1}$ and $A_2 \in \mathcal{A}_{c_2^-}$. Note that in this case, there are monic polynomials of degree $m$ in the kernel of $H_{l,m+1} (\mathbfe{\alpha}^-)$. That is, there are vectors $\mathbf{f} \in {\mathbb{F}_q}^m \times \{ 1 \}$ such that $H_{l , m+1} (\mathbfe{\alpha}^-) \mathbf{f} = \mathbf{0}$. By similar reasoning as above, these are the only candidates for non-zero contributions to the left side of (\ref{statement, variance section, evaluation of exp sums of vectors and matrices lemma, pi geq 2 eq}). \\

Now, $A_1 , A_2$ are the characteristic polynomials of $\mathbfe{\alpha}^-$. Theorem \ref{theorem, kernel structure subsection, char polys of extended sequences} tells us that 
\begin{align*}
\mathbfe{\alpha} 
\in \begin{cases}
\mathscr{L}_{n} (n_1, n_1 , 0 ) &\text{ if $n$ is even and $r= n_1 -1$,} \\
\mathscr{L}_{n} (r+1 , \rho_1 , \pi +1)  &\text{ otherwise.}
\end{cases}
\end{align*} 
Consider the latter case. Theorem \ref{theorem, kernel structure subsection, char polys of extended sequences} also tells us that the characteristic polynomials of $\mathbfe{\alpha}$ are $A_1$ and $\beta T^{c_2^- - c_1^- } A_1 + A_2$, and so Theorem \ref{theorem, all Hankel matrices have characteristic polynomial kernel} tells us that
\begin{align*}
\kernel H_{l+1 , m+1} (\mathbfe{\alpha}) 
= \bigg\{ B_1 A_1 + B_2 (\beta T^{c_2^- - c_1^- } A_1 + A_2 ) : \substack{B_1 , B_2 \in \mathcal{A} \\ \degree B_1 \leq m - c_1^- -1 \\ \degree B_2 \leq m - c_2^- } \bigg\} .
\end{align*}
Comparing this to (\ref{statement, variance section, evaluation of exp sums of vectors and matrices lemma proof, kernel smaller matr}), we see that
\begin{align} \label{statement, variance section, evaluation of exp sums of vectors and matrices lemma proof, kernel smaller matr in terms of larger matr}
\kernel H_{l,m+1} (\mathbfe{\alpha}^-) 
=\kernel H_{l+1 , m+1} (\mathbfe{\alpha}) 
	\hspace{1em} \oplus \hspace{1em} \{ \gamma T^{m-c_1^- } A_1 : \gamma \in \mathbb{F}_q \} .
\end{align}
(This is not surprising as the dimension of $\kernel H_{l,m+1} (\mathbfe{\alpha}^-) $ is just one more than the dimension of $\kernel H_{l+1 , m+1} (\mathbfe{\alpha})$). So, for $\mathbf{f} \in {\mathbb{F}_q}^m \times \{ 1 \}$ that are in $\kernel H_{l,m+1} (\mathbfe{\alpha}^-) $, we can write $\mathbf{f} = \mathbf{f}^- + \gamma \mathbf{g}_m$; where $\mathbf{f}^- \in \kernel H_{l+1 , m+1} (\mathbfe{\alpha})$, $\gamma \in \mathbb{F}_q$, and $\mathbf{g}_m$ is the polynomial associated with $T^{m-c_1^- } A_1$. Hence,
\begin{align*}
\mathbf{e}^T H_{l+1 , m+1} (\mathbfe{\alpha}) \mathbf{f}
= \gamma R_{l+1} \mathbf{g}_m ,
\end{align*}
where $R_{l+1}$ is the last row of $H_{l+1 , m+1} (\mathbfe{\alpha})$. Thus, we have
\begin{align*}
&\sum_{\substack{\mathbf{e} \in {\mathbb{F}_q}^l \times \{ 1 \} \\ \mathbf{f} \in {\mathbb{F}_q}^m \times \{ 1 \} }}
	\psi \Big( - \mathbf{e}^T H_{l+1 , m+1} (\mathbfe{\alpha}) \mathbf{f} \Big) \\
= &\sum_{\substack{\mathbf{e} \in {\mathbb{F}_q}^l \times \{ 1 \} \\ \mathbf{f}^- \in ({\mathbb{F}_q}^m \times \{ 1 \}) \cap \kernel H_{l+1 , m+1} (\mathbfe{\alpha}) \\ \gamma \in \mathbb{F}_q }}
	\psi \Big( - \mathbf{e}^T H_{l+1 , m+1} (\mathbfe{\alpha}) (\mathbf{f}^- +\gamma \mathbf{g}_m ) \Big) \\
= &\sum_{\substack{\mathbf{e} \in {\mathbb{F}_q}^l \times \{ 1 \} \\ \mathbf{f}^- \in ({\mathbb{F}_q}^m \times \{ 1 \}) \cap \kernel H_{l+1 , m+1} (\mathbfe{\alpha}) \\ \gamma \in \mathbb{F}_q }}
	\psi \Big( - \gamma R_{l+1} \mathbf{g}_m \Big) \\
= &q^l C \sum_{\gamma \in \mathbb{F}_q} \psi \Big( - \gamma R_{l+1} \mathbf{g}_m \Big) 
= 0 ,
\end{align*}
where for the last equality we have used (\ref{statement, intro, exp identity sum}) and the fact that $R_{l+1} \mathbf{g} \neq 0$ (since $\mathbf{g} \not\in \break \kernel H_{l+1 , m+1} (\mathbfe{\alpha})$), and $C$ is the number of $\mathbf{f}^-$ in $({\mathbb{F}_q}^m \times \{ 1 \}) \cap \kernel H_{l+1 , m+1} (\mathbfe{\alpha})$ (which we can calculate explicitly but do not need to). \\

The case where $n$ is even and $r=n_1 -1$ follows similarly as above. \\

The second result of Claim 1 is also proved similarly to the above. \\

\textbf{Claim 2:} By similar means as in Claim 1, we can show that
\begin{align} \label{statement, variance section, evaluation of exp sums of vectors and matrices lemma proof, claim 2, equals 0 cases}
\sum_{\substack{\mathbf{e} \in {\mathbb{F}_q}^l \times \{ 1 \} \\ \mathbf{f} \in {\mathbb{F}_q}^m \times \{ 1 \} }}
	\psi \Big( \mathbf{e}^T H_{l+1 , m+1} (\mathbfe{\alpha}) \mathbf{f} \Big)
= 0 ,
\end{align}
for $m \leq r-1$ and $m \geq n+1-r$. \\

So, suppose that $r \leq m \leq n-r$. Consider
\begin{align*}
\sum_{\substack{\mathbf{e} \in {\mathbb{F}_q}^l \times \{ 1 \} \\ \mathbf{f} \in {\mathbb{F}_q}^m \times \{ 1 \} }}
	\psi \Big( - \mathbf{e}^T H_{l+1 , m+1} (\mathbfe{\alpha}) \mathbf{f} \Big) ;
\end{align*}
As previously, a non-zero contribution requires that $\mathbf{f} \in \kernel H_{l,m+1} (\mathbfe{\alpha}^-)$. Let $R_{l+1} (\mathbfe{\alpha})$ be the last row of $H_{l+1,m+1} (\mathbfe{\alpha})$. We can deduce the following two points from Theorem \ref{theorem, kernel structure subsection, char polys of extended sequences}:
\begin{itemize}
\item There are $q-1$ values of $\alpha_{n}$ such that 
\begin{align*}
\mathbfe{\alpha} 
\in \begin{cases}
\mathscr{L}_{n} (n_1 , n_1 , 0) &\text{ if $n$ is even and $r=n_1 -1$,} \\
\mathscr{L}_{n} (r+1 , r , 1) &\text{ otherwise.}
\end{cases} 
\end{align*}
Theorem \ref{theorem, all Hankel matrices have characteristic polynomial kernel} tells us that 
\begin{align*}
\kernel H_{l,m+1} (\mathbfe{\alpha}^-)
= \Big\{ B_1 A_1 : \substack{ B_1 \in \mathcal{A} \\ \degree B_1 \leq m-r} \Big\}
\end{align*}
and
\begin{align*}
\kernel H_{l+1,m+1} (\mathbfe{\alpha})
= \Big\{ B_1 A_1 : \substack{ B_1 \in \mathcal{A} \\ \degree B_1 \leq m-1-r} \Big\} ,
\end{align*}
for some $A_1 \in \mathcal{M}_{r}$. Thus, any $\mathbf{f} \in ({\mathbb{F}_q}^m \times \{ 1 \}) \cap \kernel H_{l,m+1} (\mathbfe{\alpha}^-)$ can be written as $\mathbf{f} = \mathbf{f}^- + \mathbf{a}_m$ for some $\mathbf{f}^- \in \kernel H_{l+1,m+1} (\mathbfe{\alpha})$, and $\mathbf{a}_{m}$ is the polynomial associated with $T^{m-r} A_1$. This gives 
\begin{align*}
\mathbf{e}^T H_{l+1 , m+1} (\mathbfe{\alpha}) \mathbf{f} = R_{l+1} (\mathbfe{\alpha}) \mathbf{a}_m \neq 0. 
\end{align*}
Note that $R_{l+1} (\mathbfe{\alpha}) \mathbf{a}_m$ is independent of the values of $l,m$ (so long as $r \leq m \leq n-r$) This is because the non-zero entries of $\mathbf{a}_m$ occur in the last $r+1$ entries, and the last $r+1$ entries $R_{l+1} (\mathbfe{\alpha})$ are independent of the value of $l$. \\

\noindent Similar statements hold for the sums over $\mathbf{g} , \mathbf{h}$, but we should keep in mind that there is no negative sign before $\mathbf{e}^T$ in these sums. Hence, we have
\begin{align*}
&\Bigg( \sum_{\substack{0 \leq l,m \leq n \\ l+m=n}}
	\sum_{\substack{\mathbf{e} \in {\mathbb{F}_q}^l \times \{ 1 \} \\ \mathbf{f} \in {\mathbb{F}_q}^m \times \{ 1 \} }}
	\psi \Big( - \mathbf{e}^T H_{l+1 , m+1} (\mathbfe{\alpha}) \mathbf{f} \Big) \Bigg) \\
	&\hspace{6em} \times \Bigg( \sum_{\substack{0 \leq l' , m' \leq n \\ l' + m' = n}}
	\sum_{\substack{\mathbf{g} \in {\mathbb{F}_q}^{l'} \times \{ 1 \} \\ \mathbf{h} \in {\mathbb{F}_q}^{m'} \times \{ 1 \} }}
	\psi \Big( \mathbf{g}^T H_{l' + 1 , m' + 1} (\mathbfe{\alpha}) \mathbf{h} \Big) \Bigg) \\
= &\Bigg( \sum_{\substack{ l+m=n \\ r \leq m \leq n-r}}
	\sum_{\substack{\mathbf{e} \in {\mathbb{F}_q}^l \times \{ 1 \} \\ \mathbf{f}^- \in \kernel H_{l+1,m+1} (\mathbfe{\alpha}) }}
	\psi \Big( - R_{l+1} (\mathbfe{\alpha}) \mathbf{a}_m \Big) \Bigg) \\
	&\hspace{6em} \times \Bigg( \sum_{\substack{l' + m' = n \\ r \leq m' \leq n-r}}
	\sum_{\substack{\mathbf{g} \in {\mathbb{F}_q}^{l'} \times \{ 1 \} \\ \mathbf{h}^- \in \kernel H_{l' +1,m' +1} (\mathbfe{\alpha}) }}
	\psi \Big( R_{l+1} (\mathbfe{\alpha}) \mathbf{a}_m \Big) \Bigg) \\
= &\Bigg( \sum_{\substack{ l+m=n \\ r \leq m \leq n-r}}
	\sum_{\substack{\mathbf{e} \in {\mathbb{F}_q}^l \times \{ 1 \} \\ \mathbf{f}^- \in \kernel H_{l+1,m+1} (\mathbfe{\alpha}) }} 1 \Bigg) 
	\Bigg( \sum_{\substack{l' + m' = n \\ r \leq m' \leq n-r}}
	\sum_{\substack{\mathbf{g} \in {\mathbb{F}_q}^{l'} \times \{ 1 \} \\ \mathbf{h}^- \in \kernel H_{l' +1,m' +1} (\mathbfe{\alpha}) }} 1 \Bigg) \\
=& \big( (n+1-2r) q^{n-r} \big)^2 ,
\end{align*}
where we have used the fact that $\lvert {\mathbb{F}_q}^l \times \{ 1 \} \rvert = q^l$ and $\lvert \kernel H_{l+1,m+1} (\mathbfe{\alpha}) \rvert = q^{(m+1) - (r+1)}$. \\

\item There is one value of $\alpha_{n+1}$ such that $\mathbfe{\alpha} \in \mathscr{L}_{n}^h (r,r,0)$. In which case, all $\mathbf{f}$ in $\kernel H_{l,m+1} (\mathbfe{\alpha}^-)$ are also in $\kernel H_{l+1,m+1} (\mathbfe{\alpha})$, thus giving $R_{l+1} (\mathbfe{\alpha}) \mathbf{f} = 0$. A similar statement holds for the sums over $\mathbf{g} , \mathbf{h}$. Hence, we have
\begin{align*}
&\Bigg( \sum_{\substack{0 \leq l,m \leq n \\ l+m=n}}
	\sum_{\substack{\mathbf{e} \in {\mathbb{F}_q}^l \times \{ 1 \} \\ \mathbf{f} \in {\mathbb{F}_q}^m \times \{ 1 \} }}
	\psi \Big( - \mathbf{e}^T H_{l+1 , m+1} (\mathbfe{\alpha}) \mathbf{f} \Big) \Bigg) \\
	&\hspace{6em} \times \Bigg( \sum_{\substack{0 \leq l' , m' \leq n \\ l' + m' = n}}
	\sum_{\substack{\mathbf{g} \in {\mathbb{F}_q}^{l'} \times \{ 1 \} \\ \mathbf{h} \in {\mathbb{F}_q}^{m'} \times \{ 1 \} }}
	\psi \Big( \mathbf{g}^T H_{l' + 1 , m' + 1} (\mathbfe{\alpha}) \mathbf{h} \Big) \Bigg) \\
= &\Bigg( \sum_{\substack{ l+m=n \\ r \leq m \leq n-r}}
	\sum_{\substack{\mathbf{e} \in {\mathbb{F}_q}^l \times \{ 1 \} \\ \mathbf{f}^- \in \kernel H_{l+1,m+1} (\mathbfe{\alpha}) }} 1 \Bigg) 
	\Bigg( \sum_{\substack{l' + m' = n \\ r \leq m' \leq n-r}}
	\sum_{\substack{\mathbf{g} \in {\mathbb{F}_q}^{l'} \times \{ 1 \} \\ \mathbf{h}^- \in \kernel H_{l' +1,m' +1} (\mathbfe{\alpha}) }} 1 \Bigg) \\
=& \big( (n+1-2r) q^{n-r} \big)^2 .
\end{align*}
\end{itemize}

Thus, we have
\begin{align*}
&\sum_{\alpha_n \in \mathbb{F}_q }
	\Bigg( \sum_{\substack{0 \leq l,m \leq n \\ l+m=n}}
	\sum_{\substack{\mathbf{e} \in {\mathbb{F}_q}^l \times \{ 1 \} \\ \mathbf{f} \in {\mathbb{F}_q}^m \times \{ 1 \} }}
	\psi \Big( - \mathbf{e}^T H_{l+1 , m+1} (\mathbfe{\alpha}) \mathbf{f} \Big) \Bigg) \\
	&\hspace{6em} \times \Bigg( \sum_{\substack{0 \leq l' , m' \leq n \\ l' + m' = n}}
	\sum_{\substack{\mathbf{g} \in {\mathbb{F}_q}^{l'} \times \{ 1 \} \\ \mathbf{h} \in {\mathbb{F}_q}^{m'} \times \{ 1 \} }}
	\psi \Big( \mathbf{g}^T H_{l' + 1 , m' + 1} (\mathbfe{\alpha}) \mathbf{h} \Big) \Bigg) \\
= &\sum_{\alpha_n \in \mathbb{F}_q } \big( (n+1-2r) q^{n-r} \big)^2 \\
= &q^{2n - 2r +1} (n+1-2r)^2 . 
\end{align*}
\end{proof}

\section{Correlations} \label{section, correlations}

We begin this section by proving Theorem \ref{theorem, correlations d(A) d(A+B)}.

\begin{proof}[Proof of Theorem \ref{theorem, correlations d(A) d(A+B)}]
The proof is very similar to the proof of Theorem \ref{main theorem, variance divisor function intervals}. We have
\begin{align*}
&\frac{1}{q^{n+h}} \sum_{A \in \mathcal{M}_n } \sum_{B \in \mathcal{A}_{<h} } d (A) d (A+B) \\
= &\frac{1}{q^{n+h}} \sum_{A \in \mathcal{M}_n } \sum_{B \in \mathcal{A}_{<h} }
	\bigg( \sum_{\substack{l,m \geq 0 \\ l+m=n }} \sum_{\substack{E \in \mathcal{M}_l \\ F \in \mathcal{M}_m }} \mathbbm{1}_{EF=A} \bigg)
	\bigg( \sum_{\substack{l' , m' \geq 0 \\ l' + m' = n }} \sum_{\substack{G \in \mathcal{M}_{l'} \\ H \in \mathcal{M}_{m'} }} \mathbbm{1}_{GH=A+B} \bigg) \\
= &\frac{1}{q^{n+h}} \sum_{A \in \mathcal{A}_{\leq n} } \sum_{B \in \mathcal{A}_{<h} }
	\bigg( \sum_{\substack{l,m \geq 0 \\ l+m=n }} \sum_{\substack{E \in \mathcal{M}_l \\ F \in \mathcal{M}_m }} \prod_{i=0}^{n} \mathbbm{1}_{ \{ EF \}_i = \{ A \}_i } \bigg)
	\bigg( \sum_{\substack{l' , m' \geq 0 \\ l' + m' = n }} \sum_{\substack{G \in \mathcal{M}_{l'} \\ H \in \mathcal{M}_{m'} }} \prod_{i=0}^{n} \mathbbm{1}_{ \{ GH \}_i = \{ A \}_i + \{ B\}_i } \bigg) \\
= &\frac{1}{q^{3n+h+2}} \sum_{\mathbf{a} \in \mathbb{F}_q^{n+1} } \sum_{\mathbf{b} \in \mathbb{F}_q^{h} }
	\bigg( \sum_{\mathbfe{\alpha} \in \mathbb{F}_q^{n+1} }
		\sum_{\substack{l,m \geq 0 \\ l+m=n }} 
		\sum_{\substack{ \mathbf{e} \in \mathbb{F}_q^{l} \times \{ 1 \} \\ \mathbf{f} \in \mathbb{F}_q^{m} \times \{ 1 \} }}
	\psi \Big( \mathbf{e}^T H_{l+1 , m+1} (\mathbfe{\alpha}) \mathbf{f} - \mathbf{a} \cdot \mathbfe{\alpha} \Big) \bigg) \\
	&\hspace{6em} \times \bigg( \sum_{\mathbfe{\beta} \in \mathbb{F}_q^{n+1} }
		\sum_{\substack{l' , m' \geq 0 \\ l' + m' = n' }} 
		\sum_{\substack{ \mathbf{g} \in \mathbb{F}_q^{l'} \times \{ 1 \} \\ \mathbf{h} \in \mathbb{F}_q^{m'} \times \{ 1 \} }}
	\psi \Big( \mathbf{g}^T H_{l'+1 , m'+1} (\mathbfe{\beta}) \mathbf{h} - (\mathbf{a} + \mathbf{b}) \cdot \mathbfe{\beta} \Big) \bigg) .
\end{align*}

As previously, the sum over $\mathbf{a}$ will force $\mathbfe{\beta} = - \mathbfe{\alpha}$, while the sum over $\mathbf{b}$ will force the first $h$ entries of $\mathbfe{\alpha}$ (and $\mathbfe{\beta}$) to be zero. Thus, we have
\begin{align*}
&\frac{1}{q^{n+h}} \sum_{A \in \mathcal{M}_n } \sum_{B \in \mathcal{A}_{<h} } d (A) d (A+B) \\
= &\frac{1}{q^{2n+1}} \sum_{\mathbfe{\alpha} \in \mathscr{L}_{n}^h }
	\bigg( \sum_{\substack{l,m \geq 0 \\ l+m=n }} 
		\sum_{\substack{ \mathbf{e} \in \mathbb{F}_q^{l} \times \{ 1 \} \\ \mathbf{f} \in \mathbb{F}_q^{m} \times \{ 1 \} }}
	\psi \Big( \mathbf{e}^T H_{l+1 , m+1} (\mathbfe{\alpha}) \mathbf{f} \Big) \bigg) 
	\bigg( \sum_{\substack{l' , m' \geq 0 \\ l' + m' = n' }} 
		\sum_{\substack{ \mathbf{g} \in \mathbb{F}_q^{l'} \times \{ 1 \} \\ \mathbf{h} \in \mathbb{F}_q^{m'} \times \{ 1 \} }}
	\psi \Big( \mathbf{g}^T H_{l'+1 , m'+1} (- \mathbfe{\alpha}) \mathbf{h} \Big) \bigg) .
\end{align*}
This is identical to the third-to-last line of (\ref{statement, div var section, variance 2nd last step}) multiplied by $q^{-n-2h}$. In particular, (\ref{statement, div var section, variance 2nd last step}) and (\ref{statement, div var section, variance as Delta in terms of noncentred d(B)}) gives us
\begin{align*}
\frac{1}{q^{n+h}} \sum_{A \in \mathcal{M}_n } \sum_{B \in \mathcal{A}_{<h} } d (A) d (A+B)
= &\frac{1}{q^{2h+n}} \sum_{A \in \mathcal{M}_n} \bigg( \sum_{B \in I (A;h)} d(B) \bigg)^2 \\
= &(n+1)^2 + \frac{1}{q^{2h+n}} \sum_{A \in \mathcal{M}_n} \lvert \Delta (A;h) \rvert^2 ,
\end{align*}
and the proof is completed by an application of Theorem \ref{main theorem, variance divisor function intervals}.
\end{proof}

We will now prove Theorem \ref{theorem, divisor correlations d(KQ+N)d(N)}.

\begin{proof}[Proof of Theorem \ref{theorem, divisor correlations d(KQ+N)d(N)}]
As with previous proofs, we will use additive characters. We first prove that
\begin{align*}
&\frac{1}{q^{k+n}} \sum_{K \in \mathcal{M}_k } \sum_{N \in \mathcal{M}_n } d (KQ+N) d (N) \\
= &(\degree Q + k +1) (n+1) - q^{-\degree Q} (k - \degree Q -1) (n+1) .
\end{align*}
We have
\begin{align*}
&\frac{1}{q^{k+n}} \sum_{K \in \mathcal{M}_k } \sum_{N \in \mathcal{M}_n } d (KQ+N) d (N) \\
= &\frac{1}{q^{k+n}} \sum_{K \in \mathcal{M}_k } \sum_{N \in \mathcal{M}_n }
	\bigg( \sum_{\substack{l,m \geq 0 \\ l + m = k + \degree Q }} \sum_{\substack{E \in \mathcal{M}_{l} \\ F \in \mathcal{M}_{m} \\ EF = KQ+N }} 1 \bigg)
	\bigg( \sum_{\substack{l' , m' \geq 0 \\ l' + m' = n }} \sum_{\substack{G \in \mathcal{M}_{l'} \\ H \in \mathcal{M}_{m'} \\ GH = N }} 1 \bigg) \\
= &\frac{1}{q^{k+n}} \sum_{K \in \mathcal{A}_{\leq k} } \sum_{N \in \mathcal{A}_{\leq n} } 
	\bigg( \sum_{\substack{l,m \geq 0 \\ l + m = k + \degree Q }} \sum_{\substack{E \in \mathcal{M}_{l} \\ F \in \mathcal{M}_{m} \\ EF = KQ+N }} 
		\prod_{i=0}^{k+\degree Q} \mathbbm{1}_{ \{ EF \}_i = \{ KQ \}_i + \{ N \}_i } \bigg) \\
	&\hspace{6em}  \times \bigg( \sum_{\substack{l' , m' \geq 0 \\ l' + m' = n }} \sum_{\substack{G \in \mathcal{M}_{l'} \\ H \in \mathcal{M}_{m'} \\ GH = N }} 
		\prod_{i=0}^{n} \mathbbm{1}_{ \{ GH \}_i = \{ N \}_i } \bigg) \\
= &\frac{1}{q^{\degree Q + 2k + 2n + 2}} \sum_{\mathbf{k} \in \mathbb{F}_q^{k+1}} \sum_{\mathbf{n} \in \mathbb{F}_q^{n+1}} \\
	&\hspace{3em} \bigg( \sum_{\substack{l,m \geq 0 \\ l + m = k + \degree Q }}
		\sum_{\substack{\mathbf{e} \in \mathbb{F}_q^{l} \times \{ 1 \} \\ \mathbf{f} \in \mathbb{F}_q^{m} \times \{ 1 \} }}
		\sum_{\mathbfe{\alpha} \in \mathbb{F}_q^{k+\degree Q +1}}
		\psi \big( \mathbf{e}^T H_{l+1 , m+1} (\mathbfe{\alpha}) \mathbf{f}
			- \mathbf{k}^T H_{k+1 , \degree Q + 1} (\mathbfe{\alpha}) \mathbf{q}
			- \mathbf{n} \cdot \mathbfe{\alpha}' \big) \bigg) \\
	&\hspace{3em} \times \bigg( \sum_{\substack{l' , m' \geq 0 \\ l' + m' = n}}
		\sum_{\substack{\mathbf{g} \in \mathbb{F}_q^{l'} \times \{ 1 \}  \\ \mathbf{h} \in \mathbb{F}_q^{m'} \times \{ 1 \} }}
		\sum_{\mathbfe{\beta} \in \mathbb{F}_q^{n +1}}
		\psi \big( \mathbf{g}^T H_{l'+1 , m'+1} (\mathbfe{\beta}) \mathbf{h}
			- \mathbf{n} \cdot \mathbfe{\beta} \big) \bigg) ;
\end{align*}
where $\mathbfe{\alpha} = (\alpha_0 , \alpha_1 , \ldots , \alpha_{k + \degree Q})$ and we define $\mathbfe{\alpha}' := (\alpha_0 , \alpha_1 , \ldots , \alpha_{n})$, and we also write $Q = q_0 + q_1 T + \ldots + q_{\degree Q} T^{\degree Q}$ and define $\mathbf{q} := (q_0 , q_1 , \ldots , q_{\degree Q})$. Now, similar to what we have seen in previous proofs, the sum over $\mathbf{n}$ will force $\mathbfe{\beta}$ to equal $-\mathbfe{\alpha}'$, while the sum over $\mathbf{k}$ will require $\mathbfe{\alpha}$ to be such that $H_{k+1 , \degree Q + 1} (\mathbfe{\alpha}) \mathbf{q} = \mathbf{0}$. Thus, we have		
\begin{align}
\begin{split} \label{statement, divisor correlations d(KQ+N)d(N) theorem proof, express as prod of two add sums}
&\frac{1}{q^{k+n}} \sum_{K \in \mathcal{M}_k } \sum_{N \in \mathcal{M}_n } d (KQ+N) d (N) \\
= &\frac{1}{q^{\degree Q +k+n}} 
	\sum_{\substack{\mathbfe{\alpha} \in \mathbb{F}_q^{k+\degree Q +1} \\ H_{k+1 , \degree Q + 1} (\mathbfe{\alpha}) \mathbf{q} = \mathbf{0} }} 
	\bigg( \sum_{\substack{l,m \geq 0 \\ l + m = k + \degree Q}} \sum_{\substack{\mathbf{e} \in \mathbb{F}_q^{l} \times \{ 1 \}  \\ \mathbf{f} \in \mathbb{F}_q^{m} \times \{ 1 \} }}
		\psi \big( \mathbf{e}^T H_{l+1 , m+1} (\mathbfe{\alpha}) \mathbf{f} \big) \bigg) \\
	& \hspace{10em} \times \bigg( \sum_{\substack{l' , m' \geq 0 \\ l' + m' = n}} \sum_{\substack{\mathbf{g} \in \mathbb{F}_q^{l'} \times \{ 1 \}  \\ \mathbf{h} \in \mathbb{F}_q^{m'} \times \{ 1 \} }}
		\psi \big( \mathbf{g}^T H_{l'+1 , m'+1} (- \mathbfe{\alpha}') \mathbf{h} \big) \bigg) .
\end{split}
\end{align}

First, we consider the sum
\begin{align*}
\sum_{\substack{l , m \geq 0 \\ l + m = \degree Q +k }} 
	\sum_{\substack{\mathbf{e} \in \mathbb{F}_q^{l} \times \{ 1 \}  \\ \mathbf{f} \in \mathbb{F}_q^{m} \times \{ 1 \} }} 
	\psi \big( \mathbf{e}^T H_{l+1 , m+1} (\mathbfe{\alpha}) \mathbf{f} \big).
\end{align*}
Let $\mathbfe{\alpha} \in \mathscr{L}_n (r , \rho_1 , \pi_1 )$, and let us investigate the values of $r , \rho_1 , \pi_1$. The theorem assumes that $k \geq \degree Q -1$, and so, by Theorem \ref{theorem, all Hankel matrices have characteristic polynomial kernel}, there exists some $A_1 \in \mathcal{M}_{\rho_1 }$ such that $\kernel H_{k+1 , \degree Q + 1} (\mathbfe{\alpha}) = \{ B_1 A_1 : B_1 \in \mathcal{A} , \degree B_1 \leq \degree Q - r \}$. Since $Q$ is in this kernel, and since it is prime, we must have that either $B_1$ can take the value $Q$, or $A_1 = Q$. That is, either $r=0$ and so $\rho_1 , \pi_1 = 0$, or $r = \degree Q$ and so $\rho_1 = \degree Q$ and $\pi_1 = 0$. The former simply means $\mathbfe{\alpha} = \mathbf{0}$. We will consider each case separately. \\

For $\mathbfe{\alpha} = \mathbf{0}$, we have
\begin{align*}
&\sum_{\substack{l , m \geq 0 \\ l + m = \degree Q +k }}
	\sum_{\substack{\mathbf{e} \in \mathbb{F}_q^{l} \times \{ 1 \}  \\ \mathbf{f} \in \mathbb{F}_q^{m} \times \{ 1 \} }} 
	\psi \big( \mathbf{e}^T H_{l+1 , m+1} (\mathbf{0}) \mathbf{f} \big) 
= \sum_{\substack{l , m \geq 0 \\ l + m = \degree Q +k }}
	\sum_{\substack{\mathbf{e} \in \mathbb{F}_q^{l} \times \{ 1 \}  \\ \mathbf{f} \in \mathbb{F}_q^{m} \times \{ 1 \} }} 
	\psi \big( 0 \big) 
= q^{\degree Q +k} \sum_{\substack{l , m \geq 0 \\ l + m = \degree Q +k }} 1 \\
= &q^{\degree Q +k} (\degree Q + k +1) .
\end{align*}

For $\mathbfe{\alpha} \in \mathscr{L}_n (\degree Q , \degree Q , 0 )$, we have
\begin{align*}
&\sum_{\substack{l , m \geq 0 \\ l + m = \degree Q +k }}
	\sum_{\substack{\mathbf{e} \in \mathbb{F}_q^{l} \times \{ 1 \}  \\ \mathbf{f} \in \mathbb{F}_q^{m} \times \{ 1 \} }} 
	\psi \big( \mathbf{e}^T H_{l+1 , m+1} (\mathbfe{\alpha}) \mathbf{f} \big) 
= \sum_{\substack{l + m = \degree Q +k \\ \degree Q \leq m \leq k }}
	\sum_{\substack{\mathbf{e} \in \mathbb{F}_q^{l} \times \{ 1 \}  \\ \mathbf{f} \in \mathbb{F}_q^{m} \times \{ 1 \} }} 
	\psi \big( \mathbf{e}^T H_{l+1 , m+1} (\mathbfe{\alpha}) \mathbf{f} \big) \\
= &\sum_{\substack{l + m = \degree Q +k \\ \degree Q \leq m \leq k }} q^l
	\sum_{\substack{\mathbf{f} \in \mathbb{F}_q^{m} \times \{ 1 \} \\ \mathbf{f} \in \kernel H_{l , m+1} (\mathbfe{\alpha}) }} 1 
= \sum_{\substack{l + m = \degree Q +k \\ \degree Q \leq m \leq k }} q^l q^{m - \degree Q} 
= q^k (k - \degree Q +1) ,
\end{align*}
where the first equality uses (\ref{statement, variance section, evaluation of exp sums of vectors and matrices lemma proof, claim 2, equals 0 cases}). \\

We apply these two results to (\ref{statement, divisor correlations d(KQ+N)d(N) theorem proof, express as prod of two add sums}). To do so, we define
\begin{align*}
S
:= \bigg\{ \mathbfe{\alpha} \in \mathbb{F}_{q}^{\degree Q + k +1} :
	\substack{ \alpha_i = 0 \text{ for all } i \in \{ 0 , \ldots , n -1 , n+1 , \ldots , \degree Q -1 \} \\
		H_{k+1 , \degree Q + 1} (\mathbfe{\alpha}) \mathbf{q} = \mathbf{0} } \bigg\} .
\end{align*}
For $\mathbfe{\alpha} \in S$, it may be helpful to note that, due to the condition $H_{k+1 , \degree Q + 1} (\mathbfe{\alpha}) \mathbf{q} = \mathbf{0}$, the terms $\alpha_{\degree Q} , \ldots , \alpha_{\degree Q + k}$ can be expressed entirely in terms of $\alpha_{0} , \ldots , \alpha_{\degree Q -1}$ of which only $\alpha_{n}$ could be non-zero. Further, if $\alpha_n = 0$, then $\mathbfe{\alpha} = \mathbf{0}$; while if $\alpha_{n} \neq 0$, then $\mathbfe{\alpha} \in \mathscr{L}_n (\degree Q , \degree Q , 0 )$. We also note that for $\mathbfe{\alpha} \not\in S$ satisfying $H_{k+1 , \degree Q + 1} (\mathbfe{\alpha}) \mathbf{q} = \mathbf{0}$, we have $\mathbfe{\alpha} \in \mathscr{L}_n (\degree Q , \degree Q , 0 )$. \\

Now, we consider the cases $\mathbfe{\alpha} = \mathbf{0}$, $\mathbfe{\alpha} \in S \backslash \{ \mathbf{0} \}$, and $\mathbfe{\alpha} \not\in S$ separately. We have
\begin{align}
\begin{split} \label{statement, divisor correlations d(KQ+N)d(N) theorem proof, express in terms of S}
&\frac{1}{q^{k+n}} \sum_{K \in \mathcal{M}_k } \sum_{N \in \mathcal{M}_n } d (KQ+N) d (N) \\
= &q^{-n} (\degree Q + k +1) \sum_{\substack{l' , m' \geq 0 \\ l' + m' = n}} \sum_{\substack{\mathbf{g} \in \mathbb{F}_q^{l'} \times \{ 1 \}  \\ \mathbf{h} \in \mathbb{F}_q^{m'} \times \{ 1 \} }}
		\psi \big( \mathbf{g}^T H_{l'+1 , m'+1} (\mathbf{0} ) \mathbf{h} \big) \\
&+ q^{-\degree Q - n} (k - \degree Q +1)
	\sum_{\substack{\mathbfe{\alpha} \in S \\ \mathbfe{\alpha} \backslash \{ \mathbf{0} \} }} 
	\bigg( \sum_{\substack{l' , m' \geq 0 \\ l' + m' = n}} \sum_{\substack{\mathbf{g} \in \mathbb{F}_q^{l'} \times \{ 1 \}  \\ \mathbf{h} \in \mathbb{F}_q^{m'} \times \{ 1 \} }}
		\psi \big( \mathbf{g}^T H_{l'+1 , m'+1} (- \mathbfe{\alpha}') \mathbf{h} \big) \bigg) \\
&+ q^{-\degree Q - n} (k - \degree Q +1)
	\sum_{\substack{\mathbfe{\alpha} \in \mathbb{F}_q^{k+\degree Q +1} \\ H_{k+1 , \degree Q + 1} (\mathbfe{\alpha}) \mathbf{q} = \mathbf{0} \\ \mathbfe{\alpha} \not\in S }} 
	\bigg( \sum_{\substack{l' , m' \geq 0 \\ l' + m' = n}} \sum_{\substack{\mathbf{g} \in \mathbb{F}_q^{l'} \times \{ 1 \}  \\ \mathbf{h} \in \mathbb{F}_q^{m'} \times \{ 1 \} }}
		\psi \big( \mathbf{g}^T H_{l'+1 , m'+1} (- \mathbfe{\alpha}') \mathbf{h} \big) \bigg) .
\end{split}
\end{align}

Consider the case $\mathbfe{\alpha} \not\in S$ first. Let $\mathbfe{\alpha}'' = ( \alpha_0 , \alpha_1 , \ldots , \alpha_{n-1})$. By Claim 1 in Lemma \ref{lemma, variance section, evaluation of exp sums of vectors and matrices}, the non-zero contributions occur when $\mathbfe{\alpha}'' \in \mathscr{L}_{n-1} (r,r,0)$ for some $0 \leq r \leq n_1 -1$. By (\ref{statement, variance section, evaluation of exp sums of vectors and matrices lemma proof, claim 2, equals 0 cases}), we need only consider when $r \leq m' \leq n-r$. \\

Note that if $r \neq 0$, then $\alpha_{n+1} , \ldots , \alpha_{\degree Q -1}$ do not affect our sum, and so they are free to take any values in $\mathbb{F}_q$ (of which there are $q^{\degree Q -n-1}$ possibilities); while if $r=0$, they can take any value but they cannot all be $0$ simultaneously (of which there are $q^{\degree Q -n-1} -1$ possibilities), otherwise $\mathbfe{\alpha} \in S$. We define
\begin{align*}
c_{\mathbfe{\alpha}'' } 
= \begin{cases}
q^{\degree Q -n-1} &\text{ if $\mathbfe{\alpha} \in \mathscr{r,r,0}$ with $r \neq 0$,} \\
q^{\degree Q -n-1} -1 &\text{ if $\mathbfe{\alpha} \in \mathscr{r,r,0}$ with $r=0$.}
\end{cases}
\end{align*}
So, we have
\begin{align*}
&\sum_{\substack{\mathbfe{\alpha} \in \mathbb{F}_q^{k+\degree Q +1} \\ H_{k+1 , \degree Q + 1} (\mathbfe{\alpha}) \mathbf{q} = \mathbf{0} \\ \mathbfe{\alpha} \not\in S }} 
	\sum_{\substack{l' , m' \geq 0 \\ l' + m' = n}} \sum_{\substack{\mathbf{g} \in \mathbb{F}_q^{l'} \times \{ 1 \}  \\ \mathbf{h} \in \mathbb{F}_q^{m'} \times \{ 1 \} }}
		\psi \big( \mathbf{g}^T H_{l'+1 , m'+1} (- \mathbfe{\alpha}') \mathbf{h} \big) \\
= &\sum_{r=0}^{n_1 -1} 
	\sum_{\substack{\mathbfe{\alpha}'' \in \mathscr{L}_{n-1} (r,r,0) \\ \alpha_{n} \in \mathbb{F}_q }} 
		c_{\mathbfe{\alpha}'' }
	\sum_{\substack{l' + m' = n \\ r \leq m' \leq n-r}} \sum_{\substack{\mathbf{g} \in \mathbb{F}_q^{l'} \times \{ 1 \}  \\ \mathbf{h} \in \mathbb{F}_q^{m'} \times \{ 1 \} }}
		\psi \big( \mathbf{g}^T H_{l'+1 , m'+1} (\mathbfe{\alpha}') \mathbf{h} \big) \\
= &\sum_{r=0}^{n_1' -1} 
	\sum_{\mathbfe{\alpha}'' \in \mathscr{L}_{n-1} (r,r,0)} 
		c_{\mathbfe{\alpha}'' } 
	\sum_{\substack{l' + m' = n \\ r \leq m' \leq n-r}} q^{l'}
	\sum_{\substack{\mathbf{h} \in \mathbb{F}_q^{m'} \times \{ 1 \} \\ H_{l' , m'+1} (- \mathbfe{\alpha}') \mathbf{h} = \mathbf{0} }}
	\sum_{\alpha_{n} \in \mathbb{F}_q} \psi \Big( R_{l' +1} (\mathbfe{\alpha}') \cdot \mathbf{h} \Big) \\
= &0 ,
\end{align*}
where $R_{l' +1} (\mathbfe{\alpha}')$ is the $(l'+1)$-th row of $H_{l' , m'+1} (- \mathbfe{\alpha}')$, and we have used the fact that $\sum_{\alpha_{n} \in \mathbb{F}_q} \psi \Big( R_{l' +1} (\mathbfe{\alpha}') \cdot \mathbf{h} \Big) = \sum_{\alpha_{n} \in \mathbb{F}_q} \psi ( \alpha_n ) = 0$. \\

Let us now consider the case $\mathbfe{\alpha} \in S \backslash \{ 0 \}$ in (\ref{statement, divisor correlations d(KQ+N)d(N) theorem proof, express in terms of S}). By a similar argument as above, but using $\sum_{\alpha_{n} \in \mathbb{F}_q \backslash \{ 0 \}} \psi ( \alpha_n ) = -1$ instead, we obtain
\begin{align*}
&\sum_{\substack{\mathbfe{\alpha} \in S \\ \mathbfe{\alpha} \backslash \{ \mathbf{0} \} }} 
	\bigg( \sum_{\substack{l' , m' \geq 0 \\ l' + m' = n}} \sum_{\substack{\mathbf{g} \in \mathbb{F}_q^{l'} \times \{ 1 \}  \\ \mathbf{h} \in \mathbb{F}_q^{m'} \times \{ 1 \} }}
		\psi \big( \mathbf{g}^T H_{l'+1 , m'+1} (- \mathbfe{\alpha}') \mathbf{h} \big) \bigg) \\
= &\sum_{\substack{l' , m' \geq 0 \\ l' + m' = n}} \sum_{\substack{\mathbf{g} \in \mathbb{F}_q^{l'} \times \{ 1 \}  \\ \mathbf{h} \in \mathbb{F}_q^{m'} \times \{ 1 \} }}
	\sum_{\alpha_{n} \in \mathbb{F}_q \backslash \{ 0 \}} \psi \big( \alpha_n \big) \\
= &- q^n (n+1) .
\end{align*}

Finally, it is not difficult to see that
\begin{align*}
&\sum_{\substack{l' , m' \geq 0 \\ l' + m' = n}} \sum_{\substack{\mathbf{g} \in \mathbb{F}_q^{l'} \times \{ 1 \}  \\ \mathbf{h} \in \mathbb{F}_q^{m'} \times \{ 1 \} }}
		\psi \big( \mathbf{g}^T H_{l'+1 , m'+1} (\mathbf{0} ) \mathbf{h} \big) \\
= &q^n (n+1) .
\end{align*}

Apply these three results to (\ref{statement, divisor correlations d(KQ+N)d(N) theorem proof, express in terms of S}) gives
\begin{align*}
&\frac{1}{q^{k+n}} \sum_{K \in \mathcal{M}_k } \sum_{N \in \mathcal{M}_n } d (KQ+N) d (N) \\
= &(\degree Q + k +1) (n+1) - q^{-\degree Q} (k - \degree Q + 1) (n+1) . \\
\end{align*}

We now prove that
\begin{align*}
&\bigg( \frac{1}{q^n } \sum_{N \in \mathcal{M}_n} d(N) \bigg)
	\bigg( \frac{1}{q^{k + n}} \sum_{K \in \mathcal{M}_k } \sum_{N \in \mathcal{M}_n } d (KQ+N) \bigg) \\
= &(\degree Q + k +1) (n+1) - q^{-\degree Q} (k - \degree Q -1) (n+1) .
\end{align*}
We have
\begin{align*}
\frac{1}{q^n } \sum_{N \in \mathcal{M}_n} d(N) 
= &\frac{1}{q^{2n+1} } \sum_{\mathbf{n} \in \mathbb{F}_q^{n+1}}
	\sum_{\substack{l' , m' \geq 0 \\ l' + m' = n}}
	\sum_{\substack{\mathbf{g} \in \mathbb{F}_q^{l'} \times \{ 1 \}  \\ \mathbf{h} \in \mathbb{F}_q^{m'} \times \{ 1 \} }}
	\sum_{\mathbfe{\alpha} \in \mathbb{F}_q^{n+1} }
	\psi \big( \mathbf{g}^T H_{l'+1 , m'+1} (\mathbfe{\alpha} ) \mathbf{h} - \mathbf{n} \cdot \mathbfe{\alpha} \big) \\
= &\frac{1}{q^{n} } \sum_{\substack{l' , m' \geq 0 \\ l' + m' = n}}
	\sum_{\substack{\mathbf{g} \in \mathbb{F}_q^{l'} \times \{ 1 \}  \\ \mathbf{h} \in \mathbb{F}_q^{m'} \times \{ 1 \} }}
	\psi \big( \mathbf{g}^T H_{l'+1 , m'+1} (\mathbf{0} ) \mathbf{h} \big) \\
= & n+1 ,
\end{align*}
where, for the second equality, similar to what we have seen previously, the sum over $\mathbf{n}$ forces $\mathbfe{\alpha} = \mathbf{0}$. \\

We also have
\begin{align*}
&\frac{1}{q^{k + n}} \sum_{K \in \mathcal{M}_k } \sum_{N \in \mathcal{M}_n } d (KQ+N) \\
= &\frac{1}{q^{\degree Q + 2k + n + 1}} \sum_{\mathbf{k} \in \mathbb{F}_q^{k+1}} \sum_{\mathbf{n} \in \mathbb{F}_q^{n} \times \{ 1 \} }
	\sum_{\substack{l,m \geq 0 \\ l + m = k + \degree Q }} 
	\sum_{\substack{\mathbf{e} \in \mathbb{F}_q^{l} \times \{ 1 \} \\ \mathbf{f} \in \mathbb{F}_q^{m} \times \{ 1 \} }}
	\sum_{\mathbfe{\alpha} \in \mathbb{F}_q^{k+\degree Q +1}} \\
	&\hspace{6em} \psi \big( \mathbf{e}^T H_{l+1 , m+1} (\mathbfe{\alpha}) \mathbf{f}
			- \mathbf{k}^T H_{k+1 , \degree Q + 1} (\mathbfe{\alpha}) \mathbf{q}
			- \mathbf{n} \cdot \mathbfe{\alpha}' \big) \\
= &\frac{1}{q^{\degree Q + k }} \sum_{\substack{l,m \geq 0 \\ l + m = k + \degree Q }} 
	\sum_{\substack{\mathbfe{\alpha} \in \{ 0 \}^{n} \times \mathbb{F}_q^{k+\degree Q - n +1} \\ H_{k+1 , \degree Q + 1} (\mathbfe{\alpha}) \mathbf{q} = \mathbf{0} }}
	\sum_{\substack{\mathbf{e} \in \mathbb{F}_q^{l} \times \{ 1 \} \\ \mathbf{f} \in \mathbb{F}_q^{m} \times \{ 1 \} }}
		\psi \big( \mathbf{e}^T H_{l+1 , m+1} (\mathbfe{\alpha}) \mathbf{f} - \alpha_n \big) .
\end{align*}
Again, for the last equality, the sum of $\mathbf{n}$ over $\mathbb{F}_q^{n} \times \{ 1 \}$ forces $\alpha_0 , \alpha_1 . \ldots , \alpha_{n-1} = 0$, while the sum over $\mathbf{k}$ forces the requirement that $H_{k+1 , \degree Q + 1} (\mathbfe{\alpha}) \mathbf{q} = \mathbf{0}$. As previously, the contribution of $\mathbfe{\alpha} \in S$ is zero. Thus, we have
\begin{align*}
&\frac{1}{q^{k + n}} \sum_{K \in \mathcal{M}_k } \sum_{N \in \mathcal{M}_n } d (KQ+N) \\
= &\frac{1}{q^{\degree Q + k }} \sum_{\substack{l,m \geq 0 \\ l + m = k + \degree Q }} 
	\sum_{\substack{\mathbf{e} \in \mathbb{F}_q^{l} \times \{ 1 \} \\ \mathbf{f} \in \mathbb{F}_q^{m} \times \{ 1 \} }}
		\psi \big( \mathbf{e}^T H_{l+1 , m+1} (\mathbf{0}) \mathbf{f} \big) \\
&+ \frac{1}{q^{\degree Q + k }} \sum_{\substack{l,m \geq 0 \\ l + m = k + \degree Q }} 
	\sum_{\substack{\mathbfe{\alpha} \in S \backslash \{ 0 \} }}
	\sum_{\substack{\mathbf{e} \in \mathbb{F}_q^{l} \times \{ 1 \} \\ \mathbf{f} \in \mathbb{F}_q^{m} \times \{ 1 \} }}
		\psi \big( \mathbf{e}^T H_{l+1 , m+1} (\mathbfe{\alpha}) \mathbf{f} - \alpha_n \big) \\
= &(\degree Q + k +1)
	+ \frac{1}{q^{\degree Q + k }} \sum_{\substack{l + m = k + \degree Q \\ \degree Q \leq m \leq k }} 
	\sum_{\substack{\mathbf{e} \in \mathbb{F}_q^{l} \times \{ 1 \} \\ \mathbf{f} \in \mathbb{F}_q^{m} \times \{ 1 \} }}
		\psi \big( \mathbf{e}^T H_{l+1 , m+1} (\mathbfe{\alpha}) \mathbf{f} \big) 
	\sum_{\alpha_n \in \mathbb{F}_q^* } \alpha_n \\
= &(\degree Q + k +1)
	+ \frac{1}{q^{\degree Q + k }} \bigg( \sum_{\substack{l + m = k + \degree Q \\ \degree Q \leq m \leq k }} 
	q^{l+m-\degree Q} \bigg)
	\bigg( \sum_{\alpha_n \in \mathbb{F}_q^* } \alpha_n \bigg) \\
= &(\degree Q + k +1)  - \frac{1}{q^{\degree Q }} (k - \degree Q +1) . 
\end{align*}
\end{proof}

\begin{remark} \label{remark, extending FF Dir L func Corr to k < deg Q -1}
We can see from the proof of Theorem \ref{theorem, divisor correlations d(KQ+N)d(N)} that in evaluating the sum
\begin{align*}
\frac{1}{q^{k+n}} \sum_{K \in \mathcal{M}_k } \sum_{N \in \mathcal{M}_n } d (KQ+N) d (N) ,
\end{align*}
the sequences denoted by $\mathbfe{\alpha}$ address the polynomial $KQ+N$, while their truncations $\mathbfe{\alpha}'$ address the polynomial $N$. However, because of the range we have for $k$ (the degree of $K$), the value of $\mathbfe{\alpha}'$ does not affect the $(\rho , \pi )$-form of $\mathbfe{\alpha}$ (except the special case where $\mathbfe{\alpha} = \mathbf{0}$). This is why $d (KQ+N)$ and $d (N)$ are uncorrelated for the given ranges of $k$ and $n$. \\

If, instead, we took a smaller value of $k$, which is what we would find in fourth moment calculations of Dirichlet $L$-functions, then the $(\rho , \pi)$-form of $\mathbfe{\alpha}$ becomes dependent on the value of $\mathbfe{\alpha}'$, thus making it more difficult to evaluate the sum. In effect, for given $r , \rho_1 , \pi_1$ and $r' , \rho_1 ', \pi_1 '$ we must determine how many $\mathbfe{\alpha}$ there are such that \\

\begin{enumerate}
\item $\mathbfe{\alpha} \in \mathscr{L}_n (r , \rho_1 , \pi_1 )$, \\ \label{mom corr reqs rho pi of alpha}

\item $H_{k+1 , \degree Q + 1} (\mathbfe{\alpha}) \mathbf{q} = \mathbf{0}$, \\ \label{mom corr reqs kernel}

\item $\mathbfe{\alpha}' \in \mathscr{L}_n (r' , \rho_1 ' , \pi_1 ' )$. \\ \label{mom corr reqs rho pi of alpha'}
\end{enumerate}
In fact, by Claim 1 of Lemma \ref{lemma, variance section, evaluation of exp sums of vectors and matrices}, we need only consider the cases where $\pi_1 , \pi_1 ' \in \{ 0 , 1 \}$. We can reformulate the three conditions above in terms of coprime polynomials $A_1 , A_2$. Indeed, by Theorem \ref{theorem, all Hankel matrices have characteristic polynomial kernel}, condition \ref{mom corr reqs rho pi of alpha} is equivalent to certain degree restrictions on the characteristic polynomials $A_1 , A_2$; condition \ref{mom corr reqs kernel} is equivalent to $Q$ being a certain linear combination of $A_1 , A_2$; and by Corollary \ref{corollary, Hankel matrices incorporate Euclidean algorithm theorem, full algorithm presented}, condition \ref{mom corr reqs rho pi of alpha'} is equivalent to certain degree restrictions on the polynomials we obtain by applying the Euclidean algorithm to $A_1 , A_2$. It is not difficult to satisfy any two of the three conditions, but satisfying all three is more difficult. 
\end{remark}

\textbf{Acknowledgements:} This research was conducted during a postdoctoral fellowship funded by the Leverhulme Trust research project grant ``Moments of $L$-functions in Function Fields and Random Matrix Theory'' (grant number RPG-2017-320) secured by Julio Andrade. The author is most grateful for this support. \\

The author is very grateful to Ze\'{e}v Rudnick for pointing out that the function $a \mapsto \exp \Big( \frac{2 \pi i}{p} a \Big)$ on what is essentially $\mathbb{F}_p$ (where $p$ is prime) can be replaced by a non-trivial additive character on $\mathbb{F}_q$ (indeed, the former is a special case of the latter), thus allowing an immediate generalisation of our results on divisor sums from $\mathbb{F}_p [T]$ to $\mathbb{F}_q [T]$.

\bibliography{YiasemidesBibliography1}{}
\bibliographystyle{YiaseBstNumer1}

\end{document}